\documentclass[12pt, a4paper]{amsart}

\usepackage{amssymb, amsmath}
\usepackage[all]{xy}
\usepackage[inline]{enumitem}
\usepackage{hyphenat}
\usepackage[colorlinks,linkcolor=blue,citecolor=blue,urlcolor=teal]{hyperref}
\usepackage{xcolor}
\usepackage{mathtools}
\usepackage{stackrel}
\usepackage{tikz-cd}
\usepackage[labelformat=empty]{caption}
 \usepackage[ left=3cm, right=3cm, top = 2.8cm, bottom = 2.8cm ]{geometry}

\makeatletter
\def\@tocline#1#2#3#4#5#6#7{\relax
  \ifnum #1>\c@tocdepth 
  \else
    \par \addpenalty\@secpenalty\addvspace{#2}%
    \begingroup \hyphenpenalty\@M
    \@ifempty{#4}{%
      \@tempdima\csname r@tocindent\number#1\endcsname\relax
    }{%
      \@tempdima#4\relax
    }%
    \parindent\z@ \leftskip#3\relax \advance\leftskip\@tempdima\relax
    \rightskip\@pnumwidth plus4em \parfillskip-\@pnumwidth
    #5\leavevmode\hskip-\@tempdima
      \ifcase #1
      \or\or \hskip 2em \or \hskip 2em \else \hskip 3em \fi%
      #6\nobreak\relax
    \dotfill\hbox to\@pnumwidth{\@tocpagenum{#7}}\par
    \nobreak
    \endgroup
  \fi}
\makeatother

\newcommand{\eq}[2]{\begin{equation}\label{#1}#2 \end{equation}}

\newcommand{\mlnl}[1]{\begin{multline*}#1 \end{multline*}}

\newcommand{\arir}{\ar@{^{(}->}}
\newcommand{\aril}{\ar@{_{(}->}}
\newcommand{\are}{\ar@{>>}}

\newcommand{\xr}[1] {\xrightarrow{#1}}

\newcommand{\lra}{\longrightarrow}

\newtheorem{lem}{Lemma}[section]

\newtheorem{thm}[lem]{Theorem}
\newtheorem{theorem}{Theorem}
\newtheorem{prop}[lem]{Proposition}

\newtheorem{cor}[lem]{Corollary}

\newtheorem{cor-intro}{Corollary}

\theoremstyle{definition}
\newtheorem{defn}[lem]{Definition}
\newtheorem{defn-prop}[lem]{Definition-Proposition}

\newtheorem{para}[lem]{}
\newtheorem*{para*}{}

\newtheorem*{acknowledgement}{Acknowledgement}

\theoremstyle{remark}

\newtheorem{remark}[lem]{Remark}
\newtheorem{rmk}[lem]{Remark}

\newtheorem{ex}[lem]{Example}
\newtheorem{exs}[lem]{Examples}
\newtheorem{exs-rmks}[lem]{Examples and Remarks}

\newtheorem{claim}{Claim}[lem]
\newtheorem*{claim*}{Claim}

\newcounter{zaehler} 
\numberwithin{equation}{lem}

\newcommand{\shuji}[1]{\begin{color}{green}{#1}\end{color}}

\newcommand{\kay}[1]{\begin{color}{teal}{#1}\end{color}}

\newcommand{\N}{\mathbb{N}}

\newcommand{\Q}{\mathbb{Q}}
\newcommand{\Z}{\mathbb{Z}}

\renewcommand{\P}{\mathbf{P}}
\newcommand{\A}{\mathbf{A}}

\newcommand{\G}{\mathbf{G}}

\newcommand{\sE}{\mathcal{E}}

\newcommand{\sG}{\mathcal{G}}
\newcommand{\sH}{\mathcal{H}}

\newcommand{\sO}{\mathcal{O}}

\newcommand{\fm}{\mathfrak{m}}

\newcommand{\Xb}{{\overline{X}}}

\newcommand{\Db}{{\overline{D}}}

\newcommand{\tF}{{\widetilde{F}}}

\newcommand{\ux}{{\underline{x}}}
\newcommand{\uy}{{\underline{y}}}
\newcommand{\uz}{{\underline{z}}}

\newcommand{\Cor}{\operatorname{\mathbf{Cor}}}

\newcommand{\RSC}{{\operatorname{\mathbf{RSC}}}}

\newcommand{\uMDM}{\operatorname{\mathbf{\underline{M}DM}}}

\newcommand{\Ext}{\operatorname{Ext}}
\newcommand{\ul}[1]{{\underline{#1}}}

\newcommand{\PST}{{\operatorname{\mathbf{PST}}}}

\newcommand{\NST}{\operatorname{\mathbf{NST}}}

\newcommand{\Hom}{\operatorname{Hom}}
\newcommand{\uHom}{\operatorname{\underline{Hom}}}

\newcommand{\Ker}{\operatorname{Ker}}

\renewcommand{\Im}{\operatorname{Im}}

\newcommand{\Tr}{\operatorname{Tr}}
\newcommand{\Nm}{\operatorname{Nm}}
\newcommand{\Div}{\operatorname{Div}}

\newcommand{\Spec}{\operatorname{Spec}}

\newcommand{\Proj}{\operatorname{Proj}}
\newcommand{\Sm}{\operatorname{\mathbf{Sm}}}

\newcommand{\Ab}{\operatorname{\mathbf{Ab}}}

\newcommand{\tr}{{\operatorname{tr}}}
\newcommand{\Ztr}{{\operatorname{\mathbb{Z}_{\tr}}}}

\newcommand{\red}{{\operatorname{red}}}

\newcommand{\Zar}{{\operatorname{Zar}}}
\newcommand{\Nis}{{\operatorname{Nis}}}

\newcommand{\et}{{\operatorname{\acute{e}t}}}

\newcommand{\inj}{\hookrightarrow}

\newcommand{\Res}{\operatorname{Res}}

\newcommand{\id}{{\operatorname{id}}}

\newcommand{\ch}{{\operatorname{ch}}}
\newcommand{\Sym}{{\operatorname{Sym}}}

\newcommand{\CH}{{\operatorname{CH}}}
\newcommand{\Frac}{{\operatorname{Frac}}}

\newcommand{\fil}{{\operatorname{fil}}}

\newcommand{\e}{{\epsilon}}
\renewcommand{\b}{{\rm b}}
\renewcommand{\c}{{\rm c}}
\newcommand{\mc}{{\rm mc}}

\newcommand{\AS}{{\rm AS}}
\newcommand{\FAS}{{F^{\rm AS}}}

\newcommand{\colim}{\operatornamewithlimits{\varinjlim}}

\newcommand{\ol}{\overline}

\renewcommand{\epsilon}{\varepsilon}
\renewcommand{\div}{\operatorname{div}}

\newcommand{\bcube}{{\ol{\square}}}

\newcommand{\uMNST}{\operatorname{\mathbf{\underline{M}NST}}}

\def\rmapo#1{\overset{#1}{\longrightarrow}}

\def\PXR{P^{(R)}_X}

\def\Db{\overline{D}}

\def\Db{\overline{D}}

\def\Db{\overline{D}}

\def\ShSmB{\mathrm{Sh}(\SmB)}

\def\otShSmB{\otimes_{\ShSmB}}
\def\ShSm#1{\mathrm{Sh}(\Sm_{#1})}

\def\Fad{F_{ad}}

\def\qwith{\;\text{ with }\;}

\def\isom{\overset{\sim}{\longrightarrow}}

\def\charr{{\mathrm{char}}}

\def\charFR{{\mathrm{char}^{(R)}_F}}

\def\ShSmB{\mathrm{Sh}_B}

\def\otShSmB{\otimes_{\Sh_B}}
\def\ShSm#1{\mathrm{Sh}_{#1}}
\def\Sh{\mathrm{Sh}}

\def\ad{\mathrm{ad}}

\title[Abbes-Saito Formula]{Ramification theory for reciprocity sheaves, III, \\Abbes-Saito formula}
\author{Kay R\"ulling \and Shuji Saito}
 
\address{Bergische Universit\"at Wuppertal\\ Gau\ss str. 20, D-42119 Wuppertal, Germany}
\email{ruelling@uni-wuppertal.de}
 
\address{Graduate School of Mathematical Sciences, University of Tokyo, 3-8-1 Komaba, Tokyo 153-8941, Japan}
\email{sshuji@msb.biglobe.ne.jp}

\thanks{S.S.\ is supported by the JSPS KAKENHI Grant (20H01791). }
\begin{document}
\begin{abstract}
We give a new geometric characterization of the motivic ramification filtration of reciprocity sheaves, by imitating
a method used by Abbes and (Takeshi) Saito to study the ramification of torsors under finite \'etale groups.
This new characterization is used to define  characteristic forms for reciprocity sheaves.
We obtain applications on pseudo-rational singularities and on questions regarding the
 representability of certain cohomology groups of reciprocity sheaves
in the triangulated category of motives with modulus introduced by Kahn-Miyazaki-Saito-Yamazaki.

\end{abstract}
\maketitle

\tableofcontents

\section*{Introduction}

In the study of geometric or arithmetic  questions on non-proper smooth  - or on proper but  singular varieties,
it is often essential that the objects of interest come with a measure of their complexity at infinity or along the singularities, 
such as pole order or degree of ramification.
A  motivic framework which provides this measure is the
theory of reciprocity sheaves as introduced by Kahn, Saito, and Yamazaki in \cite{KSY2} (see also \cite{KSY1}).
Examples include smooth commutative group schemes, $\A^1$-invariant sheaves with transfers, differential forms, the sheaf of rank $1$-connections, the Brauer group, and more generally
\'etale motivic cohomology sheaves with torsion coefficients.
A reciprocity sheaf $F$ is a Nisnevich sheaf with transfers on $\Sm$ the category of smooth separated schemes over a fixed perfect base field $k$, 
with the additional property that  for any  pair $(X, R)$  consisting of a  proper $k$-scheme $X$ and an effective Cartier divisor $R$
on $X$ with smooth complement $U=X\backslash R$, there is  a canonical  subgroup
\[\tF(X,R)\subset F(U),\]
which heuristically can be thought of as those sections over $U$ whose ramification or pole order at infinity  is bounded by $R$.
We say a section $a\in F(U)$ has {\em modulus $(X,R)$} if $a\in \tF(X,R)$.
The definition of these subgroups is  motivic and relies on  the transfer structure of $F$, see  \ref{para:rec}.
This has the advantage that we have good  functorial properties, i.e., $(X,R)\mapsto \tF(X,R)$ defines a
Nisnevich sheaf on the category of modulus correspondences as introduced in \cite{KMSY1}, \cite{KMSY2} and therefore
defines a birational invariant in the sense
\[\tag{1}\label{intro1} \tF(X,R)=\tF(X',f^*R),\]
for any proper dominant morphism $f: X'\to X$ which induces an isomorhism $X\setminus R\cong X'\setminus f^*R$.
Furthermore, the definition of $\tF(X,R)$ can be extended to the case where $X$ is not necessarily proper over $k$.
Thus we obtain  a sheaf $\tF_{(X,R)}$ on $X_{\Nis}$ defined by
\[V\mapsto \tF(V, R_{|V}).\]
If $X$ is projective and smooth and $R_{\red}$ is a divisor with simple normal crossings, many fundamental properties
of the sheaves $\tF_{(X,R)}$ and their cohomology are by now known. For example, 
the sheaves $\tF_{(X,R)}$ satisfy  Zariski-Nagata purity
(see \cite[Cor 8.6(3)]{S-purity} for $R_\red$ smooth, \cite[Th 2.3]{SaitologRSC}  for $R=R_\red$, and \cite[Th 1.6]{RS-ZNP} in general);
they admit pairings, called reciprocity pairings, generalizing the pairing used in geometric higher dimensional class field theory from \cite{KS-GCFT} (see \cite{RS-ZNP});
  there are projective bundle - and blow-up formulas for the cohomology groups of $\tF_{(X,R)}$ as well as Gysin sequences, see \cite{BRS}. 
 
However, a drawback of the motivic definition of $\tF(X,R)$ is that it is almost impossible to compute directly from its definition. There have been some computations of 
$\tF(X,R)$ in a henselian local situation in \cite{RS}, e.g.,  in characteristic zero for differential forms 
and in positive characteristic for  the Witt vectors of finite length and the torsion characters of the abelianized fundamental 
group. But these computations are very technical and rely on many specific results of several authors.
As a remedy of this drawback, we developed in \cite{RS-HLS} a theory
of higher local symbols which provides an effective tool to determine the modulus of a given section $a\in F(U)$ (see \ref{para:HLS} for its review). The theory plays a crucial role in the proof of the principal result of the present paper, which is a new characterization of $\tF(X, R)$ and a description of the quotient $\tF(X, R)/\tF(X, R-D')$, 
where $D'=(R- R_{\red})_{\red}$, i.e., $D'$ is a SNCD supported on the irreducible components of $R_\red$ which appear with multiplicity $\ge 2$ in $R$. 
It does not use the transfer structure of $F$ and imitates a method  used by 
Abbes and (Takeshi) Saito  in \cite{AbbesSaito11}, \cite{TakeshiSaito},
where they study the ramification of torsors under finite \'etale groups. 
The new characterization of $\tF(X, R)$ is geometric and hopefully more effective 
in applications. 
In section \ref{TopOmega} (see also later in the introduction) we provide its applications on pseudo-rational singularities and representability of cohomology of the sheaf $\tF_{(X,R)}$ on $X_\Nis$ in the triangulated category of motives with modulus $\mathbf{MDM}$ defined  in \cite{KMSY3}.


\medskip

In order to explain the new characterization, we now assume $X$ is smooth over a perfect field $k$ and $R$ is an effective Cartier divisor with  simple normal crossing support.
 We let  $U=X\setminus R$ and $D'=(R- R_{\red})_{\red}$ be as above.
The main geometric construction needed for the new characterization is a certain dilatation. Dilatations were introduced and studied in more general situations 
in \cite[\S2]{AbbesSaito}, \cite[4.1]{AbbesSaito11}, and \cite[1.3]{TakeshiSaito} and originate from \cite[3.2]{BLR}.
For us the following special case is relevant:
The dilatation $P^{(R)}_X$ is the blow-up of $X\times X$ in the closed subscheme $R$ diagonally embedded into $X\times X$ and with the 
strict transforms of $X\times R$ and $R\times X$ removed. The two projections from $X\times X$ to $X$ induce projections from
$P^{(R)}_X$ to $X$ and the inverse image of  $D'$ under either projection is 
the same vector bundle
\[\tag{2}\label{intro2} P^{(R)}_X\stackrel[p_2]{p_1}{\rightrightarrows} X, \qquad 
P^{(R)}_X\times_{p_i, X} D'= \mathbb{V}(\Omega^1_X(R)_{|D'}),\]
see \ref{para:dil}. Here and in the following, differential forms are  relative to $k$
and $\Omega_X^1(R)=\Omega_X^1\otimes_{\sO_X}\sO_X(R)$.  We define
\[\tag{3}\label{intro2.5}
\FAS(X,R)=\{a\in F(U)\mid p_1^*a-p_2^*a\in F(P^{(R)}_X)\}, \]
where we use that $U\times U$ is an open dense subset of  $P^{(R)}_X$ and that the restriction $F(P^{(R)}_X)\to F(U\times U)$ is injective,
see Definition \ref{defn:FAS}.

\medskip

One of the  main results of the present paper is the following, see Theorem \ref{thm:ASII}:

\begin{theorem}\label{thm:intro1}
In the above situation  assume
\[(*)\qquad (X,R) \text{ has a projective SNC-compactification (see \ref{para:sm-cptf}).}\]
Then 
\[\tF(X,R)=\FAS(X,R).\]
\end{theorem}
The assumption $(*)$ is satisfied tautologically if $X$ is projective and also if $X$ is quasi-projective and $\ch(k)=0$ by resolution of singularities. Even if $X$ is not projective or $\ch(k)>0$, the assumption $(*)$ is not necessary under a certain extra condition (namely that $F$ has level $\le 3$).
In particular, the theorem holds without this extra condition  for $F=H^1_{\rm fppf}(-, G)$, where $G$ is a commutative finite $k$-group scheme, 
see \ref{ex:ASII} for this and more examples.  Thus  for $G$ an \'etale commutative $k$-group,  
a $G$-tosor over $U$ has ramification bounded by $R$ as defined in 
\cite{TakeshiSaito}\footnote{The notion of bounded ramification considered in \cite{AbbesSaito11} is a log-version of this.}
 if and only if  the $G$-torsor has modulus $(X,R)$ in the sense  of reciprocity sheaves.
Another remark is that by Theorem \ref{thm:intro1}, $(X,R)\mapsto \FAS(X,R)$ defines a Nisnevich sheaf on the category of modulus correspondences satisfying $(*)$, 
which seems non-trivial from the definition \eqref{intro2.5}.

The proof of Theorem \ref{thm:intro1} in the base case $R=nD$ with $D$ a smooth divisor 
relies heavily on the theory of higher local symbols along Par\v{s}in chains for reciprocity sheaves and a characterization 
of the modulus in terms of these symbols, which was proven in \cite{RS-HLS} and is recalled in section \ref{sec:1}. 
The general statement then follows from this base case and the purity statements from \cite{S-purity} and \cite{RS-ZNP}. 
The inclusion $\tF(X,R)\subset \FAS(X,R)$  holds without assuming $(*)$, and is proven in section \ref{sec:2}, see Theorem \ref{thm;AS}.
From this, we deduce a {\em Brylinski-Kato formula} for the motivic conductor defined by $F$, see Theorem \ref{thm:BK} and Remark \ref{rmk:BK} for details.
The other inclusion is proved assuming $(*)$ at the end of section \ref{sec:AS-ctd}.

\medskip

The definition of bounded ramification  along a divisor of a $G$-torsor 
for an \'etale commutative $k$-group $G$ in terms of  dilatations is used in 
\cite{TakeshiSaito} to define a non-log version of Kato's refined Swan conductor, the so called 
characteristic form of the torsor. We can use Theorem \ref{thm:intro1} to define a 
characteristic form for any reciprocity sheaf $F$, which for $R=nD$ with $D$ smooth on $X$ and $n\ge 2$
is  a map (see \ref{rmk:charF})
\[\charr^{(nD)}_F: \tF_{(X,nD)}\to \Omega_X^1(nD)\otimes_{\sO_X} i_*\uHom_{\Sh}(\sO, F)_{D_\Nis},\]
where $\sO$ denotes the structure sheaf on $\Sm$,
$i: D\inj X$ denotes the closed immersion, and  $\uHom_{\Sh}(\sO, F)_{D_\Nis}$ denotes the 
restriction to $D_{\Nis}$ of the internal hom in the category of abelian sheaves on $\Sm$, which has a natural structure of $\sO_{D}$-module.

In case  $F=H^1:=\Hom(\pi_1^{\rm ab}(-),\Q/\Z)$ is the sheaf of torsion characters 
of the abelianized \'etale fundamental group, the Artin-Schreier-Witt sequences induce a natural morphism $\sO_D\to \uHom_{\Sh}(\sO, H^1)_{D_\Nis}$ 
and the characteristic form factors over the more familiar
looking map
\[\widetilde{H^1}_{(X,nD)}\to \Omega_X^1(nD)\otimes_{\sO_X} i_*\sO_D,\]
which at the generic point of $D$ coincides with the map defined in \cite[Definition 1.18]{Yatagawa},
except that for $n=2$ in characteristic 2, our formula differs from the one in {\em loc. cit.} by a square root,
see \ref{para:CW}. Note however that this factorization is particular to the case $F=H^1$. 
For example, if $F$ is equal to $W_n$ the Witt vectors of length $n$,  or the differential forms $\Omega^j$, then this is not the case, 
see Corollary \ref{prop:charW} and Theorem \ref{prop:Omega-char}.

The characteristic form of  $a\in F(X,nD)$ is defined as follows: By Theorem \ref{thm:intro1}
we have $p_1^*a-p_2^*a\in F(P^{(nD)}_X)$ which we may restrict to $F(\mathbb{V}(\Omega^1_X(nD)_{|D}))$
to obtain an element $\psi_n(a)$, see \eqref{intro2}. In fact $\psi_n(a)$ is contained in the \emph{additive part}, on which an isomorphism
\[\chi_{F,\sO_D} :F(\mathbb{V}(\Omega^1_X(nD)_{|D}))_{\ad}\xr{\simeq} 
 \Gamma(X, \Omega^1_X(nD)\otimes_{\sO_X} i_*\uHom_{\Sh}(\sO, F)_{D_{\Nis}}),\]
 is constructed, see \eqref{eq;HomSmB}. 
 This defines the characteristic form. The definition can be extended to a general pair $(X,R)$, where $R_{\red}$ is not necessarily regular but a SNCD 
 and $D$ is replaced by $D'=(R- R_{\red})_{\red}$.

\medskip

The following is the sheaf version of Theorem \ref{thm;charform}, see \ref{para:charform-local}.
\begin{theorem}\label{thm:intro2}
Assume $(X,R)$ satisfies condition $(*)$ from Theorem \ref{thm:intro1}. Then the characteristic form 
$\charr^{(R)}_F$  induces an injection of Zariski sheaves on $X$
\[\tF_{(X, R)}/\tF_{(X, R-D')}\inj \Omega^1_{X}(R)\otimes_{\sO_X}i_*\uHom(\sO, F)_{D'_{\Zar}}.\]
\end{theorem}
We remark that for smooth $D$, the Gysin sequence from \cite{BRS} which generalizes Voevodsky's Gysin sequence for $\A^1$-invariant sheaves with transfers, yields an isomorphism
\[ \tF(X,D)/F(X)\xr{\simeq} \uHom_{\PST}(\G_m, F)(D).\]
\medbreak

In section \ref{sec:ex-grps} the examples of $F=W_n$ and $F=H^1$ are discussed (in positive characteristic). 
Computations from Abbes-Saito \cite[\S 12]{AbbesSaito},  Yatagawa \cite[\S 2]{Yatagawa}, and \cite{RS}  
 determine the characteristic form in these cases completely. 
In fact, this shows that the general definition of the characteristic form fits nicely into the picture of various refined Swan conductors which
were defined before by different methods in \cite[\S 5]{KatoSwan}, \cite[\S 3]{Matsuda}, and \cite[\S 4]{Kato-Russell}, 
see Remark \ref{rmk:KR} and \ref{para:CW}.
As another exemplary application, we  explain in \ref{para:app-Chow} how  Theorem \ref{thm:intro2} 
reveals an interesting structure of Chow groups of zero-cycles with modulus, as introduced in \cite{Kerz-Saito}.
In particular we obtain some new specialization maps.

\medbreak

In view of these explicit computations  and also the one from section \ref{sec:ex-diff} discussed below,
it is an intriguing  question, whether it is possible  to determine the image of the characteristic form
for a general reciprocity sheaf and a general pair $(X,R)$. 
For the moment the image seems to be rather mysterious and we do not have a guess for a general formula.

\medskip

Let $p={\rm char}(k)\ge 0$. In section \ref{sec:ex-diff} we use Theorems \ref{thm:intro1} and \ref{thm:intro2} to compute
$\widetilde{\Omega^j}_{(X, nD)}$, for $X$, $D$ smooth, and $j, n\ge 1$. We find 
\[\tag{4}\label{intro3}
\widetilde{\Omega^j}_{(X,nD)}=
\begin{cases} \Omega^j_X(\log D)((n-1)D) & \text{if } p=0 \text{ or } (p\neq 0 \text{ and } p\nmid n)\\
                     \Omega^j_X(nD)& \text{if } p\neq 0 \text{ and } p\mid n,   \end{cases}
\]
see Theorem \ref{prop:Omega-char} and Corollary \ref{cor:tOmega}.
The case $p=0$ was proved by a different method in \cite[Theorem 6.4]{RS} and the case $p>0$ is new.
We  also compute the characteristic form. In particular, the formula for $\charr^{(nD)}_{\Omega^j}$ in the case $p|n$, came as a surprise,
see Theorem \ref{prop:Omega-char}\ref{prop:Omega-charIII}, \ref{prop:Omega-chariv} and \ref{prop:Omega-charv}.

In \cite[Question 2]{KMSY1} it is asked whether the isomorphism \eqref{intro1} extends to an isomorphism 
$H^i(X, \tF_{(X,R)})\cong H^i(X', \tF_{(X',f^*R)})$, at least under the additional assumption that $X$ is smooth, $R_\red$ is a SNCD, and 
$f:X'\to X$ is the blow-up in a smooth closed subscheme contained in $R_{\red}$. The formula 
\eqref{intro3} shows that in positive characteristic the answer to this question is  in general negative.\shuji{\footnote{In his forthcoming paper, Shane Kelly gives a positive answer to the same question for $F=\Omega^j$ in case $\ch(k)=0$.}}  
Indeed, this follows from the blow-up formula by considering  $\widetilde{\Omega^j}_{(X,pD)}$, see  \ref{para:Q}. 
 
However top differentials have a very good behavior even in positive characteristic. More precisely, 
we deduce from \eqref{intro3} that on a smooth scheme $X$ of dimension $d$, we have for any effective Cartier divisor $R$ on $X$ 
(without $R_{\red}$ assumed to be a SNCD)
\[\widetilde{\Omega^d}_{(X,R)}= \Omega^d_X(R),\]
see Lemma \ref{lem:tomega}. This leads to the following new characterization of pseudo-rational singularities, see \ref{para:pr} for a reminder of the definition
and Corollary \ref{cor:pr} for the result.

\begin{theorem}\label{thm:intro3}
Let $Y$ be an integral, normal, Cohen-Macaulay scheme of finite type over $k$ and dimension $d$.
Assume that the sheaf  $\widetilde{\Omega^d}_{(Y,R)}$  is $S2$, 
for each  effective Cartier divisor $R$ on $Y$ with $Y\setminus R$ smooth. 
Then $Y$  has pseudo-rational singularities. 
If there exits a proper and birational morphism $f:Z\to Y$ with $Z\in \Sm$,
then the inverse implication holds. 
In both cases  \eqref{eq1;para:pr} induces
\[\widetilde{\Omega^d}_{(Y, R)}\simeq \omega_{Y/k}(R),\]
where    $\omega_{Y/k}$ denotes the relative dualizing sheaf of $Y/k$.
\end{theorem}

A version of the above theorem for rational singularities in characteristic 0, was given in \cite[7.1]{RS-ZNP}.
The sheaves $\widetilde{\Omega^j}_{(Y, R)}$, $j\ge 0$, are defined for any reduced $k$-scheme $Y$ with effective Cartier divisor $R$,
as long as $R$ contains the singular locus and their formation is functorial in the pair $(Y,R)$ (in contrast to reflexive differentials, which are also used in
the study of singular varieties). This together with Theorem \ref{thm:intro3} suggests that these sheaves might be useful in the study of singular varieties.

\medskip

Combining Theorem \ref{thm:intro3} with
some of the main results from \cite{Kovacs} and \cite{KMSY1}, we obtain (see Corollary \ref{cor:Ext-MNST}):
\begin{cor-intro}\label{cor:intro1}
Let $Y$ have pseudo-rational singularities and assume
for any proper and birational morphism $Z\to Y$, there exits a proper and birational morphism 
	$Z'\to Z$ with $Z'$ smooth (e.g. $\ch(k)=0$ or $d\le 3$ and $Y$ quasi-projective by \cite{Cossart-PiltantI}, \cite{Cossart-PiltantII}).

Then, for any effective Cartier divisor  $R$ on $Y$ such that $Y\setminus R$ is smooth, we have
\[\tag{5}\label{intro4} H^i(Y_\Zar, \widetilde{\Omega^d}_{(Y,R)})=\Ext^i_{\uMNST}(\Ztr(Y,R), \widetilde{\Omega^d}),\]
where $\uMNST$ denotes the category of modulus Nisnevich sheaves with transfers introduced in \cite{KMSY1}
and $\Ztr(Y,R)\in \uMNST$ is the object represented by $(Y,R)$.
\end{cor-intro}
Under some stronger assumptions on resolutions of singularities, the cohomology group on the left in \eqref{intro4} becomes representable 
in the triangulated category of motives with modulus $\uMDM$ introduced in \cite{KMSY3}, see Remark \ref{rmk:MDM}.

\begin{acknowledgement}
The second named author thanks Ahmed Abbes and Yuri Yatagawa for their kindly  answering questions on ramification theory and Takeshi Saito for his helpful suggestion on Lemma \ref{lem4;AS}. The first named author thanks Fei Ren for a careful reading of a preliminary version of this article. 
\end{acknowledgement}

\addtocontents{toc}{\protect\setcounter{tocdepth}{1}}
\subsection{Notations}\label{nota}
In this article $k$ denotes a {\em perfect} field and $\Sm$ the category of separated schemes 
which are smooth and of finite type over $k$. For $k$-schemes $X$ and $Y$ we write $X\times Y:= X\times_k Y$.
For $n\ge 0$ we write $\P^n:=\P^n_k$, $\A^n:=\A^n_k$.

For a scheme $X$ we denote by $X_{(i)}$ (resp. $X^{(i)}$) 
the set of $i$-dimensional (resp. $i$-codimensional) points of $X$.
If $(R,\fm)$ is a local ring, then we denote by 
\[R\{x_1,\ldots, x_n\}\]
the henselization of  the polynomial ring $R[x_1,\ldots, x_n]$ at the ideal 
$\fm R[x_1,\ldots, x_n]+ (x_1,\ldots, x_n)$.

Let $F$ be a Nisnevich sheaf on a scheme $X$ and $x\in X$ a point.
Then we denote by $F_x$ its Zariski stalk and by $F_x^h=\varinjlim F(U)$ the Nisnevich stalk, 
where the limit is over all Nisnevich neighborhoods $U\to X$ of $x$.
If $X$ is reduced, we denote by
\[K^h_{X,x}=\Frac(\sO_{X,x}^h)\]
the total fraction ring of $\sO_{X,x}^h$.

\section{Review of higher local symbols of reciprocity sheaves}\label{sec:1}
In this section we recall the definition of reciprocity sheaves from \cite{KSY2}
and the construction and main properties of higher local symbols from \cite{RS-HLS},
which we will use repeatedly in the next sections.

\begin{para}\label{para:rec}
We recall the notion of reciprocity sheaf from \cite{KSY2}.
Let $F\in\NST$ be a Nisnevich sheaf with transfers on $\Sm$ in the sense of Voevodsky. 
Let $U\in \Sm$ and $a\in F(U)$. A pair $(X,D)$ consisting of a proper $k$-scheme $X$ and an effective (possibly empty) Cartier divisor $D$ with
$X\setminus D=U$ is called a {\em modulus} for $a$ if for each $S\in \Sm$  and each finite
integral correspondence $Z\in \Cor(\A^1_S, U)$ satisfying $\{\infty_S\}_{|\tilde{Z}}\ge D_{|\tilde{Z}}$, where 
$\tilde{Z}\to \P^1_S\times X$ is the normalization of the closure of $Z$, we have  $Z_0^*a=Z_1^*a$,
where $Z_\e$ denotes the restriction of $Z$ along $\e_S\inj  \A^1_S$, $\e\in \{0,1\}$. 
Following \cite{KSY2}  we say $F$ is a  {\em sheaf with $SC$-reciprocity} (or just {\em reciprocity sheaf}), if
for any $U\in\Sm$  any section $a\in F(U)$ has a modulus.
If $(X,D)$ is any modulus pair, i.e.,  $X$ is a separated finite type $k$-scheme (not necessarily proper) 
and $D$ is an effective Cartier divisor on $X$ such that $U=X\setminus D$ is smooth, then
we denote by 
\[\tF(X,D)\]
the subgroup of $F(U)$ consisting of all $a\in F(U)$ which have a modulus of the form $(\Xb, \ol{D}+B)$, where
$\ol{X}$ is proper, $\ol{D}$ and $B$ are effective Cartier divisors on $\ol{X}$, and $X=\ol{X}\setminus B$ and 
$\ol{D}_{|X}=D$.
We can extend the definition of $\tF$ to pro-modulus pairs in an obvious way.
By \cite[Theorem 0.1]{S-purity} the category of reciprocity sheaves $\RSC_\Nis$ is abelian, and by \cite[Corollary 2.4.2]{KSY2} it is even Grothendieck.
\end{para}

\begin{para}\label{para:chain}
Let  $X$ be a reduced, separated, and noetherian scheme of finite dimension $d$ and assume $X_{(0)}=X^{(d)}$. 
A specialization chain (or just chain) on $X$ is a sequence of points 
\[\ux=(x_0,x_1,\ldots, x_n)\quad \text{with } x_i\in X \text{ and } x_{i-1}\in \ol{\{x_i\}}, \quad \text{for all } i=1,\ldots, n,\]
where $\ol{\{x_i\}}$ denotes the closure of the point $x_i$ in $X$; $\ux$ is a {\em maximal} chain or Par\v{s}in chain 
if $n=d$ and $x_i\in X_{(i)}$, for all $i$. We set
\[\mc(X)=\{\text{maximal chains on } X\}.\]
A {\em  maximal chain with break at $r\in\{0,\ldots,d\}$} is a chain $\ux$ as above with $n=d-1$, 
$x_i\in X_{(i)}$, for $i<r$, and $x_{i}\in X_{i+1}$, for $i\ge r$. We write
\[\mc_r(X)=\{\text{maximal chains with break at $r$ on } X\}.\]
For $\ux=(x_0,\ldots, x_{r-1}, x_{r+1},\ldots, x_d)\in \mc_r(X)$ we denote by $\b(\ux)$ the set of breaks of $\ux$, i.e.,
the points $y\in X$ such that 
\[\ux(y)=(x_0,\ldots, x_{r-1},y,x_{r+1}, \ldots, x_d)\in \mc(X).\] 
\end{para}

\begin{para}\label{para:HLS}
Let $F\in \RSC_\Nis$. We denote by $K^M_r$ the Nisnevich sheafification of the improved Milnor $K$-theory sheaf from
\cite{Kerz}; we denote by $K^M_{r,X}$ its restriction to $X$.
Let $K$ be a function field over $k$ and $X$ an integral scheme of finite type over $K$ of dimension $d$.
By \cite[(5.1.3)]{RS-HLS} we have  a pairing, called {\em higher local symbol},   
\eq{para:HLS1}{(-,-)_{X/K,\ux}: F(K(X))\otimes_{\Z} K^M_d(K^h_{X,x_{d-1}})\to F(K),}
for each maximal chain $\ux=(x_0,\ldots,x_{d-1}, x_d)\in \mc(X)$, 
which 
satisfies  the following properties for all $a \in F(K(X))$:
\begin{enumerate}[label= ({HS\arabic*}$'$)]
\item\label{HS1}
Let $X\hookrightarrow X'$ be an open immersion, where $X'$ is an integral $K$-scheme of dimension $d$. Then 
\[(a,\beta)_{X/K,\ux}=(a,\beta)_{X'/K,\ux},\quad \text{for all } \beta\in K^M_d(K^h_{X, x_{d-1}}).\]

\item\label{HS2} Let $X_{d-1}\subset X$ be the closure 
of $x_{d-1}$, and set $\ux'=(x_0,\ldots, x_{d-1})\in \mc(X_{d-1})$.
Then for all $\beta\in K^M_d(K^h_{X,x_{d-1}})$
\[(a, \beta)_{X/K, \ux} =\begin{cases} 
      \beta\cdot \Tr_{K(X)/K}(a), & \text{if }d=0,\\
     (a(x_{d-1}), \partial_{x_{d-1}}\beta)_{X_{d-1}/K, \ux'}, & \text{if } d\ge 1 \text{ and } a\in F(\sO_{X, x_{d-1}}),
\end{cases}\]
where 
$a(x_{d-1})\in F(K(X_{d-1}))$ is the restriction of $a$, in case $d=0$, the map $\Tr_{K(X)/K}: F(K(X))\to F(K)$ is the trace induced by the transpose of the graph 
of $\Spec K(X)\to \Spec K$, and $\partial_{x_{d-1}}=\sum_{v/x_{d-1}} \Nm_{K(v)/ K(x_{d-1})}\circ \partial_v$, where the sum is over all 
discrete valuations on $K(X)$ dominating $x_{d-1}$, $\Nm_{K(v)/ K(x_{d-1})}$ is the norm on Milnor $K$-theory, and $\partial_v$  is the tame symbol at $v$.

\item\label{HS3} Let $D\subset X$ be an effective Cartier divisor such that $X\setminus D$ is regular
and $a\in \tF(X,D)$. Then
\[(a, \beta)_{X/K,\ux}=0,\quad \text{for all }  \beta\in  (V_{d, X|D})^h_{x_{d-1}},\]
where $V_{1, X|D}=\Ker(\sO_X^\times\to \sO^\times_D)$ and $V_{r, X|D}=\Im(V_{1,X|D}\otimes_\Z K^M_{r-1,X}\to K^M_{r,X})$.

\item\label{HS4} 
Let $\ux'\in \mc_r(X)$ with $0\leq r\leq d-1$. Then for all $\beta\in K^M_d(K(X))$
\[(a, \beta)_{X/K,  \ux'(y)}=0,\quad \text{for almost all } y\in \b(\ux').\]
If either $r\ge 1$ or $r=0$, $X$ is quasi-projective, and the closure of $x_1$ in $X$ 
is projective over $K$, where $\ux'=(x_1,\ldots, x_d)$, then
\[\sum_{y\in b(\ux')} (a, \beta)_{X/K,  \ux'(y)}=0. \]

\item\label{HS5}
Let $f:Y\to X$ be a dominant and quasi-projective $K$-morphism between integral $K$-schemes 
of the same dimension $d$. Set $u:=x_{d-1}$ and let $y\in Y^{(1)}$ with $f(y)=u$.
We assume that $f$ induces a projective  morphism 
between the closures of the points $y$ and $u$. 
Then $K^h_{Y, y}$ is finite over $K^h_{X, u}$ 
and  for all  $\beta\in K^M_d(K^h_{Y, y})$
we have
\[\sum_{\substack{\uz \in \mc_{d-1}(Y) \text{ with} \\ y\in\b(\uz)\text{ and } f(\uz(y))=\ux}} (f^*a, \beta)_{Y/K,\uz(y)}
=\big(a, \Nm_{y/u}(\beta)\big)_{X/K,\ux},\]
where  $\Nm_{y/u}: K^M_d(K^h_{Y, y})\to K^M_d(K^h_{X, u})$ is the norm map. 
\end{enumerate}
Here \ref{HS5} is \cite[Corollary 5.5]{RS-HLS}, the properties \ref{HS1}-\ref{HS4} are slight variants of the (stronger) properties (HS1)-(HS4) in 
\cite[Proposition 5.3]{RS-HLS}, where they are stated for $\beta\in K^M_{d}(K^h_{X,\ux})$, with $K^h_{X,\ux}$ an iterated henselization along the chain $\ux$.
The version stated here follows easily using the natural maps $K(X)\to K^h_{X, x_{d-1}}\to K^h_{X,\ux}$.
To deduce \ref{HS2} from (HS2) in {\em loc. cit.}, use the fact that 
$K^h_{X,\ux}$ is a finite product of henselian dvr's unramified over $K^h_{X, x_{d-1}}$ together with
standard formulas for the tame symbol and the norm on Milnor $K$-theory,  
namely  \cite[{\bf R1c}, {\bf R3a}]{Rost}.

As already done in \ref{HS4} we usually also write
\eq{para:HLS2}{(-,-)_{X/K,\ux}: F(K(X))\otimes_{\Z} K^M_d(K(X))\to F(K)}
for the pairing obtained from \eqref{para:HLS1} by precomposing with the natural map
\[K^M_d(K(X))\to K^M(K^h_{X, x_{d-1}}).\] 
\end{para}

If $X$ is a smooth $k$-scheme and $D$ is supported on a simple normal crossing divisor, then we can use 
the higher local symbols to decide whether $a\in F(X\setminus D)$ is regular on $X$,
 or, under some extra assumption, whether it has modulus $D$.
This  is the content of the next two results from  \cite{RS-HLS}.

\begin{prop}[{\cite[Proposition 7.3(1)]{RS-HLS}}]\label{prop:HLS-reg}
Let $X\in \Sm$  have pure dimension $d$ and let $D$ be  a SNCD on $X$. Let $W\subset X$ be an open subscheme containing all generic points of $D$.
Assume $a\in F(X\setminus D)$ satisfies for all function fields $K/k$ and all $\ux=(x_0,\ldots, x_d)\in\mc(W_K)$
with $x_{d-1}\in D_K^{(0)}$
\[(a_K, \beta)_{X_K/K,\ux}=0,\quad \text{for all }\beta\in K^M_d(\sO_{X_K,x_{d-1}}),\]
where $X_K=X\otimes_k K$ and $a_K\in F(X_K\setminus D_K)$ is the restriction of $a$. Then $a\in F(X)$.
\end{prop}

\begin{para}\label{para:sm-cptf}
Let $X\in \Sm$ and let $D$ be an effective Cartier divisor on $X$ whose support has simple normal crossings.
We say that $(X,D)$ has a {\em projective SNC-compactification}, if there exists a dense open immersion $j: X\inj \Xb$, such that 
$\Xb$ is smooth and projective over $k$ and $\Xb\setminus (X\setminus D)$ is the support of a SNCD on $\Xb$.

We note that in characteristic $0$ any pair $(X,D)$ as above has a  projective SNC-compactification by Hironaka;
 the same holds in positive characteristic if $\dim X\le 3$, by \cite[Theorem on p. 1839]{Cossart-PiltantII}
together with \cite[Proposition 4.1]{Cossart-PiltantI}.
\end{para}

\begin{thm}[{\cite[Theorem 6.1]{RS-HLS}}]\label{thm:HLS-mod}
Let $X\in \Sm$ have pure dimension $d$ and let $D$ be an effective Cartier divisor whose support has simple normal crossings.
Assume that $(X, D)$ has a projective SNC-compactification.
Set  $U=X\setminus D$. Let $W\subset X$ be an open subscheme containing all generic points of $D$. 
Then for $a\in F(U)$ the following two statements are equivalent:
\begin{enumerate}
\item $a\in \tF(X,D)$. 
\item For all function fields $K/k$ and  all $\ux=(x_0,\ldots, x_{d-1},x_d)\in \mc(W_K)$ with $x_{d-1}\in D_K$  we have 
\[(a_K,\beta)_{X_K/K,\ux}=0,\quad \text{for all } \beta\in (V_{d, X_K|D_K})_{x_{d-1}}.\]
\end{enumerate}
\end{thm}

\section{Abbes-Saito formula}\label{sec:2}
In this section $X$ is a smooth separated $k$-scheme and $R$ is an effective  Cartier divisor on $X$.
We assume that $D=R_{\red}$ is a SNCD.

\begin{para}\label{para:dil}
We recall the construction and basic properties of a dilatation. This is discussed in a more general situation 
in \cite[1.3]{TakeshiSaito}, see also  \cite[2.]{AbbesSaito}, \cite[3.2]{BLR} for similar discussions.
Let $\rho:{\rm Bl}_R(X\times X)\to X\times X$ be  the blow-up of $X\times X$ in $R$, 
which we embed diagonally in $X\times X$, i.e., we identify $R$ with 
$\Delta_X\cap (X\times R\cup R\times X)$. We set 
\[P^{(R)}_X= {\rm Bl}_R(X\times X)\setminus \widetilde{R\times X}\cup \widetilde{X\times R},\]
where $\widetilde{R\times X}$ denotes the strict transform of $R\times X$ in ${\rm Bl}_R(X\times X)$ and similar with $\widetilde{X\times R}$.
For $R=\emptyset$  the empty divisor, we have  $P^{(\emptyset)}_X=X\times X$.
Let $p_1, p_2: P^{(R)}_X\to X$ be the two morphisms induced by the composition of $\rho$ with the two projections $q_1, q_2: X\times X\to X$.

By construction we have an equality of closed subschemes of $P_X^{(R)}$ 
\[p_1^{-1}(R)= p_2^{-1}(R)=P^{(R)}_X\times_{X\times X, \Delta} R:=R_P.\]
In fact  $R_P$ is the restriction of the exceptional divisor of $\rho$ to $P^{(R)}_X$ and is hence 
also an effective Cartier divisor, furthermore
\[p_{1|R_P}=p_{2|R_P}: R_P\inj P^{(R)}_X\to X.\]
The triple $(P^{(R)}_X, p_1, p_2)$ is the universal example of such a structure, more precisely:
Let $Y$ be a $k$-scheme  and let $f,g: Y\to X$ be two $k$-morphisms satisfying
\begin{enumerate}[label=(\arabic*)]
\item\label{para:dil1} $f(Y)$, $g(Y)\not\subset R_\red$;
\item\label{para:dil2} $f^{-1}(R)=g^{-1}(R)=:Z$ and $Z$ is an effective Cartier divisor;
\item\label{para:dil3} $f_{|Z}=g_{|Z}: Z\to X$.
\end{enumerate}
Then there exits a unique morphism $h: Y\to P^{(R)}_X$ such that $f=p_1\circ h$ and $g=p_2\circ h$.
This follows easily from the universal properties of the product and the blow-up.
\end{para}

The following local description of $P^{(R)}_X$ will be useful.

\begin{para}\label{para:dil-loc}
Assume $X=\Spec A$, with $A$ a smooth $k$-algebra of dimension $d$, and $f\in A$ is an equation defining $R$.
Let $I_\Delta\subset A\otimes_k A$ be the ideal of the diagonal $X\inj X\times X$.
For any  point $y\in X\times X$ which lies in the image of the diagonal we find an affine open neighborhood 
$V=\Spec B\subset X\times X$ of $y$, such that $I_{\Delta}B$ is generated by a regular sequence $\theta_1,\ldots, \theta_d$.
We thus find $b_i\in B$ such that 
\[1\otimes f= f\otimes 1+ \sum_{i=1}^d b_i \theta_i.\]
Then a direct computation shows 
\eq{para:dil-loc1}{(P^{(R)}_X)_{|V}= \Spec \frac{B[\tau_1,\ldots, \tau_d, \tfrac{1}{1+\sum_i b_i \tau_i}]}{(\theta_i-\tau_i (f\otimes 1),\, i=1,\ldots, d)},}
where $(P^{(R)}_X)_{|V}= P^{(R)}_X\times_{X\times X} V$.

More specifically, assume $A$ is \'etale over $k[z_1,\ldots, z_d]$ and $R$ is defined by $z_1^n$.
Set
\[t_i=z_i\otimes 1 \quad \text{and}\quad   s_i=1\otimes z_i\in A\otimes_k A.\]
Then the diagonal $\Delta_X$ is open and closed in $X\times_{\A^n} X= V(s_1-t_1, \ldots, s_d-t_d)$  and hence we find
$V=\Spec B\subset X\times X$ an open neighborhood of $\Delta_X$,  
such that $I_{\Delta}B$ is generated by the regular sequence $\theta_1:=s_1-t_1,\ldots, \theta_d:=s_d-t_d$.
Since 
\[s^n_1= (t_1+\theta_1)^n= t_1^n +\theta_1 b_1,\quad \text{with }b_1=\sum_{i=1}^n \binom{n}{i} t_1^{n-i}\theta_1^{i-1}\]
we get 
\eq{para:dil-loc2}{(P^{(R)}_X)_{|V}
=\Spec \frac{B[\tau_1,\ldots, \tau_d, \tfrac{1}{1+t_1^{n-1} \tau_1}]}{(\theta_i-\tau_i t_1^n,\, i=1,\ldots, d)},}
where we use 
\[1+b_1 \tau_1= (1+t_1^{n-1}\tau_1)^n\quad \text{in}\quad \frac{B[\tau_1,\ldots, \tau_d]}{(\theta_i-\tau_i t_1^n,\, i=1,\ldots, d)}.\]
\end{para}

\begin{lem}\label{lem:dil}
Let the assumptions and notations be as in \ref{para:dil}.
\begin{enumerate}[label=(\arabic*)]
\item\label{lem:dil1} 
Set $D'=(R-D)_{\red}$.  Then
\[R_P\times_X D'=\mathbb{V}(\Omega^1_{X/k}(R)_{|D'}):=\Spec ({\rm Sym}^\bullet \Omega^1_{X/k}(R)_{|D'}),\]
where $\Omega^1_{X/k}(R)=\Omega^1_{X/k}\otimes_{\sO_X} \sO_X(R)$.
\item\label{lem:dil2} $P^{(R)}_X\setminus R_P=U\times U$ and the restrictions of $p_1$,  $p_2$ to $U\times U$ are the
projections to the first and second factor, where $U=X\setminus R$.
\item\label{lem:dil3} $P^{(R)}_X$ is smooth.
\item\label{lem:dil3.5} The diagonal $X\inj X\times X$ lifts by the universal property 
of $(P^{(R)}_X, p_1, p_2)$ to a closed immersion $X\inj P^{(R)}_X$.
\item\label{lem:dil4} There exists a smooth map 
\[ \mu: P_X^{(R)} \times_{p_2,X, p_1} P^{(R)}_X \to P^{(R)}_X\]
which makes the following diagram commutative:
\eq{eq1;lem,AS}{\xymatrix{
P^{(R)}_X \times_{p_2,X, p_1} P^{(R)}_X \ar[r]^-{\mu} \ar[d] &  P^{(R)}_X \ar[d] \\
(X\times X)\times_{p_2,X, p_1} (X\times X) =X\times X\times X \ar[r]^-{p_{13}} & X\times X.}}
Moreover,  the pullback of $\mu$ via $D'=(R-D)_{\red}\hookrightarrow X\times X$ 
\[ (R_P\times_X D')\times_{D'} (R_P\times_X D') \to R_P\times_X D'\]
is the addition of the vector bundle over $D'$ from \ref{lem:dil1}. 
\item\label{lem:dil5} Let $S$ be another effective Cartier divisor with $(R+S)_\red$ SNCD. Then there exists a morphism 
\[\pi: P_X^{(R+S)}\to P_X^{(R)},\]
which identifies $P_X^{(R+S)}$ with the complement of the strict transforms of $p_1^*S$ and $p_2^*S$
in the blow-up of $P_X^{(R)}$ in $S$, where we embed $S$ into $P^{(R)}_X$ via  \ref{lem:dil3.5}.
\end{enumerate}
\end{lem}
\begin{proof}
Versions of the  statements can be found in \cite[1., 2.]{TakeshiSaito}. Since we also need an explicit description 
of the maps involved we sketch the proofs.
\ref{lem:dil1}. We use the description \eqref{para:dil-loc1}.
Let $J\subset A$ be the ideal of $D'$ in $X$.
We have $f\in J^2$ and thus  
$1\otimes f-f\otimes 1\in I_{\Delta} \cdot (A\otimes J)+I_{\Delta}^2$.
Hence the 
$b_i$ in \eqref{para:dil-loc1} are in $I_{\Delta}B+ (A\otimes J)B$
and 
\[P^{(R)}_X\times_{X\times X}(V\cap D')= \Spec( (A/J)[\tau_1, \ldots, \tau_d]).\]
It is straightforward to check that
\eq{lem:dil7}{\bigoplus_{i=1}^d (A/J)\cdot \tau_i \lra \Omega^1_{A}\cdot \frac{1}{f}\otimes_{A} A/J
= \frac{I_{\Delta}}{I_{\Delta}^2}\cdot\frac{1}{f}\otimes_A A/J,\quad  \tau_i\mapsto \frac{\theta_i}{f},\, i=1,\ldots, d,}
defines an $A/J$-linear isomorphism,
which is independent of the choice of the local equation $f$ 
and the local generators $\theta_1, \ldots, \theta_d$ of $I_{\Delta}$.
\ref{lem:dil2} is obvious. 
\ref{lem:dil3} is local. We may therefore assume that there is an affine and \'etale map $X\to \A^d=\Spec k[z_1,\ldots, z_d]$,
an open subset $V=\Spec B\subset X\times X$, such that $V\times_{\A^d\times\A^d} \Delta_{\A^d}= \Delta_X$, 
and a function $f\in k[z_1, \ldots, z_d]$ whose pullback to $X$ defines $R$.
 In this case
\[P^{(R)}_{X}\times_{X\times X} V= P^{(\Div(f))}_{\A^d}\times_{\A^d\times\A^d} V \]
is \'etale over $P^{(\Div(f))}_{\A^d}$. 
The formula \eqref{para:dil-loc1} for  $P^{(\Div(f))}_{\A^d}$, with $\theta_i=s_i-t_i$, 
directly implies that the latter scheme is smooth, hence so is $P^{(R)}_X$.
\ref{lem:dil3.5} is immediate.
\ref{lem:dil4}. Composing the map $P^{(R)}_X\times_X P^{(R)}_X\to X\times X$ defined by the lower part of \eqref{eq1;lem,AS}
with the two projections gives two maps $f,g: P^{(R)}_X\times_X P^{(R)}_X\to X$, which satisfy the properties
\ref{para:dil1}-\ref{para:dil3} of \ref{para:dil}. Hence the map $\mu$ from the statement exists by the universal property of $P^{(R)}_X\to X$.
Using the local description \eqref{para:dil-loc1} we see that it is induced by the map
\[B[\tau_1,\ldots, \tau_d]\to C[\tau^{(1)}_1,\ldots, \tau^{(1)}_d, \tau^{(2)}_1,\ldots, \tau^{(2)}_d],\]
with $C$ a localization of $B\otimes_A B$, which on the coefficients is given by 
a localization of the  map $A\otimes A\to B\otimes_A B$,  $a_1\otimes a_2\mapsto a_1\otimes 1\otimes 1\otimes a_2$, for $a_i\in A$, and on the 
$\tau_i$ by the assignment
\[\tau_i\mapsto \tau^{(1)}_i+ (1+\sum_i (b_i\otimes 1\otimes 1)\tau^{(1)}_i)\tau^{(2)}_i, \quad i=1,\ldots, d,\]
where $b_i$ are from \eqref{para:dil-loc1}. 
From this description and the fact that $b_i\in I_{\Delta}B+(A\otimes J)B$ 
the second statement of \ref{lem:dil4} is immediate. Finally, the existence of a morphism $\pi$ as in \ref{lem:dil5}
follows from the universal property of $(P^{(R)}_X, p_1, p_2)$, and the explicit 
description can be for example deduced from the local descriptions as in \ref{para:dil-loc} or 
the universal property of $(P^{(R+S)}_X, p_1, p_2)$.
\end{proof}

\begin{defn}\label{defn:FAS}
Let $F\in \RSC_\Nis$.
We define
\[F^\AS(X,R):=\{a\in F(X\setminus R)\mid p_1^*a-p_2^*a\in F(P^{(R)}_X) \}.\]
Here we use the fact that the restriction $F(P^{(R)}_X)\inj F(P^{(R)}_X\setminus R_P)$ is injective by 
\cite[Theorem 3.1(2)]{S-purity}. 
\end{defn}
The group $F^\AS(X,R)$ defined above is a central object of study in this paper. 
The definition is inspired by \cite[Definition 2.12]{TakeshiSaito},
where  for a   Galois-torsor the notion of having  ramification bounded by $R$ is defined using
the schemes $P^{(R)}_X$, as a non-logarithmic variant of \cite[Definition 7.3]{AbbesSaito11}. 
More precisely, if $G$ is a commutative finite \'etale $k$-group-scheme and $F=H^1_{\et}(-, G)$,
 then $F\in\RSC_\Nis$ (see \cite[Theorem 9.12]{RS}) and  with the above definition $F^\AS(X,R)$ coincides
with the group of $G$-torsors over $U=X\setminus R$ with ramification bounded by $R$
in the sense of \cite[Definition 2.12]{TakeshiSaito}. We note that in {\em loc. cit.} $G$ 
does not need to be commutative and $R$ may have rational coefficients.
On the other hand, in view of \cite[9.]{RS}, we can consider  here fppf-tosors under any commutative finite $k$-group scheme, not just the \'etale ones.
The superscript ``AS'' stands for {\em Abbes-(Takeshi) Saito}.

\begin{lem}\label{lem:FAS}
Let $F\in \RSC_\Nis$.
\begin{enumerate}[label=(\arabic*)]
\item\label{lem:FAS1} $F^\AS(X, \emptyset)=F(X)$.
\item\label{lem:FAS2} Let $f:Y\to X$ be a morphism in $\Sm$ with $f(Y)\not\subset R_\red$ 
       and assume there is an effective Cartier divisor  $S$ with $S_\red$ SNCD and $S\ge f^*R$. 
       Then $f^*: F(X\setminus R)\to F(Y\setminus S)$ induces a morphism
\[f^* : F^\AS(X,R)\to F^\AS(Y, S).\]
In particular with $f=\id_X$, we obtain $F^\AS(X,R)\subset F^\AS(X,S)$.
\end{enumerate}
\end{lem}
\begin{proof}
\ref{lem:FAS1} holds by definition.  Note that the assumptions in \ref{lem:FAS2} imply that $(f^*R)_{\red}$ is SNCD.
By the universal property of $(P^{(R)}_X, p_1, p_2)$ the map
$f$ induces a natural morphism  
$f_P: (P^{(f^*R)}_Y, p_1^{(f^*R)}, p_2^{(f^*R)})\to (P^{(R)}_X, p_1^{(R)}, p^{(R)}_2)$ 
and hence a natural map $F^\AS(X, R)\to F^\AS(Y, f^*R)$. 
By Lemma \ref{lem:dil}\ref{lem:dil5} we have  $F^\AS(Y, f^*R)\subset F^\AS(Y, S)$.
This yields \ref{lem:FAS2}.
\end{proof}

\begin{thm}\label{thm;AS}
Let $F\in \RSC_\Nis$. Then
\[\tF(X,R)\subset F^\AS(X,R).\]
In particular, for $U=X\setminus R$ we have $F(U)=\cup_{S} F^\AS(X,S)$, where the union is over 
all effective Cartier divisors $S$ on $X$ with $S_\red\subset D$.
\end{thm}
\begin{proof}
Let $a\in \tF(X,R)$. Note $p_i^*(a)\in F(\PXR\setminus R_P)$ for $i=1,2$.
We will  use Proposition \ref{prop:HLS-reg} to show $p_1^*(a)-p_2^*(a)\in F(\PXR)$.
First we need some preparation. By the purity of $F$ (\cite[Theorem 0.2]{S-purity}), 
the question is local at the generic points of $R$ and hence we may assume the following:
\begin{enumerate}[label=(\alph*)]
\item\label{thm;AS1}
$R= n D$, with $D$ irreducible smooth and $n\ge 1$, and 
\item\label{thm;AS2}
$X=\Spec A$ is affine and \'etale  over $k[z_1,\ldots, z_d]$ with $D=\{z_1=0\}$.
\end{enumerate}
Let $V=\Spec B\subset X\times_k X$ be an open neighborhood of $\Delta_X$ as in \eqref{para:dil-loc2}
 and set $t_i:=z_i\otimes 1$, $s_i:=1\otimes z_i\in B$, $i=1,\ldots, d$, so that
\eq{thm:AS2.5}{I_{\Delta}B=(\theta_1,\dots,\theta_d)B\subset B\qwith \theta_i=s_i-t_i}
and 
\eq{Pexplicit}{P:=(\PXR)_{|V} 
=\Spec \frac{B[\tau_1,\ldots, \tau_d, \frac{1}{1+t_1^{n-1} \tau_1}]}{(\theta_i-\tau_i t_1^n,\, i=1,\ldots, d)}.}
We have $R_P\subset P$ and 
\eq{DPexplicit}{R_P= \Spec (A/z_1^n)[\tau_1,\dots,\tau_d,\tfrac{1}{1+z_1^{n-1}\tau_1}].}
For $i=1,2$, we denote by  $p_i: P\to X$ also the restriction of $p_i: \PXR \to X$.
Set  
\[Y:=\begin{cases}
\Spec A[\tau_1, \ldots, \tau_d, \frac{1}{1+\tau_1}] & \text{if }n=1,\\
\Spec A[\tau_1, \ldots, \tau_d] &\text{if } n\ge 2.
\end{cases}\]
and denote by
\[ \pi_i: P\to Y\]
the unique morphism which composed with the projection $\rho:Y\to X$ is equal to $p_i$ 
and satisfies $\pi_i^*(\tau_j)= \tau_j$, for $j=1,\dots,d$. 
Note that the maps $\pi_1$ and $\pi_2$ induce the same isomorphism
\eq{DDQ}{(\pi_1)_{|R_P}=(\pi_2)_{|R_P} : R_P \xr{\simeq} R_Y:=\rho^*R=\begin{cases}
                  \Spec (A/z_1)[\tau_1,\ldots, \tau_d,\frac{1}{1+\tau_1}] & \text{if }n=1,\\
                   \Spec (A/z_1^n)[\tau_1,\ldots, \tau_d] & \text{if } n\ge 2.
                   \end{cases}}
Furthermore, $\pi_2$ is \'etale on $P\setminus V(1+n\tau_1t_1^{n-1})$ and $\pi_1$ is \'etale,
in particular both are \'etale in a neighborhood of $R_P$.
Indeed, by a similar argument as in the proof of \ref{lem:dil3} in Lemma  \ref{lem:dil}, it
suffices to show this in the case $A=k[z_1,\ldots, z_d]$, where it follows directly from a
differential  criterion for \'etaleness. 
Let $K$ be  a function field over $k$ and denote by $Y_K=Y\otimes_k K$ and $P_K=P\otimes_k K$ 
the base changes. We obtain  commutative diagrams, $i=1,2$,
\eq{thm:AS3}{\xymatrix{
P_K\ar[d]_{\pi_{i,K}}\ar[r]^{\phi_P} & P\ar[d]_{\pi_i}\ar[dr]^{p_i}\\
Y_K\ar[r]^{\phi_Y} & Y\ar[r]^{\rho} & X,
}}
in which $\phi_Y$ is the projection and the square is cartesian.
Let $\eta\in (R_{P,K})_\red^{(0)}$ be a generic point. 
By the above  we have $\pi_1(\eta)=\pi_2(\eta)=:\xi\in (R_{Y,K})_\red^{(0)}$ and the maps
\eq{piidvr}{\pi_{i,K}^*: \sO_{Y_K, \xi}^h\to \sO_{P_K,\eta}^h, \quad i=1,2,}
are isomorphisms of dvr's so that the induced local norm maps
\eq{thm:AS4}{\Nm_{\eta/\xi}^i: K_{2d}^M(\sO_{P_K,\eta}^h)\to K_{2d}^M(\sO_{Y_K,\xi}^h)}
are isomorphisms. Let $\ux=(x_0,\ldots,  x_{2d})\in \mc(P_K)$ with $x_{2d-1}=\eta$.
Then $\pi_1(\ux)=\pi_2(\ux):=\uy=(y_0,\ldots, y_{2d})\in \mc(Y_K)$ with $y_{2d-1}=\xi$.
Let $\beta\in K^M_{2d}(\sO^h_{P_K, \eta})$. We have 
\begin{align*}
(p_i^*(a)_K, \beta)_{P_K/K,\ux}& = (\phi_P^*p_i^* a, \beta)_{P_K/K,\ux}\\
                                                &= (\pi_{i,K}^*\phi_Y^*\rho^*a,\beta)_{P_K/K,\ux} & & \text{by } \eqref{thm:AS3}\\
                                                &= (\phi_Y^*\rho^*a, \Nm_{\eta/\xi}^i(\beta))_{Y_K/K, \uy} & & 
                                                                         \text{by  \ref{HS5}  and  \eqref{DDQ}.}                                                 
\end{align*}
Set
\[ \omega= \Nm_{\eta/\xi}^1(\beta) - \Nm_{\eta/\xi}^2(\beta)\in K_{2d}^M(\sO_{Y_K,\xi}^h).\]
\begin{claim}\label{claim;normgamma}
$\omega\in (V_{2d, Y_K|R_{Y,K}})^h_\xi$, see \ref{HS3} for notation.
\end{claim}
Since $\phi_Y^*\rho^*a\in \tF(Y_{K}, R_{Y,K})$, the claim together with \ref{HS3} implies
\[((p_1^*(a)-p_2^*(a))_K, \beta)_{P_K/K,\ux}= (\phi_Y^*\rho^*a, \omega)_{Y_K/K,\uy}=0.\]
Since $K/k$, $\ux=(x_0,\ldots, x_{2d})\in \mc(P_K)$ with $x_{2d-1}=\eta \in (R_{P,K})_\red^0$, and 
$\beta\in K^M_{2d}(\sO_{P_K,\eta})$ were taken arbitrarily, Proposition \ref{prop:HLS-reg} implies
$a\in F(P)$ and hence also $a\in F(P^{(R)}_X)$. Thus it remains to prove the claim.

{\em Proof of Claim \ref{claim;normgamma}.}
Since the norm \eqref{thm:AS4} is an isomorphism with inverse induced by $\pi_{i,K}^*$,
we find for $\beta\in K_{2d}^M(\sO_{P_K,\eta}^h)$
\begin{multline*}
 \pi_{1,K}^*(\Nm_{\eta/\xi}^1(\beta) -  \Nm_{\eta/\xi}^2(\beta) )=
\beta -   \pi_{1,K}^*\Nm_{\eta/\xi}^2(\beta) \\
= \pi_{2,K}^*\Nm_{\eta/\xi}^2(\beta) -   \pi_{1,K}^* \Nm_{\eta/\xi}^2(\beta)\in 
K^M_{2d}( \sO_{Y_{K},\xi}^h).
\end{multline*}
where $\pi_{i,K}^*$ are induced by \eqref{piidvr}.
Thus it suffices to show
\[  \pi_{2,K}^*(\alpha) - \pi_{1,K}^*(\alpha) \in (V_{2d, P_K|R_{P,K}})_{\eta}^h,
\quad \text{for all }\alpha\in  K^M_{2d}( \sO_{Y_K,\xi}^h).\]
This follows from the following claim: 
\eq{claim;pi'}{
\pi_{2,K}^*(u)/\pi_{1,K}^*(u) \in 1+ t_1^n \sO_{P_K,\eta}^h,
\quad \text{for all } u\in (\sO_{Y_K,\xi}^h)^\times.}
Write $t=t_1$, $z=z_1$ and $\tau=\tau_1$ for simplicity. 
Let $E$ be the residue field of $\sO_{Y_K,\xi}^h$ identified with that of $\sO_{P_K,\eta}^h$ 
via either of $\pi_{i,K}^*$, see \eqref{DDQ}.
Choose a coefficient field $\sigma: E \to \sO_{Y_K,\xi}^h$. 
We obtain two induced coefficient fields
\[\sigma_i=\pi_{i,K}^*\circ \sigma:  E \to \sO_{P_K,\eta}^h, \quad i=1,2.\]
By \eqref{thm:AS2.5}  and \eqref{Pexplicit} we have 
\eq{coefficient}{
\sigma_2(c) -\sigma_1(c) \in I_\Delta\sO_{P_K, \eta}^h\subset   t^n\sO_{P_K,\eta}^h, \quad \text{for all } c\in E.}
These coefficient fields induce isomorphisms of dvr's:
\[ \sigma: E[[z]] \isom \widehat{\sO_{Y_K,\xi} },\quad
\sum_{m\geq 0} a_m z^m \to \sum_{m\geq 0} \sigma(a_m) z^m,\]
\[\sigma_i :E[[t]]\isom \widehat{\sO_{P_K,\eta}},\quad 
\sum_{m\geq 0} a_m t^m \to \sum_{m\geq 0} \sigma_i(a_m) t^m,\]
where $\widehat{\sO_{Y_K,\xi} }$ denotes the $z$-adic completion of $\sO_{Y_K,\xi}$
and $\widehat{\sO_{P_K,\eta}}$ denotes the $t$-adic completion of $\sO_{P_K,\eta}$
(which is the same as the $s=(t+t^n\tau)$-adic completion.)
By definition we have commutative diagrams, for $i=1,2$
\[\xymatrix{
E[[z]]  \ar[r]^{\nu_i} \ar[d]^{\sigma}&  E[[t]]\ar[d]^{\sigma_i}\\
\widehat{\sO_{Y_K,\xi}} \ar[r]^{\pi_{i,K}^*} & \widehat{\sO_{P_K,\eta}} \\ 
}\]
where 
\[\nu_1(\sum_{m\geq 0} a_m z^m) = \sum_{m\geq 0} a_m t^m\quad \text{and}\quad
\nu_2(\sum_{m\geq 0} a_m z^m) = \sum_{m\geq 0} a_m (t+t^n\tau)^m.\]
Together with \eqref{coefficient} this yields for  $u\in E[[z]]^\times$
\[\frac{\pi_{2,K}^*(\sigma(u))}{\pi_{1,K}^*(\sigma(u))} =
\frac{\sigma_2(\nu_2(u))}{\sigma_1(\nu_1(u))} \equiv 1 \; \text{mod } t^nE[[t]].
\]
This proves \eqref{claim;pi'} and completes the proof of Claim \ref{claim;normgamma} and the theorem.
\end{proof}

\begin{para}\label{para:cond}
Let $L$ be a henselian discrete valuation ring of geometric type over $k$, 
i.e., there is an isomorphism of $k$-algebras $L\cong \Frac(\sO_{U,u}^h)$, 
for some $U\in \Sm$ and $u\in U^{(1)}$.
Denote by $\sO_L$ its valuation ring and set 
$S_L=\Spec \sO_L$ and denote by $s_L$ the closed point of $S_L$. Following \cite{RS}
we define the motivic conductor $c^F_L:F(L)\to \N_0$ of $F$ on $L$ by
\[c_L^F(a)=\min\{n\in \N_0\mid a\in\tF(S_L, n\cdot s_L)\}.\]
For $Y$ a $k$-scheme, $y\in Y^{(1)}\cap Y_{\rm sm}$  a 1-codimensional point in the smooth locus of $Y$, 
and $a\in F(\Frac(\sO_{Y,y}))$ we set 
\[c_{Y,y}^F(a):= c^F_{K^h_{Y,y}}(\rho^*a),\]
where $\rho: \Spec K^h_{Y,y}\to \Spec \Frac(\sO_{Y,y})$ is the natural map.
\end{para}

The following theorem  is a version of the {\em Brylinski-Kato formula} for reciprocity sheaves, see Remark \ref{rmk:BK} 
below for  more detailed references to the literature.

\begin{thm}\label{thm:BK}
Let $X$ be a smooth separated $k$-scheme and $D\subset X$ an integral divisor with generic point $\eta$.
Let $x\in D_{\rm sm}$ be a smooth point of $D$ and let $Z,Z'\subset X$  be integral closed subschemes such that 
\eq{thm:BK1}{(Z\cap D)_x= x= (Z'\cap D)_x,}
where we use the notation $Y_x=\Spec \sO_{Y,x}$.
Let $V\subset X$ be some open containing the generic points of $Z$ and $Z'$.
Then $Z$ and $Z'$ are smooth around $x$ and  for $a\in F(V)$ the following implication holds
\eq{thm:BK2}{\text{length}_{\sO_{X,x}}(\sO_{Z\cap Z',x})\ge c^F_{X,\eta}(a)\Longrightarrow
c^F_{Z,x}(a_{Z})=c^F_{Z',x}(a_{Z'}),}
where $a_Z$ and $a_{Z'}$ denote the pullback of $a$ to $Z\cap V$ and $Z'\cap V$, respectively.
\end{thm}
\begin{proof}
If $Z\subset D$ or $Z'\subset D$, then condition \eqref{thm:BK1} implies that
$x$ is the generic point of $Z$ or $Z'$. Hence $x\in V$ and $a$ is regular at $x$.
Therefore the implication trivially holds in this case.

We consider the case $Z, Z'\not\subset D$. By \eqref{thm:BK1} $Z\cap D$ is a regular Cartier divisor around $x$ on $Z$, hence $Z$ is smooth around $x$,
similarly with $Z'$. We may assume $Z\neq Z'$ and $V$ does not contain $x$. 
We may shrink $X$ Zariski locally around $x$  and assume $X$, $D$, $Z$, $Z'$, and $T:=\ol{\{x\}}$ are smooth and integral,  $D=\Div(f)$, and 
\[Z\cap D=T=Z'\cap D.\]
Set $n=c^F_{X,\eta}(a)$. By \cite[Corollary 8.6(2)]{S-purity} 
\eq{thm:BK3}{a\in \tF(X,nD).}
The statement in \eqref{thm:BK2} is Nisnevich local around $x$.
By \cite[Lemma 7.13]{BRS} we may assume that there exists an 
affine $k$-morphism $X\to S=\Spec R$, such that the composition $T\inj X\to S$ is an isomorphism.
We obtain a natural morphism $X\to \A^1_S$, induced by $R[t]\to \sO_X(X)$, $t\mapsto f$, 
which composed with the closed immersion $Z\inj X$ yields a   map,
\eq{thm:BK4}{Z\to \A^1_S \quad \text{inducing an isomorphism } T=D\cap Z\xr{\simeq} 0_S,}
where $0_S$ denotes the zero-section of $\A^1_S$. Note that \eqref{thm:BK4}  is \'etale in a neighborhood of $T$.
The same reasoning applies to  $Z'$. 
We obtain isomorphisms
\eq{thm:BK5}{Z\times_X nD\xr{\simeq} \Spec R[t]/(t^n)\xleftarrow{\simeq} Z'\times_X nD.}
By \eqref{thm:BK1} and $Z'\neq Z$ we find  $x\in (Z\cap Z')^{(0)}$. 
The condition $\text{length}_{\sO_{X,x}}(\sO_{Z\cap Z',x})\ge n$ therefore implies the following equality of closed subschemes of $X$
\eq{thm:BK6}{Z\times_X nD=Z\times_X Z'\times_X nD=Z'\times_X nD.}
We consider the compositions
\[q: Z\times_{\A^1_S} Z'\to Z\inj X, \quad q':Z\times_{\A^1_S}Z'\to Z'\inj X.\]
By \eqref{thm:BK5} we have $q^{-1}(nD)= {q'}^{-1}(nD)$ and by \eqref{thm:BK6}
\[q_{|Z\cap Z'\cap (nD)}=q'_{|Z\cap Z'\cap (nD)}: Z\cap Z'\cap (nD)=q^{-1}(nD)= {q'}^{-1}(nD)\to nD.\]
Thus by the universal property of $P^{(nD)}_X$ (see \ref{para:dil}) we obtain a 
map $f:Z\times_{\A^1_S} Z'\to P^{(nD)}_X$ making the following diagram commutative
\eq{thm:BK7}{\xymatrix{
(Z')^h_x\ar[r]& Z'\ar[r] & X\\
Z^h_x\times_{(\A^1_S)_{x_0}^h} (Z')^h_x\ar[u]^{\simeq}_{\sigma'}\ar[d]_{\simeq}^{\sigma}\ar[r]^-\nu & 
                                                        Z\times_{\A^1_S}Z'\ar[u]\ar[d]\ar[r]^-f &     P^{(nD)}_X\ar[u]_{p_2}\ar[d]^{p_1}\\
Z^h_x\ar[r]\ar@/^4pc/[uu]^{\tau}  & Z\ar[r] & X.
}}
Here $x_0$ denotes the generic point of $0_S\subset \A^1_S$,  the horizontal maps are the natural ones, the vertical maps are induced by projections, and 
$\tau=\sigma'\circ {\sigma}^{-1}$. Note that $\sigma$ and $\sigma'$ are isomorphisms by \eqref{thm:BK4}.
We find for $r\ge 0$
\[a_{(Z')^h_x}\in \tF((Z')_x^h, r\cdot x)\Longleftrightarrow \tau^*(a_{(Z')^h_x}) \in \tF(Z^h_x, r\cdot x),\]
where $a_{(Z')^h_x}$ denotes the pullback of $a\in F(V)$ to $F((Z')^h_x\times_X V)=F((Z')^h_x\setminus\{x\})$.
Diagram \eqref{thm:BK7} restricted to $V$ (on both $X$) yields
\[\tau^*(a_{(Z')^h_x})= (\sigma^{-1})^*\nu^*f^*(p_2^*a)=
(\sigma^{-1})^*\nu^*f^*(p_1^*a+ (p_2^*(a)-p_1^*(a))).\]
By \eqref{thm:BK3} and Theorem \ref{thm;AS} we find 
\[\tau^*(a_{(Z')^h_x})\equiv (\sigma^{-1})^*\nu^*f^*(p_1^*a)=a_{Z^h_x}\quad \text{mod } F(Z^h_x).\]
Altogether
\[a_{(Z')^h_x}\in \tF((Z')_x^h, r\cdot x)\Longleftrightarrow a_{Z^h_x}\in \tF(Z^h_x, r\cdot x).\]
This yields the statement.
\end{proof}

\begin{rmk}\label{rmk:BK}
\begin{enumerate}
\item For $F=H^1_{\et}(-, \Q/\Z)$ and $X$ a smooth surface over a finite field   Brylinski proved a version 
of Theorem \ref{thm:BK} for the Artin conductor, see \cite[Proposition 4]{Brylinski}.
In fact since the motivic  conductor of $H^1_{\et}(-, \Q/\Z)$ is equal to the Artin conductor, by \cite[Theorem 8.8]{RS},
the above theorem reproves and largely generalizes Brylinsky's result. 
On the other hand Kato defined in  \cite{KatoSwan} a Swan conductor for $H^n=R^n\e_*\Q/\Z(n-1)$, $n\ge 1$, 
$\e: X_{\et}\to X_{\Nis}$,  and generalized Brylinsky's result to all $H^n$ and all  excellent regular  schemes using 
his Swan conductor. Since $F=H^n\in \RSC_\Nis$, Theorem \ref{thm:BK} is a version of the Brylinsky-Kato formula
for the motivic conductor of $H^n$. Note that  it is not expected that Kato's Swan conductor coincides with 
the motivic conductor for $H^n$, we rather expect $c^{H^n}_L(a)\ge {\rm Swan}_L(a)\ge c^{H^n}_L(a)-1$,
for $a\in H^n(L)$ and $L$ a henselian discrete valuation field over $k$. 
(For $n=1$ this holds by \cite[8.]{RS}.) 
\item A basic version of Theorem \ref{thm:BK}, namely for $X=\Spec K\{x,t\}$, 
with $K$ a function field over $k$, $D=\Div(t)$, $Z=V(x-t)$, and $Z'=(x-t-at^n)$, $a\in K$, 
was proven  in \cite[Corollary 3.3, Remark 3.4]{RS-ZNP} and was an essential ingredient in the construction
of the higher local symbol in \cite{RS-HLS}, in particular to obtain the properties \ref{HS3} and \ref{HS5}, 
which were essential in the proof of Theorem \ref{thm;AS}.
\end{enumerate}
\end{rmk}

\begin{thm}\label{thm:ASII}
Let $(X,R)$ be as at the beginning of this section.
Assume $(X,R)$ has a projective SNC-compactification (see \ref{para:sm-cptf}).
Then 
\[\tF(X,R)=\FAS(X,R).\]
\end{thm}

The proof will be given after the proof of Theorem \ref{thm:psi} in the next section.

\begin{cor}\label{cor:ASII}
Assume $F$ has level $n\le \infty$ (in the sense of \cite[Definition 1.3]{RS-ZNP}) and 
any pair $(Z,E)$ with $Z\in \Sm$ affine, $\dim (Z)\le n$, and $E$ an effective Cartier divisor with SNC support,
has a projective SNC-compactification (see \ref{para:sm-cptf}). 
Then
\eq{cor:ASII1}{\tF(X,R)=\FAS(X,R).}
\end{cor}
\begin{proof}
Let $U=X\setminus D$. For $G\in \{\tF, \FAS\}$ define 
\eq{cor:ASII2}{G^{\le n}(X,R)=\left\{a\in F(U)\,\bigg\vert\,f^*a\in G(Z, f^*R),\, 
                   \begin{minipage}{4.1cm}$\forall f:Z\to X$  \\ with $Z\in \Sm$  affine, \\ 
                                                  $\dim Z\le n$,  $f(Z)\not\subset D$, \\ and $(f^*D)_{\red}\in\Sm $   \end{minipage}\right\}.}
We have 
\[\tF(X,R)\stackrel{1}{\subset} \FAS(X,R)\stackrel{2}{\subset} (\FAS)^{\le n}(X,R)
\stackrel{3}{=} (\tF)^{\le n}(X,R)\stackrel{4}{=}\tF(X,R),\]
where $1$ holds by Theorem \ref{thm;AS}, $2$ by functoriality (see Lemma \ref{lem:FAS}),
$3$ by Theorem \ref{thm:ASII}, and $4$ by \cite[Corollary 4.18]{RS}.
\end{proof}

\begin{exs}\label{ex:ASII}
Since the assumption in Corollary \ref{cor:ASII} on the existence of projective SNC-compactifications
in dimension $\le n$ is satisfied for $n\le 3$ (or for all $n$ if ${\rm char}(k)=0$) we get the 
equality \eqref{cor:ASII1} unconditionally, for example, in the following cases: 
if ${\rm char}(k)=0$, for all $F\in\RSC_\Nis$;
if ${\rm char}(k)=p>0$, for 
\begin{enumerate}[label=(\arabic*)]
\item\label{ex:ASII1} $F=G$ a smooth commutative $k$-group (they have level 1 by \cite[Theorem 5.2]{RS}),
\item\label{ex:ASII2} $F=H^1_{\et}(-,\Q/\Z)$ or ${\rm Lisse}^1_\ell$  
(they have level 1 by \cite[Theorem 8.8, Corollary 8.10]{RS}),
\item $F=H^1_{\rm fppf}(-, G)$, with $G$ a finite $k$-group scheme 
($H^1_{\rm fppf}(-, G)$ has level $\le 2$ by \cite[Theorem 9.12]{RS}),
\item $\Omega^j$ for $j\le 2$ ($\Omega^j$ has level $j+1$ by Corollary \ref{cor:tOmega}).
\end{enumerate}
By \cite{RS-ZNP}
The equality \eqref{cor:ASII1} for $F=H^1_{\et}(-,\Q/\Z)$ was known before: 
Indeed since the motivic conductor of $F$ in this case is the Artin conductor defined by Matsuda's filtration \cite{Matsuda} 
(this holds by \cite[Theorem 8.8]{RS}), it follows from \cite[Corollary 2.8]{Kerz-Saito}, that 
\[\tF(X,R)=\{a\in F(U)\mid {\rm Art}_{X,\eta}(a)\le m_\eta,\,\forall\eta\in D^{(0)} \},\]
where we write $R=\sum_{\eta\in D^{(0)}} m_\eta \bar{\eta}$.
Now the expression on the right equals $\FAS(X,R)$ by \cite[Theorem 0.1]{Yatagawa}, 
\cite[Proposition 2.27]{TakeshiSaito}, and  Zariski-Nagata purity (cf. \cite[Lemma 2.3]{Yatagawa}).
Note that the proof of the equality in this way uses in an essential way the theory of higher ramification groups 
defined by Abbes-Saito. As far as we know the other examples mentioned above are new.
\end{exs}

\section{Abbes-Saito formula (continued)}\label{sec:AS-ctd}
In this section $X\in \Sm$, $D$ is a smooth divisor on $X$, and $n\ge 2$.
We fix $F\in \RSC$.

\begin{para}\label{para:psi}
Let $p_{1,n},p_{2,n}:P^{(nD)}_X\rightrightarrows X$ be as in \ref{para:dil}, i.e., $p_{i,n}$ is the composition of the natural map $P^{(nD)}_X\to X\times X$ 
with the projection to the $i$th factor. Set (see Lemma \ref{lem:dil}\ref{lem:dil1})
\eq{para:psi1}{V_n:=P^{(nD)}_X\times_{X\times X} D:= R_P\times_R D=\mathbb{V}(\Omega^1_{X/k}(nD)_{|D})}
and denote by $\iota_n: V_n\inj P^{(nD)}_X$ the closed immersion.
By definition of $\FAS$ (see Definition \ref{defn:FAS})  and the injectivity of the restriction $F(P_X^{(nD)})\inj  F(U\times U)$, where $U=X\setminus D$,
see \cite[Theorem 3.1(2)]{S-purity}, we have a unique morphism
\[\varphi_n:\FAS(X,nD)\lra F(P^{(nD)}_{X}),\]
such that $\varphi_n(a)_{|U\times U}= p_{2}^*(a_{|U})-p_1^*(a_{|U})\in F(U\times U)$.
We define the map $\psi_n$ as the composition
\eq{para:psi2}{\psi_n:=\iota_n^*\circ \varphi_n: \FAS(X,nD) \lra F(P^{(nD)}_X)\lra F(V_n).}
\end{para}

\begin{lem}\label{lem:ker-psi}
$\FAS(X,(n-1)D)\subset \Ker\psi_n$.
\end{lem}
\begin{proof}
We have a commutative diagram $(i=1,2)$
\[\xymatrix{
V_n\ar[r]^{\iota_{n}}\ar[d]_q\ar@/_2pc/[dd]_{\pi_V} &P^{(nD)}_X\ar[dd]^{\pi}\ar[dr]^{p_{i,n}}& \\
D\ar[dr]^-{\iota_D}\ar[d]_{s_0} & & X,\\
V_{n-1}\ar[r]^-{\iota_{n-1}}& P^{((n-1)D)}_X\ar[ur]_{p_{i,{n-1}}}& 
}\]
where $\iota_D = \Delta_{n-1}\circ i$ is the composition of the closed immersion $i:D\inj X$ followed by the lift of the diagonal 
$\Delta_{n-1}:X\inj P^{((n-1)D)}_X$ (see Lemma \ref{lem:dil}\ref{lem:dil3.5}), $\pi$ is the natural morphism obtained from interpreting 
$P^{(nD)}_X$ as the blow-up of $P^{((n-1)D)}_X$ in $\iota_D(D)$ with the strict transforms of $p_1^*D$ and $p_2^*D$ removed 
(see Lemma \ref{lem:dil}\ref{lem:dil5}), $\pi_V$ is the base change of $\pi$ along $\iota_{n-1}$, $q$ is the structure map of the vector bundle $V_n$, and 
$s_0$ is the zero section of the vector bundle $V_{n-1}$. Hence for $a\in \FAS(X,(n-1)D)$ we have 
$\psi_n(a)=\pi_V^*\psi_{n-1}(a)=q^*i^*\Delta_{n-1}^*\varphi_{n-1}(a)$.
Since 
\[(\Delta_{n-1}^*\varphi_{n-1}(a))_{|U}= \Delta_U^* (p_2^*(a_{|U})-p_1^*(a_{|U}))=0\quad \text{in } F(U),\]
with $\Delta_U:U\to U\times U$ the diagonal, the injectivity of the restriction $F(X)\to F(U)$ (see \cite[Theorem 3.1(2)]{S-purity}) yields
$\Delta_{n-1}^*\varphi_{n-1}(a)=0$ in $F(X)$, hence $\psi_n(a)=0$.
\end{proof}

\begin{thm}\label{thm:psi}
We keep the assumptions from the beginning of this section.
Assume $(X,D)$ admits a  projective SNC-compactification (see \ref{para:sm-cptf}). Then
the map \eqref{para:psi2} induces an injective morphism
\eq{thm:psi1}{\bar{\psi}_n:\frac{\tF(X,nD)}{\tF(X,(n-1)D)} \inj F(V_n).}
\end{thm}
\begin{proof}
We can assume that $X$ and $D$ are integral and $\dim X=d\ge 1$.
Note that by Theorem \ref{thm;AS} and Lemma \ref{lem:ker-psi}, the map $\psi_n$ 
induces a well-defined map as in the statement. Let $a\in \tF(X,nD)\cap \Ker(\psi_n)$. 
It remains to show $a\in \tF(X,(n-1)D)$, which by Theorem \ref{thm:HLS-mod}  is equivalent to show the 
following:
\begin{claim}\label{thm:psi-claim} There exists an open neighborhood  $W\subset X$ of the generic point of $D$, such that
for all function fields $K/k$ and all $\ux=(x_0,\ldots, x_{d-1},x_d)\in\mc(W_K)$ with $x_{d-1}\in D_K^{(0)}$ we have 
\eq{thm:psi2}{(a_K, \beta)_{X_K/K,\ux}=0,\quad \text{for all }\beta\in (V_{d, X_K|(n-1)D_K})_{x_{d-1}},}
where $a_K\in \tF(X_K, nD_K)\cap \Ker(\psi_{n,K})$ denotes the pullback of $a$ along $X_K\to X$.
\end{claim}
We prove the claim. Up to shrinking $X$ around the generic point of $D$, we may assume that $X=\Spec A$ is \'etale over
$k[z_1,\ldots, z_d]$ and $D=\{z_1=0\}$. 
Choose $\alpha_1,\ldots, \alpha_r\in A$ such that their residue classes $\bar{\alpha}_{i}\in A/(z_1)$ are not zero  and  the monomials  in the $\bar{\alpha}_i$
generate $A/(z_1)$ as a $k$-vector space. 
Set 
\[W:= X\setminus \Div_X(\alpha_1\cdots \alpha_r)\cup\Div_X(z_2\cdots z_n)\]
which is an open neighborhood of the generic point of $D$. 

Let $K$ be a function field over $k$ and $\ux=(x_0,\ldots, x_{d-1}, x_d)\in\mc(W_K)$ such that
$\xi:=x_{d-1}\in D^{(0)}_K$. By Theorem \ref{thm:HLS-mod} it suffices to show \eqref{thm:psi2} for $\beta$ running through a class of representatives of 
$(V_{d, X_K|(n-1)D_K})_\xi/(V_{d, X_K|nD_K})_\xi$.
Set 
\[U_{d, X_K|nD_K}:= \Im( V_{1, X_K|nD_K}\otimes j_*K^M_{d-1, U_K}\to  j_*K^M_{d, U_K}),\]
where $j:U_K=X_K\setminus D_K \inj X_K$ is the open immersion.
By, e.g., \cite[Lemma 2.3]{RS-cycle-map}, we have  surjective maps
\[ \Omega^{d-1}_{K(\xi)/\Z}\lra  \frac{(V_{d, X_K|(n-1)D_K})_\xi}{(U_{d, X_K|nD_K})_\xi},\quad 
 b\,\frac{d u_1}{u_1}\wedge\ldots\wedge \frac{d u_{d-1}}{u_{d-1}}\mapsto \{1+z_1^{n-1}\tilde{b}, \tilde{u_1},\ldots, \tilde{u}_{d-1}\},\]
 and 
 \[ \Omega^{d-2}_{K(\xi)/\Z}\lra  \frac{(U_{d, X_K|nD_K})_\xi}{(V_{d, X_K|nD_K})_\xi},\quad 
 b\,\frac{d u_1}{u_1}\wedge\ldots\wedge \frac{d u_{d-2}}{u_{d-2}}\mapsto \{1+z_1^n\tilde{b}, z_1,\tilde{u_1},\ldots, \tilde{u}_{d-2}\},\]
where $b\in K(\xi)$, $u_i\in K(\xi)^\times$ and $\tilde{b}, \tilde{u}_i$ are lifts to $\sO_{X_K,\xi}$. 
Hence it suffices to show \eqref{thm:psi2} for 
\eq{thm:psi3}{\beta=\{1+z_1^{n-1} \frac{m}{g}, z_{i_2}, \ldots, z_{i_c}, \sigma_{1},\ldots, \sigma_{d-c}\},}
or 
\eq{thm:psi3.1}{\beta=\{1+z_1^{n} \frac{m}{g}, z_1,  z_{i_3}, \ldots, z_{i_c}, \sigma_{1},\ldots, \sigma_{d-c}\},}
where $g\in \sO_{X,\xi}^\times\cap A$, $c\in \{1,\ldots, d\}$, $\sigma_1,\ldots, \sigma_{d-c}$ is part of a differential basis of $K$ 
over its prime field, $2\le i_2<\ldots<i_c\le d$, and  
\eq{thm:psi3.2}{m=\lambda \alpha_1^{n_1}\cdots \alpha_r^{n_r}\in \Gamma(W_K, \sO_{W_K}^\times),\quad \text{for some }\lambda\in K,\, (n_1,\ldots, n_r)\in\N_0^r.}
Let $V=\Spec B\subset X\times X$ be as in \eqref{para:dil-loc2} and  (with the notation from there) set
\[P:=(P_X^{(nD)})_{|V}:=\Spec \frac{B[\tau_1,\ldots, \tau_d,\frac{1}{1+t_1^{n-1}\tau_1}]}{(\theta_i-\tau_it_1^n, i=1,\ldots, d)}.\]
Let $p_i:P\to X$ be induced by the projection to the $i$th factor.
These extend to quasi-projective morphisms  
\[\bar{p}_i:\ol{P}:=\Proj\left(\frac{B[T_0,T_1,\ldots, T_d]}{(T_0\theta_i-T_it_1^n, i=1,\ldots, d)}\right)\to X.\]
Let $V_n\subset P$ be as in \eqref{para:psi1} be given by $\{t_1=0\}$ and denote by $\ol{V}_n\subset \ol{P}$ its closure,
which is defined by the ideal $(t_1, \theta_1,\ldots, \theta_d)=(s_1, \theta_1,\ldots, \theta_d)$.
Therefore $\bar{p}_1$ and $\bar{p}_2$ induce the same  isomorphism of $D$-schemes
\[\bar{\pi}_V:\ol{V}_n\xr{\simeq} \Proj A/(z_1)[T_0,\ldots, T_d].\]
Fix a tuple $\ul{g}=(g_1,\ldots, g_d, m_1,\ldots, m_d)$, where $g_i\in \sO_{X,\xi}^\times\cap A$  and the $m_i$ 
are monomial as in \eqref{thm:psi3.2}, and set
\eq{thm:psi3.3}{Z:=Z_{\ul{g}}:= \bigcap_{1\le i\le d}V_+\Big(\bar{p}_1^*(g_i) T_i- \bar{p}_1^*(m_i) T_0\Big)
\subset \bar{p}_1^{-1}(W_K)\subset \ol{P}_K.}
Here  we also denote by $\bar{p}_i: \ol{P}_K\to X_K$ the base change of $\bar{p}_i$ along $X_K\to X$. 
Denote by $\ol{Z}\subset \ol{P}_K$ the closure of $Z$  and let
\[\sigma: \ol{Z}\inj \ol{P}_K \quad \text{and}\quad p_{Z,i}:=\bar{p}_i\circ\sigma: \ol{Z}\to X_K,\]
where $\sigma$ is the closed immersion. We observe that  $\bar{\pi}_V$ induces a projective morphism
\eq{thm:psi4}{\bar{\pi}_{V,Z}: \Theta:= \ol{Z}\cap \ol{V}_{n,K} \to D_K}
which restricts to an isomorphism
\[ Z\cap \ol{V}_{n,K}\cap p_1^{-1}(W_K\setminus \Div_{X_K}(g_1\cdots g_d)) 
\xr{\simeq}  D_K\cap (W_K\setminus \Div_{X_K}(g_1\cdots g_d)).\]
In particular, $\bar{\pi}_{V,Z}$ is projective and birational, and hence is surjective.
The maps  $p_{Z,1}$ and $p_{Z,2}$ are \'etale  around the generic points of $\Theta$,
since the associated  ring maps $p_{Z,i}^*$ look  \'etale locally like 
\[K[z_1,\dots,z_d] \to B_1/(s_i-t_i-u_i t_1^n,i=1,\dots,d), \quad  \text{for }u_i\in B_1^\times,\]
with $B_1$ an \'etale $K[s_1,\dots,s_d,t_1,\dots,t_d]$-algebra, and where $p_{Z,i}^*$ are $K$-algebra maps such that $p_{Z,1}^*(z_i)=t_i$ and $p_{Z,2}^*(z_i)=s_i$.
Denote by $\xi_Z$ the generic point of $\Theta$ such that $\bar{\pi}_{V,Z}(\xi_Z)=\xi$.
By the above we obtain isomorphisms
\eq{thm:psi5}{p_{Z,i}^*: \sO_{X_K,\xi}^h\xr{\simeq} \sO_{Z,\xi_Z}^h, \quad i=1,2.}

Let $\uy_{Z,\lambda} = (y_{\lambda,0}, \ldots y_{\lambda, d-2}, \xi_Z)\in \mc(\Theta)$, $\lambda\in \Lambda$, be 
all the lifts of $(x_0,\ldots, x_{d-1}=\xi)\in \mc(W_K\cap D_K)$ under  \eqref{thm:psi4}.
Let $y_{2d}\in \ol{P}_K$ be the generic point of the  connected component which contains $\xi_Z$.
For $i=d,\ldots, 2d-1$ denote by 
\[y_i\in 
\bigcap_{i\le j\le 2d-1}V_+\bigg(\ol{p}_1^*(g_{2d-j})T_{2d-j}-\ol{p}_1^*(m_{2d-j})T_0\bigg)\cap \ol{\{y_{2d}\}}\]
the generic point. We obtain  maximal chains
\[\uy_\lambda:=(y_{\lambda,0},\ldots, y_{\lambda, d-2},\xi_{Z}, y_d, \ldots, y_{2d})\in \mc(\ol{P}_K), 
\quad \lambda\in \Lambda.\]
Note that $y_d\in \ol{Z}^{(0)}$ and $p_{Z,i}(y_d)=x_d$, $i=1,2$.
Let $a_K$ be as in Claim \ref{thm:psi-claim} and let $\gamma\in K^M_{2d}(K(P_K))$ be arbitrary.
Since $\bar{p}_i^*a_K\in F_{\ol{P}_K, y_d}$ we obtain
\begin{align}
\sum_{\lambda\in \Lambda} (\bar{p}_i^*a_K, \gamma)_{\ol{P}_K/K, \uy_\lambda} &
   = \sum_{\lambda\in \Lambda} 
   (p_{Z,i}^*a_K, \partial_Z\gamma)_{\ol{Z}/K, (y_{\lambda, 0},\ldots, y_{\lambda, d-2}, \xi_Z, y_d)}, & & 
\text{by  \ref{HS2}} \label{thm:psi6},\\
                                                        &=(a_K, \Nm^i_{\xi_Z/\xi}(\partial_Z\gamma))_{X_K/K, \ux}, & & \text{by \ref{HS5}},\notag
\end{align}
where we use the projectivity of $\bar{\pi}_{V,Z}$ to apply \ref{HS5} and  
\[\partial_Z=\partial_{y_d}\circ \partial_{y_{d+1}}\circ\ldots\circ \partial_{y_{2d-1}}: K^M_{2d}(K(P_K))\to K^M_d(K(Z)),\]
(see \ref{HS2}  for notation) and 
\[\Nm^i_{\xi_Z/\xi}: K^M_d(\Frac( \sO_{Z,\xi_Z}^h))\to K^M_d(\Frac(\sO_{X_K,\xi}^h)),\quad i=1,2,\] 
is the norm map induced by \eqref{thm:psi5} which is an isomorphism.

By Theorem \ref{thm;AS} we have $a\in \tF(X,nD)\subset \FAS(X,nD)$. Thus  
the  element $\varphi_n(a_K)\in F((P^{(nD)}_X)_K)\subset F(P_K)$ from  \ref{para:psi} is defined.
Consider the diagram
\eq{thm:psi7}{\xymatrix{
\Theta\ar[r]^-{\sigma_V}\ar[d]_{\bar{\iota}_{n,Z}} & \ol{V}_{n,K}\ar[d]^{\bar{\iota}_n} & V_{n,K}\ar[l]\ar[d]^{\iota_n}\\
\ol{Z}\ar[r]^-{\sigma}     &\ol{P}_K & P_K\ar[l]
}}
where both squares are cartesian, all vertical maps and the horizontal maps on the left are closed immersions and the horizontal maps on the right are open immersions.
We obtain 
\begin{align}
\sum_{\lambda\in \Lambda}(\bar{p}_2^*(a_K)-\bar{p}^*_1(a_K),\gamma)_{\ol{P}_K/K,\uy_\lambda} 
&= \sum_{\lambda\in \Lambda} (\varphi_n(a_K), \gamma)_{\ol{P}_K/K,\uy_\lambda} \label{thm:psi8}\\
&   =\sum_{\lambda\in \Lambda} 
   (\bar{\iota}_{n,Z}^*\sigma^*\varphi_n(a_K), 
   \partial_{\Theta}\gamma)_{\Theta/K, (y_{\lambda,0},\ldots, y_{\lambda,d-2}\xi_Z)}\notag\\
                                                            & =0,\notag
\end{align}
where $\partial_{\Theta}$ is the composite of $ \partial_Z$ and 
$\partial_{\xi_Z}: K^M_d(K(Z))\to K^M_{d-1}(K(\Theta))$,
the second equality follows from $\varphi_n(a)\in F_{\ol{P}_K, \xi_Z}$ and \ref{HS2} and 
the last equality follows from \eqref{thm:psi7}, \ref{HS1}, 
and $\iota_n^*\varphi_n(a_K)=\psi_n(a_K)=0$ by the assumption on $a$.
Combining \eqref{thm:psi6} and \eqref{thm:psi8} yields the equality 
\[\left(a_K, \Nm^2_{\xi_Z/\xi}(\partial_Z\gamma)-\Nm^1_{\xi_Z/\xi}(\partial_Z\gamma)\right)_{X_K/K, \ux}=0.\]
Thus Claim \ref{thm:psi-claim}  for $\beta$ as in \eqref{thm:psi3} or in \eqref{thm:psi3.1}  (and hence also in general)
is a consequence of the following claim:
\begin{claim}\label{thm:psi9}
The elements $\eqref{thm:psi3}$ and  $\eqref{thm:psi3.1}$ modulo $(V_{d, X_K|nD_K})_{\xi}$ have representatives  in
\[\sum_{\ul{g}}
\Im\Big((\Nm^2_{\xi_{Z_{\ul{g}}}/\xi} -\Nm^1_{\xi_{Z_\ul{g}}/\xi})\circ \partial_{Z_{\ul{g}}} : K^M_{2d}(K(P_K))\to K^M_d(\Frac(\sO_{X_K,\xi}^h))\Big),\]
where $\ul{g}$ runs over all tuples $\ul{g}=(g_1,\ldots, g_d, m_1,\ldots m_d)$ as in \eqref{thm:psi3.3}.
\end{claim}
We prove the claim. Fix $\ul{g}$ and $Z=Z_{\ul{g}}$ as above.
Let $\sigma=(\sigma_1,\ldots, \sigma_{d-c})$, $c\in\{1,\ldots,d\}$, be part of a differential basis of $K$ over its prime field  and set
\[\gamma:=\{p_1^*(g_1)\tau_1-p_1^*(m_1), \ldots, p_1^*(g_d)\tau_d-p_1^*(m_d), s_1,s_{i_2}, \ldots, s_{i_c}, \sigma\}\in K^M_{2d}(K(P_K)),\]
where  $2\le i_2<\ldots<i_c\le d$. Then 
\[\partial_Z\gamma=\e\cdot \{(s_1)_{|Z},(s_{i_2})_{|Z},\ldots, (s_{i_c})_{|Z}, \sigma\}\in K^M_d(K(Z)),\]
where $\e\in \{\pm 1\}$ only depends on the choice of a sign in the definition of the tame symbol.
Since $\Nm^2_{\xi_Z/\xi}$ is an isomorphism  and $(s_i)_{|Z}=p_{Z,2}^*(z_i)$ we obtain
\[\Nm^2_{\xi_Z/\xi}(\partial_Z\gamma)=\e\cdot \{z_1,z_{i_2},\ldots, z_{i_c}, \sigma\}\in K^M_d(K^h_{X,\xi}).\]
On the other hand since $\theta_i=s_i-t_i=\tau_i t_1^n$ on $P$, we find in $K(Z)$
\[(s_j)_{|Z}= (t_j)_{|Z}+p_{Z,1}^*(\tfrac{m_j}{g_j}) (t_1)_{|Z}^n=p_{Z,1}^*(z_j+\tfrac{m_j}{g_j}z_1^n), \quad j=1,\ldots, d.\]
Since $\Nm^1_{\xi_Z/\xi}$ is an isomorphism  we obtain
\[\Nm^1_{\xi_Z/\xi}(\partial_Z\gamma)= 
\e\cdot \{z_1+\tfrac{m_1}{g_1} z_1^n, z_{i_2}+\tfrac{m_{i_2}}{g_{i_2}} z_1^n,\ldots, z_{i_c}+\tfrac{m_{i_c}}{g_{i_c}} z_1^n, 
\sigma_1,\ldots, \sigma_{d-c}\}.\]
Since $n\geq 2$ and we have 
\[\{(V_{1, X_K|(n-1)D_K})_{\xi}, (V_{1, X_K|nD_K})_{\xi}\}\subset (V_{2, X_K|(2n-2)D_K})_{\xi}\subset (V_{2, X_K|nD_K})_{\xi},\]
see, e.g., \cite[Lemma 2.7]{RS18}, we find 
\mlnl{
-\e\left(Nm^2_{\xi_Z/\xi}(\partial_Z\gamma)-\Nm^1_{\xi_Z/\xi}(\partial_Z\gamma) \right) \equiv 
\{1+\tfrac{m_1}{g_1} z_1^{n-1}, z_{i_2}, \ldots, z_{i_c}, \sigma\}\\
+\sum_{j=2}^c \{z_1, z_{i_2},\ldots, z_{i_{j-1}}, 1+\tfrac{m_{i_j}}{g_{i_j}z_{i_j}}z_1^n, z_{i_{j+1}}, \ldots, z_{i_c}, \sigma\}
\quad \text{mod } (V_{d, X_K|nD_K})_{\xi}.}
This  implies Claim \ref{thm:psi9} and  completes the proof of  the theorem.
\end{proof}

\begin{proof}[Proof of Theorem \ref{thm:ASII}.]
By Theorem \ref{thm;AS} it remains to show $\FAS(X,R)\subset \tF(X,R)$.
By the assumption on the existence of a projective SNC-compactification of $(X,R)$ and \cite[Theorem 6.8]{RS-ZNP}, it
suffices to show $\FAS(X,R)\subset \tF(X^h_\eta, R_{\eta}^h)$, for all $\eta\in R_\red^{(0)}$, where $X^h_\eta=\Spec \sO_{X,\eta}^h$ and 
$R^h_{\eta}=R\times_X X_{\eta}^h$. In particular we can shrink $X$ Zariski  locally  around the generic points of $R$ and hence assume
$R=nD$ with $D$ smooth and $n\ge 1$. Take  $a\in \FAS(X, nD)$. There exists an $N\ge n$ such that $a\in \tF(X, ND)$.
If $N>n$, then $a\in \FAS(X, (N-1)D)$ and  hence $\psi_N(a)=0$, by Lemma \ref{lem:ker-psi}.
Thus $a\in  \tF(X, (N-1)D)$ by Theorem \ref{thm:psi}.  This completes the proof.
\end{proof}

\section{Characteristic forms for reciprocity sheaves}\label{sec:charF}

\subsection{Preliminaries}\label{Preliminaries}

\begin{para}\label{para:RegB}
Fix  a scheme $B$. Let $\Sm_{B}$ be the category of smooth $B$-schemes, which are separated and of finite type.
Denote by $\ShSmB$  the category of Nisnevich sheaves of abelian groups on $\Sm_B$.
A morphism $h:B'\to B$  of schemes  gives rise to an adjoint pair 
\eq{para:RegB1}{h^*: \Sh_{B}\rightleftarrows \Sh_{B'}: h_*,}
where $(h_*G)(X)= G(X\times_B B')$ and $h^*F$ is given by the sheafification of 
\[ \Sm_{B'}\ni Y\mapsto \varinjlim_{Y/B'\to X/B} F(X).\]
If $Y$ is smooth over $B$,  then we have $(h^*F)(Y)= F(Y)$. In particular, if $h:B'\to B$ is smooth, 
then $\Sm_{B'}$ is a subcategory of $\Sm_B$ and $h^*F=F_{|\Sm_{B'}}$ is the restriction.
Let $\otShSmB$ be the tensor product in $\ShSmB$ given by the sheafification of
\[T\mapsto G_1(T)\otimes_\Z G_2(T).\]
For $F,G\in \Sh_B$  the internal hom
$\uHom_{\ShSmB}(G,F)$ in $\ShSmB$ is given by
\[ \uHom_{\ShSmB}(G,F)(T)=\Hom_{\ShSm T}(G_{|\Sm_T}, F_{|\Sm_T}), \quad \text{for } T\in \Sm_B.\]
Denote by 
\eq{adjunctionSmB}{\nu : \Hom_{\ShSmB}\big(G_1,\uHom_{\ShSmB}(G_2,F)\big) \xr{\simeq}
\Hom_{\ShSmB}(G_1\otShSmB G_2,F).}
the adjunction isomorphism. 
Denote by $\sO_B$ the restriction of the structure sheaf to  $\Sm_B$ 
and  by $\Sh_{\sO_B}\subset \ShSmB$ the full subcategory of $\sO_B$-modules.
The tensor product $\otimes_{\sO_B}$ in $\Sh_{\sO_B}$ is the sheafification 
of $T\mapsto G_1(T)\otimes_{\sO(T)} G_2(T)$.
\end{para}

\begin{lem}\label{lem1;HomSmB}
For any $G_1,G_2\in \Sh_{\sO_B}$ and $F\in \Sh_B$, \eqref{adjunctionSmB} induces an isomorphism
\[\Hom_{\Sh_{\sO_B}}\big(G_1,\uHom_{\ShSmB}(G_2,F)\big)\simeq
\Hom_{\ShSmB}(G_1\otimes_{\sO_B} G_2,F),\]
where $\uHom_{\ShSmB}(G_2,F)$ is endowed with an $\sO_B$-module structure via
that on $G_2$. 
\end{lem}
\begin{proof}
For $f\in \Hom_{\ShSmB}\big(G_1,\uHom_{\ShSmB}(G_2,F)\big)$, $\lambda\in \sO(T)$, and 
$g_i \in G_i(T)$, $i=1,2$, $T\in \Sm_B$, we have
\[ \nu(f)(\lambda g_1\otimes g_2)=f(\lambda g_1)(g_2),\quad
\nu(f)(g_1\otimes \lambda g_2)=f(g_1)(\lambda g_2)=(\lambda f(g_1))(g_2).\]
Hence $f$ belongs to $\Hom_{\Sh_{\sO_B}}\big(G_1,\uHom_{\ShSmB}(G_2,F)\big) $ if 
and only if $\nu(f)$ belongs to 
$\Hom_{\ShSmB}\big(G_1\otimes_{\sO_B} G_2,F).$
\end{proof}
\medbreak

\begin{para}\label{para:rep}
Let $\Sh^{\rm set}_B$ be the category of Nisnevich sheaves of sets on $\Sm_B$.
For $Y\in \Sm_B$ we denote also by $Y$ the sheaf  it represents in $\Sh^{\rm set}_B$, i.e.,
\[Y(T)=\Hom_{\Sm_B}(T,Y).\]
We denote by $\Z(Y)$ the free abelian sheaf generated by $Y$, i.e., 
the sheaf associated to  $T\mapsto \oplus_{a\in Y(T)}\Z\cdot [a]$.
For $F\in \Sh_B$ we obtain canonical isomorphisms
\eq{para:Fad1}{\Hom_{\Sh_B}(\Z(Y), F)\cong \Hom_{\Sh^{\rm set}_B}(Y,F)\cong F(Y).}
\end{para}

\begin{para}\label{para:Fad}
Let $G$ be a  commutative group scheme over $B$, which is smooth over $B$,  with group law 
$m:G\times_B G\to G$. Denote by $\sG\in \Sh_B$ the sheaf of abelian groups defined by $G$.
Thus the forgetful functor $\Sh_B \to \Sh^{\rm set}_B$ sends $\sG$ to $G$ and hence 
the counit of the adjunction \eqref{para:Fad1} yields the map $\sigma: \Z(G)\to \sG$,
which is induced by
\[\bigoplus_{a\in G(T)} \Z [a]\to \sG(T)=G(T), \quad \sum n_a [a]\mapsto \sum n_a a.\]
We obtain the  exact sequence in $\Sh_B$
\[\Z(G\times_B G)\xr{q_1+q_2-m} \Z(G)\xr{\sigma} \sG\to 0,\]
where $q_i: G\times_B G\to G$ is the $i$th-projection. 
Applying $\Hom_{\Sh_B}(-, F)$ with $F\in \Sh_B$ yields the isomorphism
\eq{para:Fad2}{ \Hom_{\Sh_B}(\sG, F)\xr{\simeq}
 F(G)_\ad:=\Ker\left(F(G)\xr{q_1^*+q_2^*-m^*} F(G\times_B G)\right).}
We call $F(G)_\ad$ the group of {\em additive elements} of $F(G)$ and we denote by
\[\chi_F: F(G)_\ad\xr{\simeq} \Hom_{\Sh_B}(\sG, F)\]
the inverse map of \eqref{para:Fad2}. 

For a smooth 
morphism $h:T\to B$ we obtain a commutative square
\[\xymatrix{
F(G)_\ad\ar[d]\ar[r]^-{\chi_F} & \Hom_{\Sh_B}(\sG,F)\ar[d]\\
F(G_T)_\ad\ar[r]^-{\chi_F} & \Hom_{\Sh_{T}}(\sG_{|\Sm_T}, F_{|\Sm_T}),
}\]
where $G_{T}= G\times_B T$ and the vertical maps are induced by pullback and restriction, respectively.
The  square commutes since both vertical maps are in fact induced by the natural map $F\to h_*h^*F$.
In particular we obtain a  sheaf $\ul{F(G)_\ad}\in \Sh_B$ given by 
\[\ul{F(G)_\ad}(T)=F(G_{T})_\ad \qquad (T\in \Sm_B)\] 
and $\chi_F$ induces an  isomorphism  in $\Sh_B$
\eq{para:Fad3}{\chi_F: \ul{F(G)_\ad}\xr{\simeq} \uHom_{\Sh_B}(\sG, F). }
\end{para}

\begin{para}\label{para:FNC}
Let $D$ be a normal crossing variety of $k$, i.e., the irreducible components $D_1,\ldots, D_r$ of $D$ are smooth of the same dimension and intersect transversely.
We define a functor 
\[\Sh_D\to \Sh_D, \quad F\mapsto F_D\]
by setting
\[F_D:= \Ker\left(\bigoplus_{1\le i\le r} h_{i*}h_i^*F\xr{\partial} \bigoplus_{1\le i<j\le r}h_{ij*}  h_{ij}^*F\right),\]
where $h_i: D_i\inj D$ and $h_{ij}: D_{i}\cap D_j\inj D$ are the closed immersions and $\partial$ is the difference of the obvious restrictions.
If $F\in \Sh_k$ and $\pi:D\to \Spec k$ is the structure map, then we set
\eq{para:FNC1}{F_D:= (\pi^*F)_D= 
\Ker\left(\bigoplus_{1\le i\le r} h_{i*}(F_{|\Sm_{D_i}})\xr{\partial} \bigoplus_{1\le i<j\le r} h_{ij*}(F_{|\Sm_{D_i\cap D_j}})\right).}

Let $G$ be a commutative group scheme over $D$, which is smooth over $D$.  
With the above notations the isomorphism $\chi_F$ from \eqref{para:Fad3} induces  an isomorphism
\[\chi_F: (\ul{F(G)_\ad})_{D}\xr{\simeq} \uHom_{\Sh_D}(\sG, F)_D.\]
\end{para}

\def\chiFO{\chi_{F,\sO}}

\begin{para}\label{para:chiFO}
Let $D$ be a normal crossing variety  with structure map $\pi: D\to \Spec k$ and let $E$ be  a vector bundle over $D$.  
The sheaf of abelian groups $\sE$ represented by $E$ has the additional structure of an $\sO_D$-module.
Let $F\in \Sh_k$.
Note that 
\[E=\Spec({\rm Sym}^\bullet \sE^\vee), \quad \text{where } \sE^\vee=\uHom_{\Sh_{\sO_D}}(\sE, \sO_D).\] 
From \eqref{para:Fad3} and \ref{para:FNC} we obtain an isomorphism 
\eq{eq;HomSmB}{\chi_{F,\sO_D} :(\ul{F(E)_\ad})_D\xr{\simeq} 
 \sE^\vee\otimes_{\sO_D} \uHom_{\Sh_k}(\sO_k, F)_{D},}
 where on the right hand side we use the notation \eqref{para:FNC1},
as the composition
\begin{align*}
(\ul{F(E)_\ad})_D  \rmapo{\chi_F} \uHom_{\Sh_D}(\sE, \pi^*F)_D &\cong\uHom_{\Sh_D}(\sE\otimes_{\sO_D}\sO_D , \pi^*F)_D \\
           &\cong \uHom_{\Sh_{\sO_D}}\big(\sE,\uHom_{\Sh_D}(\sO_D, \pi^*F)\big)_D\\
           &\cong \left(\sE^\vee\otimes_{\sO_D} \pi^*\uHom_{\Sh_k}(\sO_k, F) \right)_{D}\\
           &\cong \sE^\vee\otimes_{\sO_D} \uHom_{\Sh_k}(\sO_k, F)_D,
\end{align*}
where the second isomorphism follows from  Lemma \ref{lem1;HomSmB}, the third from the isomorphism 
$\pi^*\uHom_{\Sh_k}(\sO_k, F)\simeq \uHom_{\Sh_D}(\sO_D, \pi^*F)$ and the fact that $\sE$ is a coherent locally free $\sO_D$-module, 
and the last isomorphism   from the projection formula.

Assume $D$ is smooth and $\sE$ is a free $\sO_D$-module with basis $e_1,\ldots, e_r$.
Then the  map $\chi_{F,\sO_D}$ is explicitly given as follows:
Under the identification 
\[\Hom_{\Sm_{D}}(E,E)=\sE(E)=\oplus_{i=1}^r \sO(E)e_i\]
the identity $\id_E$ corresponds to $\sum_{i=1}^r e_i^\vee e_i$,
where $e_1^\vee,\ldots, e_r^\vee$ denotes the dual basis of \kay{$\sE^\vee$}.
Let $h\in F(E)_\ad$ and  $g=\chi_F(h)$, i.e., $g:\sE\to F_{|\Sm_D}$ is the map which  on $E$ satisfies 
(under the above identification) $g(\sum_{i=1}^r e_i^\vee e_i)= h$. 
Then 
\eq{para:chiFO1}{\chi_{F,\sO_D}(h)= \sum_{i=1}^r e_i^\vee\otimes g((-)e_i),}
where  $g((-)e_i)\in \uHom_{\Sh_k}(\sO_k, F)_D$
is defined by $g((-)e_i)(\lambda)= g(\lambda\cdot e_{i|T})$, for $\lambda\in\sO(T)$ and $T\in \Sm_{D}$.
\end{para}

\subsection{The characteristic form}\label{subsec:charF}
In this subsection $X\in \Sm$ and $R$ is an effective  Cartier divisor on $X$ with $D=R_{\red}$  a SNCD.
Set $U=X\setminus D$ and $D'= (R-D)_\red$. Let $F\in \RSC_{\Nis}$. 

\def\Db{\overline{D}}
\def\FAS{F^{AS}}
\def\Fad{F_{ad}}
\def\FDad{(F_D)_{ad}}

\begin{para}\label{para:charF}
Let $p_{1}$, $p_{2}: \PXR \to X$ and $R_P=p_1^{-1}(R)=p_2^{-1}(R)$ be as in \ref{para:dil}. 
Set  
\[ V_R:=R_P\times_X D'=\mathbb{V}(\Omega^1_{X/k}(R)_{|D'}) \subset \PXR,\]
cf.  Lemma \ref{lem:dil}\ref{lem:dil1}. In particular, $V_R$ is a smooth group scheme over the normal crossing variety $D'$.
By definition of $\FAS(X,R)$ and the injectivity of the restriction $F(P^{(R)}_X)\inj F(U\times U)$ (see \cite[Theorem 3.1(2)]{S-purity}) 
we have a unique map
\eq{para:charF1}{\varphi_R: \FAS(X,R)\to F(P^{(R)}_X),}
which satisfies $\varphi_R(a)_{|U\times U}=p_{2}^*(a_{|U})-p_{1}^*(a_{|U})$ (cf. \ref{para:psi}).
Let $D_1,\dots,D_r$ be the irreducible components of $D'$, i.e., those components of $D$ which appear with multiplicity $\ge 2$ in $R$,
and set $V_{R,i}:= V_R\times_D D_i$.
With the notation from \ref{para:FNC} we get a natural  map
\eq{para:charF2}{\psi_R : \FAS(X,R) \to F_{D'}(V_R) \subset \underset{1\leq i\leq r}{\bigoplus}\;  F(V_{R,i}), \quad 
a \mapsto \psi_R(a):=(\iota_i^*\varphi_R(a))_{1\leq i\leq r},}
where $\iota_i: V_{R,i}\to \PXR$ is the closed immersion.
\end{para}

\begin{lem}\label{lem4;AS}
 $\Im(\psi_R)\subset F_{D'}(V_R)_{\ad}$  (see \ref{para:FNC} for notation).
\end{lem}
\begin{proof}
It suffices to show $\iota_i^*\varphi_R(a)\in F(V_{R,i})_\ad$, for $i=1,\ldots, r$.
Denote by $q_1,q_2: \PXR\times_{p_2,X,p_1} \PXR \to \PXR$ 
and  $q'_1,q'_2: V_{R,i}\times_{D_i} V_{R,i}\to V_{R,i}$ the respective projections and by $\mu_V : V_{R,i}\times_{D_i} V_{R,i}\to V_{R,i}$ 
the addition of the vector bundle $V_{R,i}$ over $D_i$.
Let $\mu :\PXR\times_{p_2,X,p_1} \PXR \to \PXR$ be the map from Lemma \ref{lem:dil}\ref{lem:dil4}. For $a\in \FAS(X,R)$ we have 
by \eqref{eq1;lem,AS}
\eq{eq;lem4;AS}{ q_1^*(\varphi_R(a)) +  q_2^*(\varphi_R(a)) = \mu^* (\varphi_R(a)),}
as can be check immediately by restricting to $(U\times U)\times_U(U\times U)$.
Pulling back \eqref{eq;lem4;AS} via $\iota_i\times\iota_i$ and using Lemma \ref{lem:dil}\ref{lem:dil4}
yields the looked-for equality ${q'_1}^*(\iota_i^*\varphi_R(a)) + {q'_2}^*(\iota_i^*\varphi_R(a))=\mu_V^*(\iota_i^*\varphi_R(a))$.
\end{proof}

\begin{lem}\label{lem:add-Hom}
Let $F\in \RSC_{\Nis}$. 
Then  $\uHom_{\Sh_k}(\sO_k, F)\in \RSC_{\Nis}$.
\end{lem}
\begin{proof}
For $B\in \Sm$, let $\G_{a,B}=\G_a\times B$ be the restriction of $\G_a$ to $\Sm_B$ and denote by $q_i, m: G_{a,B}\times_B\times G_{a,B}\to G_{a,B}$
the projection to the $i$th factor and the multiplication, respectively. For $B\in\Sm_k$, we have
\[ \begin{aligned}
\uHom_{\Sh_k}(\sO_k, F)(B) &=\Hom_{\Sh_B}(\sO_B, F_{|\Sm_B}) \\
& = \Ker(F_{|\Sm_B}({\G_{a,B}})\xr{q_1^*+q_2^*-m^*} F_{|\Sm_B}({\G_{a,B}}\times_B {\G_{a,B}})) \\
& = \Ker(F({\G_{a,B}})\xr{q_1^*+q_2^*-m^*} F({\G_{a,B}}\times_B {\G_{a,B}})), \\
\end{aligned}\]
where the second equality holds by \eqref{para:Fad2}. Since
\[F(Y\times B)=\uHom_{\PST}(\Ztr(Y),F)(B) \qquad (Y\in \Sm),\]
we find that the sheaf $\uHom_{\Sh_k}(\sO_k, F)$ is equal to
\[\Ker\left(\uHom_{\PST}(\Ztr(\G_a), F)\xr{q_1^*+q_2^*-m^*} \uHom_{\PST}(\Ztr(\G_a\times \G_a), F)\right),\]
which is a reciprocity sheaf by \cite[Lemma 1.5]{MS}. 
\end{proof}

The following definition is inspired by Takeshi Saito's definition of the characteristic form in \cite{TakeshiSaito}, see Remark \ref{rmk:charF} below
for some more comments on the background.

\begin{defn}\label{def;charform}
Let the assumptions be as at the beginning of subsection \ref{subsec:charF} and write $\Omega^1_X$ for $\Omega^1_{X/k}$. 
We define the \emph{characteristic form of $F$ at $(X,R)$}
\[\charFR : \tF(X,R) \to
\Gamma\left(D', \Omega^1_X(R)_{|D'}\otimes_{\sO_{D'}}\uHom_{\Sh_k}(\sO_k,F)_{D'}\right),\]
where we apply  the notation \eqref{para:FNC1} to the reciprocity sheaf $\uHom_{\Sh_k}(\sO_k,F)$,
to be the composition
\mlnl{\tF(X,R) \inj \FAS(X,R)\xr{\psi_R}  F_{D'} (V_R)_\ad\\
\xr{\chi_{F,\sO_{D'}}} \Gamma\left(D', \Omega^1_X(R)_{|D'}\otimes_{\sO_{D'}}\uHom_{\Sh_k}(\sO_k,F)_{D'}\right),}
where the first map is the inclusion from Theorem \ref{thm;AS}, the map $\psi_R$ is induced by \eqref{para:charF2}
 and Lemma \ref{lem4;AS}, and 
$\chi_{F,\sO}$ is an isomorphism induced by \eqref{eq;HomSmB}.
\end{defn}

\begin{rmk}\label{rmk:charF}
For an element in $H^n(K, \Q/\Z(n-1))$, where $K$ is a henselian discrete valuation field, Kato defines in \cite[Definition (5.3)]{KatoSwan}
the {\em refined Swan conductor}. A non-logarithmic variant in the case $n=1$ and $p={\rm char}(K)>2$ was defined by Matsuda in \cite[Definition 3.2.5]{Matsuda}
and extended to the case $p=2$  by Yatagawa in \cite[Definition 1.18]{Yatagawa}.
By a similar method Kato-Russell define various versions  of  refined Swan conductors for Witt vectors  $W_n(K)$, see \cite[4.6, 4.7]{Kato-Russell}.
On the other hand, in \cite[Definition 2.19]{TakeshiSaito} Takeshi Saito defines the {\em characteristic form} of a torsor under an \'etale $k$-group scheme 
(not necessarily commutative) using dilatations. It is shown in \cite[Corollary 2.13]{Yatagawa} that the non-logarithmic refined Swan conductor
for $\Z/p^r\Z$-torsors defined by Matsuda-Yatagawa coincides with the characteristic form defined by Takeshi Saito. Thus one can think of the characteristic form
as a non-logarithmic version of the refined Swan conductor.
\end{rmk}

\begin{thm}\label{thm;charform}
Let the assumptions be as at the beginning of subsection \ref{subsec:charF}.
Assume that $(X,D)$ admits a projective SNC-compactification. Then we have
\[\tF(X,R-D') = \Ker(\charFR),\]
i.e., the characteristic form of $F$ at $(X,R)$ induces an injective map
\[\charFR : \frac{\tF(X,R)}{\tF(X, R-D')}\inj 
\Gamma\left(D', \Omega^1_X(R)_{|D'}\otimes_{\sO_{D'}}\uHom_{\Sh_k}(\sO_k,F)_{D'}\right).\]
\end{thm}
\begin{proof}
By the Zariski-Nagata purity for $F$ (\cite[Theorem 6.8]{RS-ZNP}), the question is local at the generic points of $R$.
Hence we may assume $R= n D$ with $D$ irreducible smooth. If $n=1$, then $D'=\emptyset$ 
 and the assertion trivially holds.
If $n\ge 2$,  then the statement follows from Theorem \ref{thm:psi} and  \eqref{eq;HomSmB}.
\end{proof}

\begin{para}\label{para:charform-local}
Let the assumptions be as at the beginning of subsection \ref{subsec:charF}.
Denote by $\tF_{(X,R)}$ the sheaf on  $X_{\Nis}$ given by 
\[(u:V\to X)\mapsto \tF(V, u^*R).\]
Denote by $i: D'\inj X$ the closed immersion, then $\psi_R$ clearly induces a morphism of sheaves on $X_{\Nis}$
\[\tF_{(X,R)}\xr{\psi_R} i_*(( \ul{F(V_R)_{\ad}})_{D'_\Nis}),\]
where $( \ul{F(V_R)_{\ad}})_{D'_\Nis}$ denotes the restriction of $(\ul{F(V_R)_{\ad}})_{D'}\in \Sh_{D'}$ from \ref{para:FNC} to the small Nisnevich site on $D'$.

Thus we obtain a local version of the characteristic form 
\[\charFR: \tF_{(X,R)}\to \Omega^1_X(R)\otimes_{\sO_X} i_*(\uHom_{\Sh_k}(\sO_k, F)_{D'_{\Nis}}).\]
If $(X,D)$ admits a projective SNC-compactification, then $\charFR$ induces by Theorem \ref{thm;charform} an injection of sheaves on $X_{\Zar}$ 
\[\tF_{(X,R)}/ \tF_{(X, R-D')}\inj \Omega^1_X(R)\otimes_{\sO_X} i_*(\uHom_{\Sh_k}(\sO_k, F)_{D'_{\Zar}}).\]
\end{para}

\begin{rmk}
Let $D=\sum_{i=1}^r D_i$ be the irreducible components of $D$.
Set $X_i=X\setminus \cup_{j\neq i} D_j$ and denote the restriction of $D_i$ to $X_i$ again by $D_i$.
By the Gysin sequence in \cite[Theorem 7.15]{BRS} we have  isomorphisms
\[\frac{F(X_i,D_i)}{F(X_i)}\xr{\partial} \Gamma\left (D_i,\uHom_{\PST}(\G_m, F)\right),\quad i=1,\ldots, r,\]
where $\partial_i$ is induced by the connecting homomorphism in the Gysin sequence for 
$(D_i,\emptyset)\inj (X_i,\emptyset)$. 
By \cite[Corollary 8.6(1)]{S-purity} the natural map induced by restriction
\[\frac{F(X,D)}{F(X)}\to \bigoplus_{i=1}^r \frac{F(X_i,D_i)}{F(X_i)}\]
is injective.
Composing the two maps above yields an injection 
\[\frac{F(X,D)}{F(X)}\inj \bigoplus_{i=1}^r \Gamma(D_i,\uHom_{\PST}(\G_m, F)),\]
which can be viewed as a characteristic form for reduced $R$.
 \end{rmk}
 
 \begin{para}\label{para:app-Chow}
 As an exemplary application, we explain how  Theorem \ref{thm;charform}  and its local form from \ref{para:charform-local}
reveals an interesting structure of Chow groups of zero-cycles with modulus, as introduced in \cite{Kerz-Saito}.
Let $Y$ be a proper $k$-scheme with an effective Cartier divisor $E$, such that $V=Y\setminus E$ is smooth. 
By \cite[Corollary 2.3.5]{KSY2} and \cite[Theorem 0.1]{S-purity} there is a reciprocity sheaf $h_0(Y,E)_\Nis$ such that for any field $K$ over $k$  we have 
\[h_0(Y,E)_\Nis(K)=\CH_0(Y_K,E_K),\]
where the right hand side denotes the Chow group of zero-cycles with modulus and $Y_K=Y\otimes_k K$.

Assume $L$ is a henselian discrete valuation field of geometric type over $k$ with ring of integers $\sO_L$, maximal ideal $\fm_L$, and residue
field $K=\sO_L/\fm_L$. If we are in positive characteristic, we assume that the transcendence degree of $L/k$ is $\le 3$, so that all 
geometric models of $(\sO_L, \fm_L)$ have a projective SNC-compactification (see \ref{para:sm-cptf}).
Then ${\rm fil}_n\CH_0(Y_L, E_L):= h_0(Y, E)_{\Nis}(S, n s)$, where $S=\Spec \sO_L$ and $s\in S$ is the closed point, defines a filtration
\[{\rm fil}_0\CH_0(Y_L, E_L)\subset {\rm fil}_1\CH_0(Y_L, E_L)\subset\ldots\subset {\rm fil}_n\CH_0(Y_L, E_L)\subset \ldots\subset \CH_0(Y_L, E_L),\]
where ${\rm fil}_0\CH_0(Y_L, E_L)$ is the subgroup of $\CH_0(Y_L, E_L)$ generated by closed points in $V_L$ whose closure in $V\times_k S$ is finite over $S$.
By  Theorem \ref{thm;charform} we have an injection
\[\frac{{\rm fil}_n\CH_0(Y_L, E_L)}{{\rm fil}_{n-1}\CH_0(Y_L, E_L)}\hookrightarrow
\Omega^1_{\sO_L}\otimes_{\sO_L}\frac{\fm_L^{-n}}{\fm_L^{-n+1}}\otimes_{K} 
\uHom_{\Sh_k}(\sO_k, h_0(Y,E)_{\Nis})(K) \qquad (n\ge 2).\]
However the internal hom on the right is not well understood
and it would be interesting to describe the image of the map.

Observe that if ${\rm trdeg}(L/k)=1$, then $\fm_L/\fm_L^2 \to \Omega^1_{\sO_L}\otimes_{\sO_L} K$, $a \text{ mod } \fm_L^2\mapsto da\otimes 1$ is an isomorphism
(since $\Omega^1_K=\Omega^1_{K/k}=0$) and we obtain an isomorphism of $K$-vector spaces
\[\Omega^1_{\sO_L}\otimes \fm_L^{-n}/\fm_{L}^{-n+1}\xr{\simeq} \fm^{-n+1}_L/\fm_L^{-n+2}, \quad 
dz\otimes (\frac{1}{z^n}\text{ mod } \fm_L^{-n+1})\mapsto \frac{1}{z^{n-1}}\text{ mod } \fm_L^{-n+2},\]
where $z\in \fm_L$ is a local parameter. Hence the choice of a basis $e$ of the 1-dimensional $K$-vector space  $\fm_L^{-n+1}/\fm_L^{-n+2}$ 
induces a canonical map
\[s_{n,e}:\frac{{\rm fil}_n\CH_0(Y_L, E_L)}{{\rm fil}_{n-1}\CH_0(Y_L, E_L)}\lra \CH_0(Y_K,E_K)\qquad (n\ge 2),\]
as the composition 
\mlnl{\frac{{\rm fil}_n\CH_0(Y_L, E_L)}{{\rm fil}_{n-1}\CH_0(Y_L, E_L)}\inj 
\Omega^1_{\sO_L}\otimes_{\sO_L}\frac{\fm_L^{-n}}{\fm_L^{-n+1}}\otimes_{K} 
\uHom_{\Sh_k}(\sO_k, h_0(Y,E)_{\Nis})(K)\\
\cong  K\cdot e\otimes_K \uHom_{\Sh_k}(\sO_k, h_0(Y,E)_{\Nis})(K)\cong  \uHom_{\Sh_k}(\sO_k, h_0(Y,E)_{\Nis})(K)\\
\xr{{\rm ev}_1} h_0(Y,E)_{\Nis}(K)=\CH_0(Y_K,E_K),}
where the first inclusion is induced by the characteristic form and ${\rm ev}_1$ sends $\varphi\in \Hom_{\Sh_K}(\sO_K, h_0(Y,E)_{\Nis|\Spec K})$
to $\varphi(K)(1)$.
This is a new map and it is tempting to view it as a specialization map. 
It remains to study its properties  more closely,  e.g., if it happens to be injective for certain pairs $(Y,E)$.
\end{para}

%

\section{The characteristic form of Witt vectors and torsion characters of the fundamental group}\label{sec:ex-grps}
In this section we give a description of the characteristic form for the group of Witt vectors of finite length. From this, we
deduce that the characteristic form for torsion characters of the fundamental group
is a variant of the refined Swan conductor defined by Kato in \cite{KatoSwan}.
The necessary computations were actually already done in \cite{Yatagawa}.

Fix $X\in \Sm$ and $i:D\inj X$ a smooth connected divisor. 
Let $p={\rm char}(k)$.
For a henselian discrete valuation field $L$ of geometric type over $k$ (see \ref{para:cond})
we denote by  $\sO_L$ its ring of integers, by $\fm_L$ the maximal ideal,  and for $F\in \RSC_{\Nis}$ we write
\[\tilde{F}(\sO_L, \fm_L^{-j}):= \tilde{F}(S, j\cdot s), \quad j\ge 0,\]
where $S=\Spec \sO_L$ and $s\in S$ is the closed point. In the following $\G_a$ denotes the additive group
and, in positive characteristic,  $W_n$ the group scheme of $p$-typical Witt vectors of length $n$;
we have $\G_a$, $W_n\in \RSC_{\Nis}$, cf. Example \ref{ex:ASII}.
We write $\Omega^1_X$ for $\Omega^1_{X/k}$.

\begin{para}\label{para:matsuda}
Let $L$ be a henselian dvf of geometric type over $k$  with normalized valuation $v$.
If $p=0$, then we have by \cite[Theorem 6.4]{RS} 
\[\tilde{\G}_a(\sO_L,\fm_L^{-r})=\fm_L^{-r+1}, \quad r\ge 0.\]
Assume $p>0$.  We denote by ${\rm fil}_*W_n(L)$ the Matsuda filtration of $W_n(L)$
(see \cite{Matsuda} for the original definition, and e.g. \cite[7.3]{RS} for the shifted version used here).
Recall that for $a=(a_0,\ldots, a_{n-1})=\sum_{i=0}^{n-1} V^i ([a_i])\in W_n(L)$ and $r\ge 0$ we have 
\eq{eq:Matsuda}{ a\in {\rm fil}_r W_n(L)\Longleftrightarrow  
p^{n-1-i}v(a_i)\ge \begin{cases}  -r & \text{if } i\neq n-1-m\\  -r+1&\text{if } i=n-1-m, \end{cases}} 
where $m=\min\{n, {\rm ord}_p(r)\}$.
 By \cite[Theorem 7.20]{RS} we have 
\eq{eq:Matsuda2}{\widetilde{W_n}(\sO_L,\fm_L^{-r})= \sum_{j\ge 0}F^j({\rm fil}_r W_n(L)),}
where $F^j: W_n(L)\to W_n(L)$, $(a_0,\ldots, a_{n-1})\mapsto (a_0^{p^j},\ldots, a_{n-1}^{p^j})$ is the $j$th iterated absolute Frobenius.
\end{para}

\begin{para}\label{para:global-matsuda}
Assume $p>0$. 
We now study a global version of \eqref{eq:Matsuda2}. Let $U=X\setminus D$. 
We define  ${\rm fil}_{(X,rD)}W_n$ for $r\ge 0$ to be the subset of $W_n(U)$ consisting of those Witt vectors 
$(a_0,\ldots, a_{n-1})\in W_n(U)$
satisfying
\[F^{n-1-i}(a_i)\in \begin{cases} H^0(X,\sO_X(rD)) & \text{if } i\neq n-1-m\\ H^0(X,\sO_X((r-1)D)) &\text{if } i=n-1-m,\end{cases}\]
where $m=\min\{n, {\rm ord}_p(r)\}$.
Let $\eta\in X^{(1)}$ be the generic point of $D$ and $L=\Frac(K^h_{X,\eta})$. Then we have 
$H^0(X,\sO_X(rD))= H^0(U,\sO_U)\cap m_L^{-r}$ and hence also
\eq{para:global-matsuda1}{{\rm fil}_{(X,rD)}W_n= W_n(U)\cap {\rm fil}_rW_n(L).}
We set
\[{\rm fil}_{(X,rD)}^F W_n:=\sum_{j\ge 0} F^j({\rm fil}_{(X,rD)}W_n)\subset W_n(U).\]
\end{para}

\begin{prop}\label{prop:global-fil}
Assumptions and notations as in \ref{para:global-matsuda}.
Then
\[W_n^{AS}(X,rD)=\widetilde{W_n}(X,rD)={\rm fil}^F_{(X,rD)}W_n.\]
\end{prop}
\begin{proof}
The first equality holds by Corollary \ref{cor:ASII} and Example \ref{ex:ASII}\ref{ex:ASII1}.
We prove the second equality. By \cite[Corollary 8.6(2)]{S-purity} we have
\eq{prop:global-fil1}{\widetilde{W_n}(X, rD)= W_n(U)\cap \widetilde{W_n}(\sO_L,\fm_L^{-r}),\quad \text{for all }r\ge 0,}
where $L=K^h_{X,D}$. Thus by \eqref{eq:Matsuda2}  
\eq{prop:global-fil2}{{\rm fil}^F_{(X,rD)}W_n \subset \widetilde{W_n}(X,rD).}
For the other inclusion we have to show that given $\alpha_i\in {\rm fil}_rW_n(L)$ with $\sum_i F^i(\alpha_i)\in W_n(U)$, then we have
$\alpha_i\in W_n(U)$, for all $i$. To this end we reduce to the case of curves as follows.
Define $\widetilde{W_n}^{\le 1}(X,rD)$ and ${\rm fil}_{(X,rD)}^{F,\le 1}W_n$ in an analogous way as in
\eqref{cor:ASII2}, where the maps $f: Z\to X$  have $\dim Z\le 1$ and we additionally assume $(f^*D)_{\red}$ to be connected, i.e., it is a point.
Since $W_n$ has level $1$, we have $\widetilde{W_n}(X, rD)= \widetilde{W_n}^{\le 1}(X, rD)$.
\begin{claim}\label{prop:global-fil-claim}
\[{\rm fil}^F_{(X,rD)}W_n= {\rm fil}^{F, \le 1}_{(X,rD)}W_n \qquad (r\ge 0).\]
\end{claim}
We prove the claim. Obviously we have ${\rm fil}^F_{(X,rD)}W_n\subset {\rm fil}^{F, \le 1}_{(X,rD)}W_n$.
For the other inclusion let $\alpha=\sum_{s=0}^{n-1} V^s([a_s])\in W_n(U)\setminus {\rm fil}^F_{(X,rD)}W_n$, with $a_s\in \sO(U)$. 
We have to show, that there exists a smooth affine curve $C$ and a map  $f:C\to  X$ with $(f^*D)_\red$ a point, 
such that $f^*a\not\in {\rm fil}^F_{(C,rf^*D)}W_n$.
Let 
\[s_0:=\min\{s\mid V^s([a_s])\in W_{n}(U)\setminus {\rm fil}^F_{(X,rD)}W_{n}\}.\]
Then for any map $f:C\to X$ as above we have 
\[f^*(\alpha)\equiv V^{s_0}([f^*a_{s_0}])+\ldots+V^{n-1}([f^*a_{n-1}])\quad \text{mod } {\rm fil}^F_{(C, r f^*D)}W_n.\]
Therefore it  remains to  find a map $f:C\to X$ as above with $V^{s_0}(f^*[a_{s_0}])\not\in {\rm fil}^F_{(C,rf^*D)}W_{n}$.
To this end let $j$ be the maximal non-negative integer such that  $a_{s_0}=F^j(b)$, for some $b\in \sO(U)$. 
After shrinking $X$ around a closed point of $D_{\red}$, we may assume
$D=\div(z)$ and  $b=e/ z^\nu$, with  $e\in \sO(X)$, such that 
$E:=\div(e)$ and $D$ have no common irreducible component, and $\nu\ge 1$ satisfies
\eq{prop:global-fil2.5}{p^{n-1-s_0}\nu>\begin{cases} r & \text{if } s_0\neq n-1-m\\ r-1 &\text{if } s_0=n-1-m,\end{cases}}
where $m=\min\{n, {\rm ord}_p(r)\}$. Let $y\in D$ be a closed point which does not lie on $E$. 
For any regular sequence of parameters  $z_1=z, z_2,\ldots, z_d$ of $\sO_{X,y}$, the vanishing set $C=V(z_2,\ldots, z_d)$
defines a smooth affine integral curve in a neighborhood of $y$ intersecting $D$ transversally at $y$ and we may assume that
it intersects $D$ only at $y$ and does not intersect $E$, so that $e_{|C}\in \sO(C)^\times$. Furthermore,  
if $e$ is no $p$th-power in $\sO(X)$, then we may assume that  $e_{|C}$ is not a $p$th-power in $\sO(C)$.
This together with the choice of $j$ above and \eqref{prop:global-fil2.5} 
implies, that $b_{|C}\in \sO(C\setminus D_{|C})$ is not a $p$th-power and that  we have
\[p^{n-1-s_0}\div(b_{|C})< \begin{cases} -rD_{|C} & \text{if }s_0\neq n-1-m\\ -(r-1)D_{|C} & \text{if } s_0=n-1-m.\end{cases}\]
Hence $a_{s_0|C}= F^j(b_{|C})\not\in {\rm fil}^F_{C, rD_{|C}}W_n$, which proves Claim \ref{prop:global-fil-claim}.

Thus we are reduced to the case, where $X$ is smooth, affine,  1-dimensional, and $D$ is a point.
The inclusion \eqref{prop:global-fil2} is an isomorphism for $r=0$ and induces maps on the graded pieces for $r\ge 1$
\eq{prop:global-fil3}{\frac{{\rm fil}^F_{(X,rD)}W_n}{{\rm fil}^F_{(X,(r-1)D)}W_n}\to \frac{\widetilde{W_n}(X,rD)}{\widetilde{W_n}(X,(r-1)D)}\inj 
\frac{\widetilde{W_n}(\sO_L,\fm_L^{-r})}{\widetilde{W_n}(\sO_L,\fm_L^{-(r-1)})},}
where the last map is injective, by \eqref{prop:global-fil1}. 
Since $X$ is a smooth affine curve, it follows directly from the definitions that the composition of \eqref{prop:global-fil3} 
is an isomorphism. This implies the statement.
\end{proof}

\begin{ex}
Let $G$ be a commutative unipontent $k$-group.
Then we can embed $G$ into a finite direct sum of $p$-typical Witt group schemes $G\subset \oplus_{i} W_{n_i}$;
if $p=0$, then we can take $W_{n_i}=\G_a$.
Thus it follows from the   local form of the ramification filtration \eqref{eq:Matsuda2} that we have $\widetilde{G}(X, D)= \widetilde{G}(X)$ 
(recall that $D$ is assumed to be smooth).  The global  description from Proposition \ref{prop:global-fil} imposes further constraints.
For example, assume $p\ge 3$ and let $E\to S$ be a relative  elliptic curve over $S\in \Sm$ and let $D\cong S\subset E$ be a section.
Then it follows from Proposition \ref{prop:global-fil}, that any morphism of $k$-schemes
$E\setminus D\to G$, whose ramification is bounded by $2 D$, is constant in the sense that it factors as $E\setminus D\to S\to G$.
\end{ex}

\begin{para}\label{para:cgs}
Let $G$ be a commutative group scheme over $k$. With the notation from \ref{para:RegB} we have 
\[\uHom_{\Sh_k}(\sO_k, G)_D (T)=\Hom_{T-{\rm grp}}(\G_{a,T}, G_T)\qquad (T\in \Sm_D),\]
where the $G_T=G\times T$  and $\Hom_{T-{\rm grp}}$ denotes the homomorphisms  of $T$-group schemes.
It follows that for $p=0$, we have the following equality of Nisnevich sheaves on $D$
\[\uHom_{\Sh_k}(\sO_k, \G_{a})_{D_\Nis}=\sO_D.\]
For $p>0$, any  homomorphism $\G_{a, D}\to W_{n,D}$  of group schemes over $D$ 
has to factor via $V^{n-1}: \G_{a,D}\to W_{n,D}$ and hence we have 
\eq{eq:HomWn}{\uHom_{\Sh_k}(\sO_k, W_{n})_{D_\Nis}=\bigoplus_{j\ge 0} V^{n-1}\cdot \sO_D\cdot F^j,}
where $F^j: \G_{a,D}\to \G_{a,D}$ is induced by $\sO_D[t]\to \sO_D[t]$, $t\mapsto t^{p^j}$.
Under this identification, the $\sO_D$-module structure induced on the right hand side is given by 
\eq{eq:HomWnMod}{a \cdot (\sum_j V^{n-1} b_j  F^j)= \sum_j V^{n-1}   b_j a^{p^j}  F^j\qquad
 (a\in \sO_D).}
\end{para}

In view of Remark \ref{rmk:charF}, the following corollary can be deduced
from Yatagawa's computations in \cite{Yatagawa}.
\begin{cor}\label{prop:charW}
Let $X,D$ be as at the beginning of this section and assume $r\ge 2$. Denote by $\nu: U=X\setminus D\inj X$ the open immersion.
\begin{enumerate}[label=(\arabic*)]
\item\label{prop:charW1} Assume $p=0$. The characteristic form
\[\charr_{\G_a}^{(rD)}:  \tilde{\G}_{a, (X,rD)}=\sO_{X}((r-1)D)\to \Omega^1_{X}(rD)\otimes_{\sO_X} i_*\sO_D\]
is induced by the differential $d: \nu_*\sO_U\to \nu_*\Omega^1_{U}\otimes_{\sO_X} i_*\sO_D$, $a\mapsto da\otimes 1$.
\item\label{prop:charW2} Assume $p>0$ and $n\ge 1$.
The characteristic form
\[\charr_{W_n}^{(rD)}: \widetilde{W_n}_{(X,rD)}\to 
\Omega^{1}_{X}(rD)\otimes_{\sO_X } \bigoplus_{j\ge 0} V^{n-1}\cdot i_*\sO_D \cdot F^j\]
is  given as follows (using Proposition \ref{prop:global-fil}):
\[\charr_{W_n}^{(rD)}(F^j(a))= \begin{cases}
F^{n-1}d(a) \otimes V^{n-1} F^j & \text{if } (p,r)\neq (2,2),\\
F^{n-1}d(a) \otimes V^{n-1} F^j  +  \frac{dz}{z^2}\otimes V^{n-1} \alpha^{2^j} F^{j+1}& 
\text{if } (p,r)=(2,2),
\end{cases}\]
where $a=(a_{0},\ldots, a_{n-1})\in {\rm fil}_{(X,rD)} W_n$, $j\ge 0$, and  $F^{n-1}d: \nu_*W_n\sO_U\to \nu_*\Omega^1_{U}$ is given by 
\[F^{n-1}d(a)=\sum_{i=0}^{n-1} a_i^{p^{n-i-1}} da_i,\]
and in the second case $z$ denotes a local equation of $D$ and  $\alpha\in \sO_D$ is defined as the image of  $z^2a_{n-1}$ in $\sO_D$, where 
we use that in the case $(p,r)= (2,2)$ we have $\kay{z^2a_{n-1}}\in \sO_X$.
\end{enumerate}
\end{cor}
\begin{proof}
\ref{prop:charW2}. We make the map $\chi_{W_n,\sO_D}$ from \eqref{eq;HomSmB} explicit.
It suffices to consider the situation from   \eqref{para:dil-loc2}; with the notation from there we have
\[V_{rD}:=(P^{(rD)}_X)_{|\Spec B}\times_{X\times X}  D = \Spec (A/z)[\tau_1,\ldots, \tau_d],\]
which we view as a vector bundle over $D$ with sheaf of sections given by the dual of  \eqref{lem:dil7}. 
In the situation of \eqref{para:psi1}, the above comes from  a natural identification
\[  \Omega^1_X(rD)_{|D}\cong \underset{1\leq i\leq d}{\bigoplus}\sO_D\cdot \tau_i,\quad
\frac{dz_i}{z^r}\mapsto \tau_i, \]
and we are in the situation of \ref{para:chiFO} with $\sE^\vee=\Omega^1_X(rD)_{|D}$.
It follows from \eqref{para:Fad2}, that $W_n(V_{rD})_{\ad}$ is the abelian subgroup of $W_n((A/z)[\tau_1,\ldots,\tau_d])$ generated by the elements
\[ V^{n-1}(\tau_i^{p^j}a),\quad a\in A/z, \,i=1,\ldots, d, \, j\ge 0.\]
Thus  by \eqref{para:chiFO1} (with $e_i^\vee=\tau_i$ and $e_i=\tau_i^\vee$) and using \eqref{eq:HomWn}, the map
\[\chi_{W_n,\sO_D}: (\ul{W_n(V_{rD})_{\ad}})_D\to \Omega^1_X(rD)_{|D}  \otimes_{\sO_D} \bigoplus_{j\ge 0} V^{n-1}\cdot \sO_D\cdot F^j\]
is given on $D_{\Nis}$ by
\[\chi_{W_n,\sO_D}(V^{n-1}(\tau_i^{p^j}a))= \tau_i\otimes V^{n-1}\cdot a\cdot F^j=\frac{dz_i}{z^r}\otimes V^{n-1}\cdot a\cdot F^j.\]
In particular,
\[\chi_{W_n,\sO_D}(F^jV^{n-1}(\tau_i a))=\tilde{a}\frac{dz_i}{z^r}\otimes V^{n-1} F^j,\]
where $\tilde{a}\in A$ is a lift of $a\in A/z$.
Since $\charr^{(rD)}_{W_n}=\chi_{W_n, \sO_D}\circ \psi_r$ it remains to compute 
\[\psi_r: \widetilde{W_n}(X,rD)=W^{\rm AS}_n(X,rD)\to W_n(V_{rD})_{\ad}\]
from \eqref{para:psi2}. This computation can be deduced from similar computations as in
the proof of \cite[Proposition 2.12]{Yatagawa},  see also \cite[(1.21), Proposition 1.17, Definition 1.42]{Yatagawa}.
Finally, the proof of \ref{prop:charW1} is similar but much easier and therefore left to the reader.
\end{proof}


\begin{rmk}\label{rmk:KR}
Let $L$ be a henselian dvf of geometric type over $k$ with residue field $K=\sO_L/\fm_L$ and assume $p>0$.
Denote by 
\[{\rm fil}^{\log}_rW_n(L)=\{(a_0,\ldots, a_{n-1})\in W_n(L)\mid p^{n-1-i}v(a_i)\ge -r\}\]
the Brylinski-Kato filtartion of $W_n(L)$ and by ${\rm fil}_r^{\log,F} W_n(L)$ its $F$-saturation.
Then by Corollary \ref{prop:charW} and in view of \eqref{eq:Matsuda2}, the characteristic form fits into the following commutative diagram
\[\xymatrix{
0\ar[r] & \frac{{\rm fil}^{\log,F}_{r-1}W_n(L)}{\widetilde{W_n}(\sO_L,\fm_L^{-r+1})}\ar[r]\ar@{^(->}[d] 
   & \frac{\widetilde{W_n}(\sO_L,\fm_L^{-r})}{ \widetilde{W_n}(\sO_L,\fm_L^{-r+1})}\ar[r]\ar[d]^{\charr^{(r)}_{W_n}} 
   &\frac{\widetilde{W_n}(\sO_L,\fm_L^{-r})}{{\rm fil}^{\log F}_{r-1} W_n(L)}\ar[r]\ar@{^(->}[d] &  0\\
0\ar[r]&\frac{\fm_L^{-r+1}}{\fm_L^{-r+2}}[F]\ar[r] 
  & \Omega^1_{\sO_L}\otimes_{\sO_L} \frac{\fm_L^{-r}}{\fm_L^{-r+1}}[F]\ar[r] 
  &\Omega^1_{K}\otimes_{K} \frac{\fm_L^{-r}}{\fm_L^{-r+1}}[F]\ar[r] & 0,
}\]
where the lower sequence is equal to $(-)\otimes_{\sO_L} \frac{\fm_L^{-r}}{\fm_L^{-r+1}}\otimes_K K[F]$ applied to the standard sequence
\[0\to \fm_L/\fm_L^2\to \Omega^1_{\sO_L}\otimes_{\sO_L} K\to \Omega^1_K\to 0\]
and  the two outer vertical maps are constructed and shown to be injective in \cite[4.7]{Kato-Russell}; they are induced by $F^{n-1}d$.
By Theorem \ref{thm;charform} we know that the middle vertical map is injective if ${\rm trdeg}(L/k)\le 3$.
On the other hand the above diagram and the injectivity of the outer vertical maps show that the middle map is always injective. 

Note that in \cite[4.6]{Kato-Russell} Kato-Russell  also construct an injective map (induced by $F^{n-1}d$)
\[\bar{\theta}_r: {\rm fil}^{\log, F}_rW_n(L)/{\rm fil}^{\log, F}_{r-1}W_n(L)\inj \Omega^1_{\sO_L}(\log)\otimes_{\sO_L} \frac{\fm_L^{-r}}{\fm_L^{-r+1}}[F],\]
which by \cite[Proposition 4.9]{Kato-Russell} is a lift of the refined Swan conductor of \cite[Definition 5.3]{KatoSwan}.
This is why the characteristic form can be viewed as  a non-log version of this $\bar{\theta}_r$.
\end{rmk}

\begin{para}\label{para:CW}
Assume $p>0$ and $r\ge 2$.
By the functoriality of the characteristic form we have a commutative diagram
\[\xymatrix{
\widetilde{W_{n+1}}_{(X, rD)}\ar[r]^-{\charr^{(rD)}_{W_{n+1}}} & 
\Omega^1_X(rD)\otimes_{\sO_X} \bigoplus_{j\ge 0}V^n\cdot i_*\sO_D\cdot F^j\\
\widetilde{W_n}_{(X, rD)}\ar[r]^-{\charr^{(rD)}_{W_n}}\ar[u]^{V} & 
\Omega^1_X(rD)\otimes_{\sO_X} \bigoplus_{j\ge 0}V^{n-1}\cdot i_*\sO_D\cdot F^j\ar[u]^{\id\otimes V\cdot}
}\]
Taking the direct limit over $V$, we obtain the characteristic form of the co-vectors $CW=\varinjlim_V W_n\sO_X$:
\[\charr^{(rD)}_{CW}: \widetilde{CW}_{(X,rD)}\to \Omega^1_X(rD)\otimes_{\sO_X} i_*\sO_D[F],\]
where $\sO_{D}[F]=\oplus_{j\ge 0} \sO_{D}\cdot F^j$ with $\sO_D$-module structure given by
$a\cdot(\sum_j b_j F^j)= \sum_{j} b_ja^{p^j}F^j$. 

The Artin-Schreier-Witt sequences  give in the direct limit rise to an exact sequence  
\[0\to \Q/\Z\to CW\xr{F-\id} CW\xr{\partial} H^1_{\et}\to 0,\]
of sheaves on $\Sm_{\Nis}$, where $H^1_{\et}$ is the reciprocity sheaf (cf. \cite[8.1]{RS})
\[X\mapsto H^1_{\et}(X):=H^1(X_{\et},\Q/\Z)= H^0(X_{\Nis}, R^1\e_*\Q/\Z)
=\Hom_{{\rm cont}}(\pi_1^{\rm ab}(X),\Q/\Z),\]
with  $\e: \Sm_{\et}\to\Sm_{\Nis}$ the change of sites map. We obtain the commutative diagram
\[\xymatrix{
\widetilde{CW}_{(X,rD)}\ar[rr]^-{\charr^{(rD)}_{CW}}\ar[d]^{\partial} & &
 \Omega^1_X(rD)\otimes_{\sO_X} i_*\sO_D[F]\ar[d]\\
 \widetilde{H^1_{\et}}_{(X,rD)}\ar[rr]^-{\charr^{(rD)}_{H^1_{\et}}}& &
 \Omega^1_X(rD)\otimes_{\sO_X} i_*\uHom_{\Sh_k}(\sO_k, H^1_{\et})_{D_\Nis},
}\]
where the vertical map on the right is induced by
\eq{para:CW1}{\sO_D[F]\to \varinjlim_V \uHom_{\Sh_k}(\sO_k, W_\bullet)\to 
\uHom_{\Sh_k}(\sO_k, CW/(F-1))\to \uHom_{\Sh_k}(\sO_k, H^1_{\et}).}
Assume  $(X,D)$ admits a projective SNC-compactification (see \ref{para:sm-cptf}).
By Zariski-Nagta purity \cite[Corollary 6.10]{RS-ZNP} and the computation of 
$\widetilde{H^1_{\et}}(\sO_L,\fm_L^{-r})$ in \cite[Theorem 8.8]{RS}, we see that 
the vertical map $\partial$ is a surjection of Nisnevich sheaves on $X$
and thus $\charr^{(rD)}_{H^1_{\et}}$ is completely determined by Corollary \ref{prop:charW}.
In particular, since \eqref{para:CW1}  factors via $\sO_D[F]\to \sO_D$, $F\mapsto 1$, 
it follows from  Theorem \ref{thm;AS} and Lemma \ref{lem:ker-psi} that $\charr^{(rD)}_{H^1_\et}$ induces a map
\eq{para:CW2}{\widetilde{H^1_{\et}}_{(X,rD)}/\widetilde{H^1_{\et}}_{(X,(r-1)D)}
\kay{\to}  \Omega^1_X(rD)\otimes_{\sO_X} i_*\sO_D.}
We obtain that $\charr^{(rD)}_{H^1_{\et}}$ (on the quotient) is a global version
of the characteristic form (alias refined Swan conductor) of Matsuda-Yatagawa, see \cite[Definition 3.2.5]{Matsuda}
and \cite[Definition 1.18]{Yatagawa}, which is the non-log version of Kato's refined Swan conductor defined in 
\cite[\S 5]{KatoSwan}. 
However, note that in the case $(p,r)=(2,2)$  and  
$(X, D)= (\Spec \sO_L, \fm_L)$ as in \ref{para:matsuda},
the map \eqref{para:CW2} differs by a square root  from the formula of Yatagawa's characteristic form.
Indeed, let $p=2$ and let $\chi\in \widetilde{H^1_{\et}}(\sO_L,\fm_L^{-2})$ have a representative 
$V^{n-1}(a)\in {\rm fil}_2W_n(L)$ ,
i.e., $a\in\fm_L^{-2}$. Then the characteristic form of $\chi$ defined  in \cite[Definition 1.18]{Yatagawa} is given by
\[da\otimes 1 + \tfrac{dz}{z^2}\otimes \sqrt{\alpha}\in \fm_L^{-2}\Omega^1_{\sO_L}\otimes_{\sO_L} K^{1/2},\]
where $K$ denotes the residue field of $\sO_L$, $K^{1/2}$ the field of square roots of elements of $K$ in an algebraic closure,
$z\in\fm_L$ is a local parameter, and $\alpha\in K$ is equal to $z^2a$ mod $\fm_L$. 
On the other hand the image of $\chi$ under the map \eqref{para:CW2} 
is given by 
\[da\otimes 1 + \tfrac{dz}{z^2}\otimes \alpha \in \fm_L^{-2}\Omega^1_{\sO_L}\otimes_{\sO_L} K,\]
as follows from Corollary \ref{prop:charW}\ref{prop:charW2} and the discussion preceding \eqref{para:CW2}.
\end{para}

\section{The characteristic form of differentials}\label{sec:ex-diff}

In this section we compute the characteristic form for the differential forms $\Omega^j$, 
$j\ge 1$, see Theorem \ref{prop:Omega-char}. As a consequence we obtain a negative answer
to a question posed in \cite{KMSY1}, see \ref{para:Q}. The computations will be used in the 
next section to give a new characterization of pseudo-rational singularities.

Throughout this section we assume $X\in \Sm$ with $\dim X=d$ and  $i:D\inj X$ is a smooth connected divisor.
Let ${\rm char}(k)=p\ge 0$ and $n\ge 1$. We write $\Omega^j_X$ for $\Omega^j_{X/k}$.

\begin{lem}\label{lem:mod-diag2}
Denote by $J^\bullet\subset \Omega^\bullet_{X\times X}$ the differential graded ideal  generated by $I_\Delta$ 
the ideal of the diagonal $\Delta: X\inj X\times X$ 
(i.e., $J^\bullet=I_\Delta \Omega^\bullet_{X\times X} + dI_{\Delta}\cdot \Omega^\bullet_{X\times X}$).
There is a short exact sequence  ($j\ge 0$)
\eq{lem:mod-diag20}{
0\to \frac{I_{\Delta}}{I^2_\Delta}\otimes_{\sO_X} \Omega^j_X \xr{a} \frac{J^j}{(J^\bullet\cdot J^{\bullet})^j}
\xr{b} \frac{I_{\Delta}}{I^2_\Delta}\otimes_{\sO_X} \Omega^{j-1}_X \to  0,}
where 
\[a(\theta\otimes \delta)= \theta p_1^*(\delta) \;{\rm  mod }\; (J^\bullet\cdot J^{\bullet})^j,\quad 
  b(\theta\e+\gamma \cdot d\theta )= \theta\otimes \Delta^*(\gamma),\]
where $\theta \in I_{\Delta}$, $p_1:X\times X\to X$ is the first projection, and $\delta\in \Omega^j_X$, 
$\gamma\in \Omega^{j-1}_{X\times X}$, $\e\in\Omega^j_{X\times X}$.

In particular, if $\theta_1, \ldots, \theta_d\in \sO_{X\times X}$ is a regular sequence generating $I_{\Delta}$, 
then for any $\beta\in J^{j}$ there exist {\em unique} forms $\gamma_i\in \Omega^{j-1}_X$ and $\delta_i\in\Omega^j_X$
such that 
\eq{lem:mod-diag21}{\beta= \sum_{i=1}^d p_1^*(\gamma_i)d\theta_i + \theta_i p_1^*(\delta_i)\quad 
\text{ in } \frac{\Omega^j_{X\times X}}{(J^\bullet\cdot J^\bullet)^j}.}
\end{lem}
\begin{proof}
Since $X$ is a smooth $k$-scheme we have a short exact sequence of locally free $\sO_X$-modules
\eq{lem:mod-diag22}{0\to \frac{I_\Delta}{I^2_{\Delta}}\xr{\bar{d}} 
\Delta^*\Omega^1_{X\times X}\xr{\Delta^*} \Omega^1_X\to 0.
}
Set
\[{\rm  fil}^r:={\rm image}\left(\bigwedge^{r}\frac{I_\Delta}{I^2_{\Delta}} \otimes \Delta^*\Omega^{j-r}_{X\times X}\to 
\Delta^*\Omega^{j}_{X\times X}\right), \quad 0\le r\le j.\]
Consider the diagram of $\sO_{X\times X}$-modules
\eq{lem:mod-diag22.5}{\xymatrix{
0\ar[r]& \frac{J^j}{(J^\bullet\cdot J^\bullet)^j}\ar[r]\ar[d] &
\frac{\Omega^j_{X\times X}}{(J^\bullet\cdot J^\bullet)^j}\ar[r]\ar[d] &
\Omega^j_X\ar[r]\ar@{=}[d] & 0\\
0\ar[r] & \frac{{\rm fil}^1}{{\rm fil}^2}\ar[r]^-{(*)} &
\frac{\Delta^*\Omega^j_{X\times X}}{{\rm fil}^2}\ar[r] & \Omega^j_X\ar[r] &0.
}}
The rows are clearly exact and  the vertical maps are the natural ones. 
The exact sequence \eqref{lem:mod-diag22} yields an isomorphism
\[\frac{I_\Delta}{I^2_{\Delta}} \otimes_{\sO_X} \Omega^{j-1}_{X}\xr{\simeq}\frac{{\rm fil}^1}{{\rm fil}^2}\]
and its composition with the map $(*)$ 
sends $\theta\otimes \gamma$ to $\tilde{\gamma}\cdot d\theta$, where $\gamma\in \Omega^{j-1}_X$
and $\tilde{\gamma}\in \Omega^{j-1}_{X\times X}$ is some lift along $\Delta$, in particular we may take
$\tilde{\gamma}=p_1^*\gamma$. 
We observe that 
\[\frac{\Delta^*\Omega^j_{X\times X}}{{\rm fil}^2}\cong 
\Delta^*\left(\frac{\Omega^j_{X\times X}}{(J^\bullet\cdot J^\bullet)^j}\right).\]
Hence the commutative diagram \eqref{lem:mod-diag22.5} gives a short exact sequence
\[0\to \frac{I_{\Delta}\Omega^j_{X\times X}}{(J^\bullet\cdot J^\bullet)^j}
\to \frac{J^j}{(J^\bullet\cdot J^{\bullet})^j}
\xr{b} \frac{I_{\Delta}}{I^2_\Delta}\otimes_{\sO_X} \Omega^{j-1}_X \to  0,\]
where $b$ is as in the statement.
Thus it remains to prove that the natural map
\eq{lem:mod-diag23}{\frac{I_\Delta}{I^2_\Delta}\otimes \Omega^j_X\to 
\frac{I_{\Delta}\Omega^j_{X\times X}}{(J^\bullet\cdot J^\bullet)^j}, \quad 
\theta\otimes \delta\mapsto \theta\cdot p_1^*(\delta) }
is an isomorphism. Since $\Delta$ is a section of $p_1$ we have a splitting
\eq{lem:mod-diag24}{\Omega^j_{X\times X}\cong p_1^*(\Omega^j_X)\oplus J^j,}
which yields the surjectivity of \eqref{lem:mod-diag23}.
Furthermore the composition
\[\frac{I_\Delta}{I^2_\Delta}\otimes \Omega^j_X\xr{\eqref{lem:mod-diag23}}
\frac{I_{\Delta}\Omega^j_{X\times X}}{(J^\bullet\cdot J^\bullet)^j}\xr{d} 
\frac{\Omega^{j+1}_{X\times X}}{(J^\bullet\cdot J^\bullet)^{j+1}}
\to \Delta^*\left(\frac{\Omega^{j+1}_{X\times X}}{(J^\bullet\cdot J^\bullet)^{j+1}}\right)
\cong \frac{\Delta^*\Omega^{j+1}_{X\times X}}{{\rm fil}^2} \]
is equal (up to sign) to the injection 
\[\frac{I_\Delta}{I^2_\Delta}\otimes \Omega^j_X\cong \frac{{\rm fil}^1}{{\rm fil}^2}\inj 
\frac{\Delta^*\Omega^{j+1}_{X\times X}}{{\rm fil}^2},\]
stemming from \eqref{lem:mod-diag22}. Hence \eqref{lem:mod-diag23} is injective as well.
This completes the proof of the exactness of \eqref{lem:mod-diag20}.

Finally, the existence of a presentation as in \eqref{lem:mod-diag21}  follows from 
the definition of $J^\bullet$ and the decomposition \eqref{lem:mod-diag24};
the uniqueness follows from the exact sequence using that $\theta_1,\ldots, \theta_d$
defines a basis of the $\sO_X$-module $I_\Delta/I^2_{\Delta}$.
\end{proof}

\begin{lem}\label{lem:non-van}
Assume $X\to \Spec k[z_1,\ldots, z_d]$ is \'etale.  Let $p_i: X\times X\to X$ be the projection to the $i$th factor
and set $t_i=p_1^*(z_i)$,  $s_i=p_2^*(z_i)$. Let $U\subset X\times X$ be an open subset, such that
the ideal $I_\Delta$ of the diagonal  restricted to $U$ is generated by $\theta_i:=s_i-t_i$, $i=1,\ldots, d$.
Let $j\ge 1$ and  $\beta\in \Omega^j(U)$ and write
\[p_2^*(\beta)-p_1^*(\beta)=\sum_{i=1}^d p_1^*(\gamma_i)d\theta_i + \theta_i p_1^*(\delta_i)\quad 
\text{ in } \Gamma(U,\Omega^j_{X\times X}/(J^\bullet\cdot J^\bullet)^j)\]
as in \eqref{lem:mod-diag21}. 
Then 
\[\beta=0 \Longleftrightarrow \gamma_i=0 \text{ for all }i=1,\ldots, d.\]
\end{lem}
\begin{proof}
By uniqueness of the $\gamma_i$ and $\delta_i$ the question is local.
Thus we may assume $X=\Spec A$ and $U=\Spec B$ with $B$ a localization of $A\otimes_k A$.
We can write $\beta$ uniquely as
\[\beta=\sum_{K} f_{K} dz_{K},\]
where the sum is over all $K=(1\le k_1< \ldots< k_j\le d)$, $f_{K}\in A$, and 
$dz_K= dz_{k_1}\cdots dz_{k_j}$.
Fix $K$ as above and write 
\[p_2^*(f_{K})= p_1^*(f_{K})+ e_{K} 
\quad \text{for some } e_{K}\in \Gamma(U,I_\Delta).\]
Thus in $\Omega^j_B/ (J^\bullet\cdot J^\bullet)^j$ we have 
\begin{align*}
p_2^*(f_K dz_K)-p_1^*(f_K dz_K) & = (p_1^*(f_K)+ e_K)d(t+\theta)_{K}  -  p_1^*(f_K) dt_K\\
                                                   &= e_K d t_K 
 +\sum_{\nu=1}^j (-1)^{j-\nu}p_1^*(f_K) dt_{K(\nu)}\cdot d\theta_{k_\nu},
\end{align*}
where $K(\nu)=(k_1,\ldots k_{\nu-1}, k_{\nu+1},\ldots, k_j)$. 
Thus 
\[\gamma_i dz_i= \sum_{\nu=1}^j\sum_{K \text{ with } k_\nu=i} (-1)^{j-\nu} f_K dz_{K(\nu)}dz_i
=\sum_{\nu=1}^j\sum_{K \text{ with } k_\nu=i}f_K dz_K.\] 
Hence $\gamma_i=0$ implies  $f_K=0$, for all $K$ which contain $i$.
This yields the non-trivial implication of the statement.
\end{proof}

\begin{para}\label{para:OOmega}
Recall from \ref{para:RegB} that $\uHom_{\Sh_k}(\sO_k,\Omega^j)_D=\uHom_{\Sh_D}(\sO_D,\Omega^j_{|\Sm_D})$ 
denotes the sheaf of maps of abelian sheaves $\sO\to \Omega^j$ on $\Sm_D$.
In what follows, we simply write 
$\uHom_{\Sh_D}(\sO_D,\Omega^j_{|\Sm_D})=\uHom_{\Sh_D}(\sO,\Omega^j)$.
It is equipped with the $\sO_D$-module structure given by $a\cdot f:= f(a\cdot -)$.
A  section $\beta\in \Gamma(D, \Omega^j)$ defines a section of 
$\uHom_{\Sh_D}(\sO,\Omega^j)$ over $D$
(which we will again denote by $\beta$) via
\[\Gamma(T,\sO)\to \Gamma(T,\Omega^j), \quad  a\mapsto af^*(\beta), \quad\text{for }f:T\to D  \text{ smooth},\]
which gives a morphism of $\sO_D$-modules 
$\Omega^j_D\to \uHom_{\Sh_D}(\sO,\Omega^j)$.
A section $\alpha\in \Gamma(D,\Omega^{j-1})$ defines a section $(\alpha\cdot d)$ of $\uHom_{\Sh_D}(\sO,\Omega^j)$ over $D$
given by
\[ \Gamma(T,\sO)\to \Gamma(T,\Omega^j_T), \quad b\mapsto f^*(\alpha)\cdot d b=: (\alpha \cdot d)(b), \quad
 \text{for }f:T\to D \text{ smooth}.\]
By the above, we obtain the subsheaf
 \[\Omega^j_D\oplus \Omega^{j-1}_D\cdot d\subset \uHom_{\Sh_D}(\sO,\Omega^j).\]
 If $p>0$, then we may precompose the above with arbitrary Frobenius powers and obtain 
 \[\bigoplus_{s\ge 0}\Omega^j_D\cdot {\rm Frob}^{s}\oplus \Omega^{j-1}_D\cdot d \subset \uHom_{\Sh_D}(\sO,\Omega^j).\]
By the $\sO_D$-module structure on $\uHom_{\Sh_D}(\sO,\Omega^j)$ defined above, we have 
\[x\cdot (\alpha\cdot d)=\alpha\cdot dx + (x\alpha\cdot d),\quad  \text{for }x\in \sO_D.\]
Using this, it is direct to check that the map
\eq{para:OOmega1}{ \xi: \Omega^j_D \oplus \Omega^{j-1}_D \to\uHom_{\Sh_D}(\sO,\Omega^j),}
\[\xi(\beta,\alpha):=  (\beta + d\alpha)+ (-1)^{j-1} (\alpha\cdot d).\]
is an injective morphism of $\sO_D$-modules. We denote its image by 
\eq{para:OOmega2}{ \Xi:=\xi(\Omega^j_D \oplus \Omega^{j-1}_D)\subset \uHom_{\Sh_D}(\sO,\Omega^j),}
which is an $\sO_D$-submodule. 
\end{para}

\begin{para}\label{para:kosz}
Recall that for any $\sO_X$-module  $\sH$ and any $j\ge 1$ we have an $\sO_X$-linear map
\eq{para:kosz1}{\bigwedge^j \sH\to \sH\otimes_{\sO_X}(\bigwedge^{j-1}\sH)}
given by
\[h_1\wedge\ldots\wedge h_j\mapsto \sum_{i=1}^j (-1)^{i-1}h_i\otimes h_1\wedge\ldots \wedge\widehat{h_i}
\wedge \ldots \wedge h_j,\]
where the hat $\widehat{\phantom{-}}$ indicates omission.
Tensoring \eqref{para:kosz1} for $\sH=\Omega^1_X$ with $\sO_X(nD)$ and composing with the map induced 
by $i^*: \Omega^{j-1}_X\to i_*\Omega^{j-1}_D$ yields the morphism of $\sO_{X}$-modules
\eq{para:kosz2}{
\partial_n^{j}: \Omega^j_X(nD)\to \Omega^1_X(nD)\otimes i_*\Omega^{j-1}_D.}
\end{para}

\begin{para}\label{para:CDX}
Let $I_D\subset \sO_X$ be the ideal sheaf of $i: D\inj X$ and set 
\eq{para:resnD1}{
C_{D/X}(nD):=\frac{I_D}{I_D^2}\otimes_{\sO_X}\sO_X(nD). }
Note that $C_{D/X}(nD)\cong \sO_X((n-1)D)_{|D}$ and  we have an exact sequence of $\sO_D$-modules
\eq{para:resnD1.5}{0\to C_{D/X}(nD)\xr{\bar{d}}  \Omega^1_X(nD)_{|D}\to \Omega^1_{D}(nD)\to 0,}
where 
$\Omega^1_{D}(nD) =\Omega^1_{D}\otimes_{\sO_D}\sO_{X}(nD)_{|D}$
and if $z$ is a local equation for $D$ with its class $\bar{z}\in I_D/I_D^2$, then 
$\bar{d}(\bar{z}\otimes a/z^n )= (adz/z^n)_{|D}$ for $a\in \sO_X$. 
In what follows, we view $C_{D/X}(nD)$ as a sub-$\sO_D$-module of 
$\Omega^1_X(nD)_{|D}$ via $\bar{d}$.  

\end{para}

\begin{thm}\label{prop:Omega-char}
Denote by $\FAS_{(X,nD)}$ the sheaf on $X_{\Nis}$ given by $V\mapsto \FAS(V, nD_{|V})$, for $F=\Omega^j$, see Definition \ref{defn:FAS}. 
Let $z\in \sO_X$ be a local equation of $D$. 
\begin{enumerate}[label=(\arabic*)]
\item\label{prop:Omega-charI}
We have 
\eq{prop:Omega-char1}{\FAS_{(X,nD)}=
\begin{cases} \Omega^j_X(\log D)((n-1)D) & \text{if } p=0 \text{ or } (p\neq 0 \text{ and } p\nmid n)\\
                     \Omega^j_X(nD)& \text{if } p\neq 0 \text{ and } p\mid n.   \end{cases}}
\item\label{prop:Omega-charII}
For $n\ge 2$  the characteristic form 
\[{\rm char}^{(nD)}_F: \FAS_{(X,nD)}\lra 
\Omega^1_X(nD)_{|D}\otimes_{\sO_D} \uHom_{\Sh_D}(\sO,\Omega^j)\]
is  given as follows (using the notation \ref{para:OOmega} and \ref{para:kosz})
\[\charr^{(nD)}_F(\omega)=
\begin{cases}
(\id\otimes\xi)(\partial^{j+1}_n(d\omega), \partial^j_n(\omega)) & \text{if } (n,p)\neq (2,2),\\
(\id\otimes\xi)(\partial^{j+1}_2(d\omega), \partial^j_2(\omega))+ \frac{dz}{z^2}\otimes (\ol{z^2\omega}\cdot {\rm Frob}) &
    \text{if } (n,p)=(2,2),
\end{cases}\]
where $\id$ is the identity on $\Omega^1_X(nD)$, $\xi$ is from \eqref{para:OOmega1}, $ \partial^j_n$ is from \eqref{para:kosz2},
$\omega$ is a local section of $\FAS_{(X,nD)}$. 
By \eqref{prop:Omega-char1}, we have $d\omega\in \Omega^{j+1}_X(nD)$,
and in the case $(n,p)=(2,2)$ the form $z^2\omega$ is regular along $D$ and 
$\ol{z^2\omega}$ denotes its image in $\Omega^{j}_D$. 

\item\label{prop:Omega-charIII} The characteristic form factors as follows, where  for $\alpha\in \Omega^s_X$, we write $\ol{\alpha}$ for its image in $\Omega^s_D$, and for the explicit description in \ref{prop:Omega-chariv} and
\ref{prop:Omega-charv}, we assume that $X$ is \'etale over $\Spec k[z, z_1,\ldots, z_{d-1}]$ and $D=\Div(z)$: 
\begin{enumerate}[label=(\roman*)]
\item\label{prop:Omega-charii}
if $p=0$ or $p>0$ and $p\nmid n(n-1)$, then $\charr^{(nD)}_F$  factors via the isomorphism
\[\frac{\Omega^j_X(\log D)((n-1)D)}{\Omega^j_X(\log D)((n-2)D)} \stackrel{\simeq}{\lra}
C_{D/X}(nD)\otimes_{\sO_D} \Xi,\]
 given by 
\[\charr^{(nD)}_F(\alpha\frac{dz}{z^n} + \frac{1}{z^{n-1}}\beta)
=\frac{dz}{z^n}\otimes \left(-(n-1)\bar{\beta} + (\bar{\alpha}\cdot d)\right),\quad \text{for }\alpha\in 
\Omega^j_X, \,\beta \in \Omega^{j-1}_X;\]
\item\label{prop:Omega-chariii} if $p>0$ and $p|n-1$, then $\charr^{(nD)}_F$ factors via the isomorphism
\[\charr^{(nD)}_F: \frac{\Omega^j_X(\log D)((n-1)D)}{\Omega^j_X((n-1)D)}\stackrel{\simeq}{\lra }
C_{D/X}(nD)\otimes_{\sO_D} (\Omega^{j-1}_D\cdot d),\]
 given by
\[\charr^{(nD)}_F(\alpha\frac{dz}{z^n})=\frac{dz}{z^n}\otimes (\bar{\alpha}\cdot d), \quad \text{for }\alpha\in \Omega^{j-1}_X,\]
where the $\sO_D$-module structure of $(\Omega^{j-1}_D\cdot d)$ is induced by the one of $\Omega^{j-1}_D$;
\item\label{prop:Omega-chariv} if $p>0$, $p|n$, and $n\neq 2$, then $\charr^{(nD)}_F$ factors via an injection 
\[\charr^{(nD)}_F: \frac{\Omega^j_X(nD)}{\Omega^j_X(\log D)((n-2)D)}\lra  
\Omega^1_X(nD)_{|D}\otimes_{\sO_D}\Xi,\]
which on the elements $\alpha\frac{dz}{z^n}+ \frac{1}{z^{n-1}}\beta$ with
$\alpha\in \Omega^{j-1}_X$ and $\beta\in \Omega^j_X$, is given by the same formula as in \ref{prop:Omega-charii}.  Furthermore, for $f\in \sO_X$ and $1\le \nu_1<\ldots<\nu_j\le d-1$, we have
\mlnl{\charr^{(nD)}_F(\frac{1}{z^{n}}fdz_{\nu_1}\cdots dz_{\nu_j})=\\
\frac{df}{z^n}\otimes d\bar{z}_{\nu_1}\cdots d\bar{z}_{\nu_j}+ 
\sum_{i=1}^j(-1)^{j-i}  \frac{dz_{\nu_i}}{z^n}\otimes (\ol{f}d\bar{z}_{\nu_1}\cdots \widehat{d\bar{z}_{\nu_i}}\cdots d\bar{z}_{\nu_j}\cdot d);}
\item\label{prop:Omega-charv} if $n=p=2$, then $\charr^{(2D)}_F$ factors via an injection
\[\charr^{(2D)}_F: \frac{\Omega^j_X(2D)}{\Omega^j_X(\log D)}\lra  \Omega^1_X(2D)_{|D}\otimes (\Xi\oplus \Omega^{j}_D\cdot {\rm Frob}),\]
which on the elements $\alpha\frac{dz}{z^2}+ \frac{1}{z}\beta$ with
$\alpha\in \Omega^{j-1}_X$ and $\beta\in \Omega^j_X$, is given by the same formula as in \ref{prop:Omega-charii}.  Furthermore for $f\in \sO_X$ and $1\le \nu_1<\ldots<\nu_j\le d-1$, we have 
\mlnl{
\charr^{(2D)}_F(\frac{1}{z^{2}}fdz_{\nu_1}\cdots dz_{\nu_j})  =\\
  \frac{df}{z^2}\otimes d\bar{z}_{\nu_1}\cdots d\bar{z}_{\nu_j} + 
\sum_{i=1}^j(-1)^{j-i} \frac{dz_{\nu_i}}{z^2}\otimes (\bar{f}d\bar{z}_{\nu_1}\cdots \widehat{d\bar{z}_{\nu_i}}\cdots d\bar{z}_{\nu_j}\cdot d) \\+\frac{dz}{z^2}\otimes (\bar{f} d\bar{z}_{\nu_1}\cdots d\bar{z}_{\nu_j}\cdot {\rm Frob}).
}
\end{enumerate}
\end{enumerate}
\end{thm}

\def\Ab{{\overline{A}}}

\begin{proof}
We start by showing \eqref{prop:Omega-char1}, i.e.,
\mlnl{\{a\in\Omega^j(X\setminus D)\mid p_2^*a-p_1^*a\in \Omega^j(P^{(nD)}_X)\}\\
=
\begin{cases}
\Gamma(X, \Omega^j_X(\log D)((n-1) D))& \text{if } p=0 \text{ or } (p\neq 0 \text{ and } p\nmid n)\\
\Gamma(X, \Omega^j_X(nD))& \text{if } p\neq 0 \text{ and } p\mid n.
\end{cases}}
This is a local question and we can assume $X=\Spec A$ and $D=\Spec \Ab$ with $\Ab=A/zA$.
Shrinking $D$ further we can assume that we have an \'etale map 
\[k[z_1, \ldots, z_{d-1}]\to \Ab,\]
where $d=\dim X$.
Up to replacing  $X$ by a Nisnevich neighborhood of $D$ and we may furthermore
assume that we have a map $X\to D$ such that $D\inj X\to D$ is the identity, e.g. \cite[Lemma 7.13]{BRS}. 
The induced map $\Ab\to A$ identifies the completion $\hat{A}$ of $A$ along $zA$ with $\Ab[[z]]$.
Since $\Omega^j_{A[1/z]}/\Omega^j_A$ only depends on $\hat{A}$ it suffices to show the following:
for any $q,r\ge 1$, and $\alpha_1,\ldots, \alpha_q\in \Omega^{j-1}_\Ab$, and 
$\beta_1,\ldots, \beta_r\in \Omega^j_\Ab$ with $\alpha_q\neq 0$ and $\beta_r\neq 0$ set 
\[\varphi_q:=\sum_{i=1}^q\alpha_i \frac{dz}{z^i},\quad \psi_r:=\sum_{i=1}^r \frac{1}{z^i}\beta_i\]
and 
\[\Phi_q=p_2^*(\varphi_q)-p_1^*(\varphi_q),\quad \Psi_r:= p_2^*(\psi_r)-p_1^*(\psi_r),\]
where $p_i: P^{(nD)}_X\to X$ are the projections from \ref{para:dil} for $R=nD$. 
Then we have to show
\eq{prop:Omega-char2}{\Phi_q+\Psi_r\in \Omega^j(P^{(nD)}_X) \Longleftrightarrow
\begin{cases} q\le n \text{ and } r\le n-1 & \text{if } p=0 \text{ or } (p> 0 \text{ and } p\nmid n)\\
                    q,r\le n & \text{if } p> 0 \text{ and } p\mid n.   \end{cases}}
We have a natural map $P^{(nD)}_X\to X\times X$  and it suffices to check  \eqref{prop:Omega-char2}
locally around the closed points of 
$D\inj X \overset{\Delta_X}{\to} X\times X$.
Let $B$ be the localization of $A\otimes_k A$ at such a closed point and set
\[t:= z\otimes 1,\quad  t_i:=z_i\otimes 1, \quad  s:=1\otimes z, \quad s_i:=1\otimes z_i\in B.\] 
Let $S$ be the localization of $\Ab\otimes_k \Ab$ at the same closed point (viewed as a point in $D\subset D\times D$).
Let $\theta_i=s_i-t_i\in S$, $i=1,\ldots, d-1$; it  is a regular sequence generating the ideal $I$ of the diagonal $D\inj D\times D$ in $S$.
Similarly,  $I_1=(s-t, \theta_1,\ldots, \theta_{d-1})B$  is equal to the ideal of the diagonal 
$X\inj X\times X$ in $B$.  Now the pullback of $P^{(nD)}_X\to X\times X$ along $\Spec B\to X\times X$ is equal to
$\Spec C$ with 
\eq{prop:Omega-char4}{
C= \frac{B[\tau, \tau_1,\ldots, \tau_{d-1}][\frac{1}{1+\tau t^{n-1}}]}{(s-t-t^n\tau, \theta_1- t^n\tau_1, \ldots,
\theta_{d-1}- t^n\tau_{d-1})}.}
It suffices to show \eqref{prop:Omega-char2}  with $\Omega^j(P^{(nD)})$ replaced by 
$\Omega^j_C$. 
Set  
\[u:=(1+\tau t^{n-1})^{-1}\in C^\times.\]
We have in $C$
\[u=\frac{1}{1+\tau t^{n-1}}= 1-\tau t^{n-1}u=1-\tau t^{n-1}+ \tau^2 t^{2n-2}u,\quad s=t+t^n\tau=tu^{-1}.\]

We first consider the differentials $\alpha_i\frac{dz}{z^i}$.
As in Lemma \ref{lem:mod-diag2} denote by 
$J^\bullet\subset \Omega^\bullet_S$ the differential graded ideal generated by 
$I=(\theta_1,\ldots, \theta_{d-1})$.
We have in $\Omega^{j-1}_S$
\[q_2^*(\alpha_i)=q_1^*(\alpha_i)+\e_i,\quad \e_i\in J^{j-1},\]
where $q_i:D\times D \to D$ are the projections. 
We compute in $\Omega^j_{C[1/t]}$ (for $n\ge 1$)
\begin{align}
p_2^*(\alpha_i \frac{dz}{z^i})-p_1^*(\alpha_i \frac{dz}{z^i}) =&
p_2^*(\alpha_i) \frac{ds}{s^i}-p_1^*(\alpha_i) \frac{dt}{t^i}\label{prop:Omega-char6}\\
= &\frac{1}{t^{i-1}}u^{i-1}(p_1^*(\alpha_i)+\e_i)\left(\frac{dt}{t}+\frac{du^{-1}}{u^{-1}}\right)
-p_1^*(\alpha_i) \frac{dt}{t^i}\notag\\
 = &  p_1^*(\alpha_i) \frac{dt}{t^i}\left(u^i(1+ n\tau t^{n-1}) -1\right) 
+t^{n-i}u^ip_1^*(\alpha_i)d\tau \notag\\
& +\underbrace{\frac{1}{t^{i-1}} u^{i-1}\e_i\left(\frac{dt}{t}+\frac{du^{-1}}{u^{-1}}\right)}_{:=\e_{1,i}},\notag
\end{align}
where $\e_i$ are viewed as elements of $\Omega_C^{j-1}$ via the map $S \to B$ induced by $\Ab\to A$ and  we used $d u^{-1}= t^{n-1}(d\tau+(n-1)\tau dt/t)$.
We observe 
\begin{enumerate}[label=(\alph*)]
\item\label{prop:Omega-char7} $u^i(1+ n\tau t^{n-1}) -1= (n-i)\tau t^{n-1}+ t^{2n-2}c_i$, for some $c_i\in C$ 
(with $c_i=0$ if $n=i=1$);
\item\label{prop:Omega-char8} $\e_{1,i}\in t^{n-i}\Omega^{j-1}_C \cdot dt+ t^{2n-i}\Omega^j_C$;
\item\label{prop:Omega-char9} let $\bar{C}=C/(\tau_1,\ldots, \tau_{d-1})$, then $\e_{1,i}\mapsto 0$ 
under the natural map $\Omega^j_{C[1/t]}\to \Omega^j_{\bar{C}[1/t]}$.

\end{enumerate}
Here \ref{prop:Omega-char8} and \ref{prop:Omega-char9} follow from the fact that
the image of $J^\bullet$ in $\Omega^\bullet_C$  lies in the graded ideal generated by
$t^n\tau_i$ and $d(t^n\tau_i)$, $i=1,\ldots, d-1$.
From \ref{prop:Omega-char7}, \ref{prop:Omega-char8}, and \eqref{prop:Omega-char6} we find
\eq{prop:Omega-char15}{\Phi_q\in\Omega^{j}_C\Longleftarrow q\le n.}

Next we consider the differentials $\frac{1}{z^{i}}\beta_i$.
By Lemma \ref{lem:non-van} we have in $\Omega^{j}_S$
\[q_2^*(\beta_i)- q_1^*(\beta_i)= \e'_{i},\] 
where $q_i:D\times D \to D$ are the projections and 
\[\e'_{i}\equiv \sum_{\nu=1}^{d-1}   q_1^*(\gamma_{i,\nu})\cdot d\theta_\nu+ \theta_\nu \cdot q_1^{*}(\delta_{i,\nu})
\quad \text{mod } (J^\bullet\cdot J^\bullet)^j,\]  
for certain $\gamma_{i,\nu}\in \Omega^{j-1}_\Ab$ and $\delta_{i,\nu}\in \Omega^{j}_\Ab$, which satisfy
\eq{prop:Omega-char9.3}{\gamma_{i,\nu}=0,  \text{ for all }\nu\Longleftrightarrow 
\beta_i=0.}
We compute in $\Omega^j_{C[1/t]}$ (using the notation from above)
\begin{align}
p_2^*(\frac{1}{z^i} \beta_i)- p_1^*(\frac{1}{z^i} \beta_i) =& 
                 \frac{1}{s^i}p_2^*(\beta_i)- \frac{1}{t^i} p_1^*(\beta_i)\label{prop:Omega-char9.5}\\
=&  \frac{1}{t^i} p_1^*(\beta_i)(u^i-1) + \underbrace{\frac{1}{t^i}u^i\e'_{i}}_{=:\e_{2,i}}.\notag
\end{align}
We observe
\begin{enumerate}[resume*]
\item\label{prop:Omega-char72} $u^i-1= -i\tau t^{n-1} + t^{2n-2}e_i$, for some $e_i\in C$ (with $e_i=\tau u$ if $i=1$);
\item\label{prop:Omega-char82} $\e_{2,i}\in nt^{n-1-i}\Omega_C^{j-1}dt+ t^{n-i}\Omega^{j}_C$;
\item\label{prop:Omega-char92} $\e_{2,i}\mapsto 0$ under $\Omega^{j}_{C[1/t]}\to \Omega^j_{\bar{C}[1/t]}$.
\end{enumerate}
Here \ref{prop:Omega-char82} and \ref{prop:Omega-char92} follow from $\theta_i=t^n \tau_i$.
We find
\[\Psi_r\in \Omega^j_C \Longleftarrow
\begin{cases}  r\le n-1 & \text{if } p=0 \text{ or } (p> 0 \text{ and } p\nmid n)\\
                    r\le n & \text{if } p> 0 \text{ and } p\mid n.   \end{cases}\]
Thus we have shown this ``$\Longleftarrow$'' direction in \eqref{prop:Omega-char2}.

We show the other direction.
Observe that $\bar{C}=C/(\tau_1,\ldots, \tau_{d-1})$ is essentially \'etale over 
\[\Ab[s,t,\tau]/(s-t-t^n\tau)\cong \Ab[t,\tau]\]
and that the image  of $\Phi_q+\Psi_r$ in 
$\Omega^j_{\bar{C}[1/t]}=\bar{C}[1/t]\otimes_{\Ab[t,\tau]}\Omega^j_{\Ab[t,\tau]}$ 
is by \eqref{prop:Omega-char6}, \ref{prop:Omega-char9}, \eqref{prop:Omega-char9.5} and \ref{prop:Omega-char92},  equal to
\[\ol{\Phi_q+\Psi_r}= \sum_{i=1}^q \left( \alpha_i \frac{dt}{t^i}\left(\bar{u}^i(1+ n\tau t^{n-1}) -1\right) +
t^{n-i}\bar{u}^i \alpha_i d\tau\right)  +\sum_{i=1}^r \frac{1}{t^i}\beta_i(\bar{u}^i-1),\]
where $\bar{u}\in \bar{C}^\times$ is the image of $u$.
Considering the coefficient in front of $d\tau$, we see that $\ol{\Phi_q+\Psi_r}\in \Omega^j_{\bar{C}}$ 
is only possible if $q\le n$, which by \eqref{prop:Omega-char15} implies that $\Phi_q$ is regular at $t$. Hence
$\Phi_q+\Psi_r\in \Omega^j_C$ implies  $q\le n$ and $\Psi_r\in \Omega^j_C$.

We claim that $\Psi_r$ can be regular at $t$ only if $r\le n$. 
Since $\beta_r\neq 0$  we find  $\gamma_{r,\nu_0}\neq 0$, for some $\nu_0$, by \eqref{prop:Omega-char9.3}; 
up to renumbering we can assume $\nu_0=1$.
The image  of $\tau_1^{-1}\Psi_r$ under the composition
\[\Omega^{j}_{C[1/\tau_1 t]}\xr{\Res_{\tau_1}} \Omega^{j-1}_{(C/\tau_1)[1/t]}\xr{\rm nat.} 
\Omega^{j-1}_{\bar{C}[1/t]}=\bar{C}[1/t]\otimes_{\Ab[t,\tau]} \Omega^j_{\Ab[t,\tau]} \]
is equal to 
\[\sum_{i=1}^r \frac{1}{t^i}\bar{u}^i t^n \gamma_{i,1}.\]
Since $0\neq \gamma_{r,1}\in \Omega^{j-1}_\Ab\subset \Omega^{j-1}_{\Ab[t,\tau]}$ this last expression can be regular at $t$ only if $r\le n$. 

It remains to show that $\Psi_r$ cannot be regular at $t$ for $r=n$ in the case $p=0$ or ($p> 0$ and $p\nmid n$). 
Similar as above, this follows immediately by considering the image of 
$\Psi_r$ in $\Omega^j_{\bar{C}[1/t]}$ and \ref{prop:Omega-char72}, \ref{prop:Omega-char92}.
This completes the proof of this ``$\Longrightarrow$'' direction in \eqref{prop:Omega-char2}
and hence the computation of $\FAS_{(X,D)}$. 

\medskip

Next we  compute the explicit description of the characteristic form given in \ref{prop:Omega-charIII}.
Note that the statement about the maps being isomorphisms (resp. injections) follows directly from this.
First we give an explicit description of the map \eqref{eq;HomSmB} 
\[\chi_{\sO, F}: F_{\rm ad}(E)\to \sE^\vee\otimes_{\sO_D} \uHom_{\Sh_D}(\sO_D, F_{|\Sm_D}),\]
where  $E=\Spec (\Sym^* \sE^\vee)$ is a  trivial vector bundle of rank $d$ over $D=\Spec \Ab$ with the sheaf of sections $\sE=\oplus_{i=1}^d \sO_D \cdot e_i$ 
and $F=\Omega^{j}$. It is direct to check from the definition that we have 
\[ \bigoplus_{i} e_i^\vee \Omega^j_{\Ab} \oplus \bigoplus_i \Omega^{j-1}_\Ab d e_i^\vee\subset F_{\rm ad}(E),\]
where $\{e_i^\vee\}$ is the dual basis of $\{e_i\}$.
Let $\gamma= \sum_i e_i^\vee \beta_i + \sum_i \alpha_i d e_i^\vee\in F_{\rm ad}(E)$,  
where $\beta_i\in \Omega^j_{\Ab}$, $\alpha_i\in \Omega^{j-1}_\Ab$.
Then the map \eqref{para:Fad3} $\chi_F:F_{\rm ad}(E)\to \Hom_{\Sh_D}(\sE, F_{|\Sm_D})$ maps $\gamma$ to
the map given by
\[\chi_F(\gamma)(\sum_i c_i e_i)= \sum_i c_i f^*\beta_i +\sum_i f^*\alpha_i d c_i,\]
where $f: T\to D$ is smooth and $c_i\in \sO(T)$. 
By \eqref{para:chiFO1} we find
\eq{prop:Omega-char10}{\chi_{\sO, F}(\gamma)= \sum_i (e_i^\vee \otimes (\beta_i+ (\alpha_i\cdot d)),}
using the notation from \ref{para:OOmega}.  If $p>0$, we furthermore have 
${\rm Frob}^s(e_i^\vee)\Omega^j_\Ab\subset F_{\ad}(E)$ and similarly we find
\eq{prop:Omega-char10Frob}{\chi_{\sO,F}({\rm Frob}^s(e_i^\vee)\beta)= e_i^{\vee}\otimes(\beta\cdot {\rm Frob}^s).}

To check the formulas for the characteristic form given in \ref{prop:Omega-charii} - \ref{prop:Omega-charv},
we may consider the above local situation. 
In particular the vector bundle $P^{(nD)}_X\times_X D$ is equal to the spectrum of 
\eq{prop:Omega-char5}{C/tC\cong \Ab[\tau,\tau_1,\ldots, \tau_{d-1}].
}
The isomorphism 
\eq{prop:Omega-char11}{\oplus_{i=1}^d \sO_{D} \tau_i \xr{\simeq} \Omega^1_X(nD)_{|D}, \quad 
\tau_i\mapsto \frac{d z_i}{z^n}, }
where $z_d:=z$ and $\tau_d:=\tau$, yields the identification $P^{(nD)}_X\times_X D=\mathbb{V}(\Omega^1_X(nD)_{|D})$, see \eqref{lem:dil7} (Thus with the notation from above $\sE^\vee= \Omega^1_X(nD)_{|D}$ and  $e_i^\vee=\tau_i$).
Writing $\FAS(A, t^m)=\FAS(\Spec(A),\Spec(A/ t^m))$, 
we have to compute the map
\[\psi_{nD}: \FAS(A, t^n)/\FAS(A, t^{n-1})\xr{p_2^*-p_1^*}F(C)\to F(C/tC)=\Omega^j_{\Ab[\tau, \tau_1,\ldots, \tau_{d-1}]}.\]

Consider the case $n\ge 2$ and either $p=0$ or $p>0$ and $p\nmid n(n-1)$.
Any element in $\Omega^j_A(\log z)\frac{1}{z^{n-1}}/ \Omega^j_A(\log z)\frac{1}{z^{n-2}}$
can be written in the form 
\[\alpha\frac{dz}{z^{n}} + \frac{1}{z^{n-1}}\beta, \quad\text{with } \alpha\in \Omega^{j-1}_\Ab, \beta\in \Omega^j_\Ab.\]
By \eqref{prop:Omega-char6}, \ref{prop:Omega-char7}, \ref{prop:Omega-char8},
and \eqref{prop:Omega-char9.5}, \ref{prop:Omega-char72}, \ref{prop:Omega-char82}, 
it is mapped under $\psi_{nD}$ to
\[\alpha d\tau - (n-1)\beta \tau.\]
Together with \eqref{prop:Omega-char10} and \eqref{prop:Omega-char11}, this yields \ref{prop:Omega-charii}.
The proof of \ref{prop:Omega-chariii} is similar.
Finally, we prove \ref{prop:Omega-chariv} and \ref{prop:Omega-charv}. 
We assume $p|n$.  It remains to compute the characteristic form of
\[\frac{1}{z^n}\beta,\quad \text{where }\beta=f dz_{\nu_1}\cdots dz_{\nu_j}, \quad\text{with } f\in A  
\text{ and } 1\le \nu_1<\ldots< \nu_j\le d-1.\]
Write $f=f_0+zg$ with $f_0\in \Ab$ and $g\in A$ and set 
\[\beta_0:=f_0 dz_{\nu_1}\cdots dz_{\nu_j}\in \Omega^j_\Ab\quad \text{and}\quad
\beta_1:=gdz_{\nu_1}\cdots dz_{\nu_j}\in \Omega^j_A,\]
so that 
\[\frac{\beta}{z^n}=\frac{\beta_0}{z^n}+\frac{\beta_1}{z^{n-1}}.\]
By Lemma \ref{lem:non-van}, letting $q_i:D\times D \to D$ be the projections,
we can write 
\[q_2^*(\beta_0)- q_1^*(\beta_0)\equiv  \sum_{i=1}^{d-1}   q_1^*(\gamma_i)\cdot d\theta_i
+ \theta_i \cdot q_1^{*}(\delta_i)
\quad \text{mod } (J^\bullet\cdot J^\bullet)^j\]  
with $\gamma_i\in \Omega^{j-1}_\Ab$ and $\delta_i\in \Omega^j_\Ab$.
By   \eqref{prop:Omega-char9.5} and the above, we find
\[\psi_{nD}(\frac{1}{z^n}\beta)=\ol{\beta_1}\tau+
\begin{cases}
 \sum_{i=1}^{d-1}   (\gamma_{i}\cdot d\tau_i+ \tau_i \delta_i) & \text{if }n> 2\\
\tau^2\beta_0+\sum_{i=1}^{d-1}   (\gamma_{i}\cdot d\tau_i+ \tau_i \delta_i)& \text{if } n=2.  
\end{cases}\]
By \eqref{prop:Omega-char10}, \eqref{prop:Omega-char10Frob} and  \eqref{prop:Omega-char11}, we find 
\eq{prop:Omega-char12}{{\rm char}^{(nD)}_F(\frac{1}{z^n}\beta)= \frac{dz}{z^n} \otimes \ol{\beta_1} +
\begin{cases}
 \sum_{i=1}^{d-1}   \frac{d z_i}{z^n}\otimes (\delta_i+ (\gamma_{i}\cdot d)) & \text{if }n> 2\\
\frac{dz}{z^2}\otimes (\beta_0\cdot{\rm Frob})+
\sum_{i=1}^{d-1} \frac{d z_i}{z^n}\otimes (\delta_i+ (\gamma_{i}\cdot d))& \text{if } n=2.  
\end{cases}}
To determine  $\gamma_i$ and $\delta_i$, we write in 
$\Omega^1_\Ab = I/I^2$ with $I\subset \Ab\otimes_k\Ab$ the  diagonal ideal
\[df_0=1\otimes f_0-f_0\otimes 1=\sum_{i=1}^{d-1} \frac{\partial f_0}{\partial z_i} dz_i
=\sum_{i=1}^{d-1} \frac{\partial f_0}{\partial z_i} \theta_i,\]
where $\theta_i=s_i-t_i$ with $t_i=z_i\otimes 1$ and $s_i=1\otimes z_i$, 
We obtain 
\[q_2^*(\beta_0)- q_1^*(\beta_0)= \left(\sum_{i=1}^{d-1} \frac{\partial f_0}{\partial z_i} \theta_i +q_1^*(f_0)\right)
 d(\theta_{\nu_1}+t_{\nu_1})\cdots d(\theta_{\nu_j}+t_{\nu_j}) -
q_1^*(f_0)dt_{\nu_1}\cdots dt_{\nu_j}.\] 
By the last statement of Lemma \ref{lem:non-van},
we get $\gamma_i=0$ for $i\not\in \{\nu_1,\ldots, \nu_j\}$,
\[\gamma_{\nu_{s}}= (-1)^{j-s}f_0 dz_{\nu_1}\cdots \widehat{dz_{\nu_s}}\cdots dz_{\nu_j},\quad s=1,\ldots, j,\]
\[\delta_i= \frac{\partial f_0}{\partial z_i} dz_{\nu_1}\cdots dz_{\nu_j}, \quad i=1,\ldots, d-1.\]
Using $df=df_0+gdz+z dg$, we obtain in $\Omega^1_X(nD)\otimes\Xi$
\[\frac{dz}{z^n}\otimes \ol{\beta_1}+\sum_{i=1}^{d-1}\frac{dz_i}{z^n}\otimes\delta_i= 
\left(\frac{gdz}{z^n}
+\sum_{i=1}^{d-1}\frac{(\partial f_0/\partial z_i) dz_i}{z^n}\right)\otimes d\bar{z}_{\nu_1}\cdots d\bar{z}_{\nu_j}
= \frac{df}{z^n}\otimes d\bar{z}_{\nu_1}\cdots d\bar{z}_{\nu_j}.\]
Thus the formulas in \ref{prop:Omega-chariv} and \ref{prop:Omega-charv} follow from \eqref{prop:Omega-char12}.

\medskip

Finally, we prove \ref{prop:Omega-charII} by computing $\Gamma:=(\id\otimes\xi)(\partial^{j+1}_n(d(-)), \partial^j_n(-))$
and comparing it with the formulas from \ref{prop:Omega-charIII}.
For $\alpha\in \Omega^{j-1}_X$ and $i\le n$ we have 
\[\Gamma(\alpha\frac{dz}{z^i})=(\id\otimes\xi)\left( \frac{dz}{z^n }\otimes( (-1)^j\bar{z}^{n-i}d\bar{\alpha}, (-1)^{j-1}\ol{z^{n-i}\alpha})\right)=
\begin{cases}
\frac{dz}{z^n}\otimes (\bar{\alpha}\cdot d) & \text{if } i=n\\
0 &\text i< n.
\end{cases}
\]
For $\beta\in \Omega^j_X$ and $i\le n-1$ we have 
\[\Gamma(\frac{1}{z^i}\beta)= (\id\otimes \xi)\left(\frac{dz}{z^n}\otimes (-i\ol{z^{n-i-1}\beta}, 0)\right)
=\begin{cases}
-(n-1)\frac{dz}{z^n}\otimes \ol{\beta} &\text{if }i=n-1\\
0   &\text{if }i<n-1.
\end{cases}\]
Assume $p>0$ and $p|n$, and consider
\[\beta= fdz_{\nu_1}\cdots dz_{\nu_j}, \quad \text{with } 1\le \nu_1<\ldots<\nu_j\le d-1, \, f\in A.\]
In this case we have 
\begin{align*}
\partial^{j+1}_n\left(d(\beta\frac{1}{z^n})\right) & =\partial^{j+1}_n\left(\frac{df dz_{\nu_1}\cdots dz_{\nu_j}}{z^n}\right)\\
                                                                        &=\frac{df}{z^n}\otimes d\bar{z}_{\nu_1}\cdots d\bar{z}_{\nu_j}  
-\sum_{i=1}^{j-1} (-1)^{i-1} \frac{dz_{\nu_i}}{z^n}\otimes d\bar{f} d\bar{z}_{\nu_1}\cdots\widehat{d\bar{z}_{\nu_i}}\cdots d\bar{z}_{\nu_j}
\end{align*}
and 
\[\partial^{j+1}_n(\beta\frac{1}{z^n})=
\sum_{i=1}^{j-1} (-1)^{i-1} \frac{dz_{\nu_i}}{z^n}\otimes \bar{f} d\bar{z}_{\nu_1}\cdots\widehat{d\bar{z}_{\nu_i}}\cdots d\bar{z}_{\nu_j},\]
whence
\[\Gamma(\frac{1}{z^n}fdz_{\nu_1}\cdots dz_{\nu_j})=\frac{df}{z^n}\otimes d\bar{z}_{\nu_1}\cdots d\bar{z}_{\nu_j}  
+\sum_{i=1}^{d-1} (-1)^{j-i}\frac{dz_{\nu_i}}{z^n}\otimes (\bar{f} d\bar{z}_{\nu_1}\cdots\widehat{d\bar{z}_{\nu_i}}\cdots d\bar{z}_{\nu_j}\cdot d).\]
Comparing the above formulas with the  formulas from \ref{prop:Omega-charIII} 
we obtain \ref{prop:Omega-charII}. 
\end{proof}

\begin{rmk}
The above proof does not work for $j=0$, since the equivalence
\eqref{prop:Omega-char9.3} does not hold in this case.
Indeed,  in positive characteristic the formula for $\tF(X, nD)$ is not 
equal to \eqref{prop:Omega-char1} for $j=0$, see Proposition \ref{prop:global-fil}.
\end{rmk}

If $(X,nD)$ has a projective SNC-compactification, then Theorem \ref{thm:ASII} implies
$\widetilde{\Omega^j}(X,nD)=(\Omega^j)^{\rm AS}(X, nD)$ and 
Theorem \ref{prop:Omega-char}\ref{prop:Omega-charI}  gives thus an explicit description of
$\widetilde{\Omega^j}(X,nD)$. In particular, in characteristic zero this reproves \cite[Theorem 6.4]{RS}.
However, we can also use Theorem \ref{prop:Omega-char} and Theorem \ref{thm;AS} to get
an unconditional statement in positive characteristic as well.  This is done in the next corollary.
The resulting description of  $\widetilde{\Omega^j}(X,nD)$ in positive characteristic is new.

\begin{cor}\label{cor:tOmega}
For all $j\ge 0$ the reciprocity sheaf $\Omega^j$ has level $j+1$ (in the sense of \cite[Definition 1.3]{RS-ZNP}, which in the language 
of \cite{RS} is equivalent to the motivic conductor of $\Omega^{j}$ having level $j+1$). Furthermore,
\eq{cor:tOmega0}{\widetilde{\Omega^j}(X,nD)= (\Omega^j)^{{\rm AS}}(X,nD), \quad n\ge 0.}
\end{cor}
\begin{proof}
We consider \eqref{cor:tOmega0}.
For $n=0$ there is nothing to show, and for $j=0$ this holds by Example \ref{ex:ASII}\ref{ex:ASII1}. 
We assume $j, n\ge 1$ and set $\Omega^j:=F$.
Let $L$ be a henselian discrete valuation field of geometric type over $k$ with ring of integers $\sO_L$ and maximal ideal $\fm_L$.
We define 
\[\fil_{n}F(L):=
\begin{cases}
\fm_L^{-n+1}\Omega^j_{\sO_L}(\log) & \text{if } p=0 \text{ or } p>0 \text{ and } p\not|\, n\\
\fm_L^{-n}\Omega^j_{\sO_L}    &\text{if } p>0 \text{ and  }p|n.
\end{cases}\]
and 
\[c_L:F(L)\to \N_0, \quad c_L(\alpha):=\min\{n\ge 0\mid \alpha\in \fil_nF(L)\}.\]
Let $(Y,E)$ be a modulus pair, i.e., $Y$ is separated and of finite type over $k$, $E$ is an effective Cartier divisor on $Y$ and $V=Y\setminus |E|$ is smooth
and set 
\[F_c(Y,E):=\bigcap_{\rho\in V(L)\cap Y(\sO_L) }\{\alpha\in F(V)\mid c_L(\rho^*\alpha)\le v_L(\rho^*E)  \},\]
where the intersection runs over all maps $\rho:\Spec L\to V$, with $L$ as above, which extend to $\Spec \sO_L\to Y$ and 
$v_L(\rho^*E)$ denotes the multiplicity of $\fm_L$ in $\rho^*E$  on $\Spec \sO_L$.
We have
\eq{cor:tOmega1}{\tF(X,nD)\subset \FAS(X,nD)\subset F_c(X,nD)}
where the first inclusion holds by Theorem \ref{thm;AS} and the second follows from
Theorem \ref{prop:Omega-char}\ref{prop:Omega-charI}.
Furthermore, it follows by the same argument as in the proof of (c6) in the proof of \cite[Theorem 6.4]{RS}, that 
\eq{cor:tOmega2}{F_c(X,nD)=\colim_{(Y, E)} F_c(Y,E),}
where the colimit runs over all compactifications $(Y,E)$ of $(X,nD)$, in the sense  of \cite[Definition 1.8.1]{KMSY1}.
Thus it suffices to show:
\begin{claim}\label{cor:tOmega-claim}
 $c=\{c_L\}_L$ defines a  semicontinuous conductor of level $j+1$ in the sense of \cite[Definition 4.3]{RS}.
 \end{claim}
Indeed, assuming the claim,  \cite[Theorem 4.15(4)]{RS} together with \eqref{cor:tOmega2} imply 
\[F_c(X,nD)\subset \tF(X,nD),\]
which together with \eqref{cor:tOmega1} implies \eqref{cor:tOmega0} which also implies that $c$ is the motivic conductor of $F$.

We prove Claim \ref{cor:tOmega-claim}. To this end, we have to show that $c$ satisfies the properties (c1)-(c6) from \cite[Definition 4.3 and 4.22]{RS},
where the statement about the level is the part (c4). If $p=0$, then it is checked in the proof of \cite[Thm 6.4]{RS}, that 
$c$ satisfies these properties. If $p>0$ the same proof works for (c1), (c2), (c4) - (c6)
Thus it remains to show  (c3) which is the following condition:
Let $L'/L$ be a finite extension of henselian discrete valuation fields over $k$ of ramification index $e$.
Then we have to show for $n\ge 0$   
\eq{cor:tOmega3}{\Tr(\fil_n F(L'))\subset \fil_{r}F(L)\quad \text{with } r=\lceil n/e\rceil,}
where $\Tr: F(L')\to F(L)$ denotes the trace. This is clear for $n=0$, hence we assume $n\ge 1$ (and thus $r\ge 1$).
Let $K$ be the residue field of $\sO_L$.
For $m\ge 2$ consider the  map induced by the characteristic form (see \ref{para:kosz})
\[\theta_m:=(\partial_m^{j+1}\circ d, \partial_m^j): 
\fil_m F(L)\to \fm_L^{-m}\Omega^1_{\sO_L}\otimes_K (\Omega^{j}_K\oplus\Omega^{j-1}_K) \]
By Theorem \ref{prop:Omega-char}  we have 
\eq{cor:tOmega3.5}{\Ker (\theta_m)= \fil_{m-1} F(L).}
(Note that this is also true in the case $(n,p)=(2,2)$.)
Let $t$ be a local parameter of $\sO_L$ and $\tau$ a local parameter of $\sO_{L'}$.
We have $t=\tau^e u$, for some $u\in \sO_{L'}^\times$. 
If $p\not|\, n$, then  any element in $\fil_n F(L')$ is modulo $F(\sO_{L'})$ a sum of elements of the form
\eq{cor:tOmega4}{\alpha \frac{d\tau}{\tau^i}, \quad\text{with }\alpha\in \Omega^{j-1}_{\sO_{L'}},\,  i=1, \ldots, n, \quad \text{and}\quad 
\frac{\beta}{\tau^i},\quad \text{with }\beta\in \Omega^j_{\sO_{L'}}, \, i=1,\ldots, n-1;}
if $p|n$, then we have to add $\beta/\tau^n$ to this list.
Let $m>r$, equivalently $(m-1)e-n\ge 0$. 
For $\omega$ equal to one of the forms in \eqref{cor:tOmega4}, it is direct to check that we have 
\[t^{m-1}\Tr(\omega)\in \Omega^j_{\sO_L} \quad \text{and}\quad t^{m-1} d(\Tr(\omega))\in \Omega^{j+1}_{\sO_L},\]
where for the second statement we use $d\circ\Tr=\Tr\circ d$; if $p|n$ the same is true for $\omega=\beta/\tau^n$. 
Hence
\[\Tr(\fil_nF(L'))\subset \fm_L^{-(m-1)}\Omega^j_{\sO_L}\quad \text{and}\quad d\Tr(\fil_nF(L'))\subset \fm_L^{-(m-1)}\Omega^{j+1}_{\sO_L}.\]
Thus
\[m>r=\lceil n/e\rceil\quad \Longrightarrow \quad \Tr(\fil_nF(L'))\subset \Ker \theta_m.\]
By \eqref{cor:tOmega3.5} this implies \eqref{cor:tOmega3} and completes the proof.
\end{proof}

\begin{para}\label{para:Q}
In \cite[Questions 1 and 2]{KMSY1} the authors ask the following question:
Let $R$ be an effective divisor with simple normal crossing support and  with irreducible components $R_{\red}=\cup_{i=1}^s D_i$ 
and let $Z=D_1\cap\ldots\cap D_r$, with $2\le r\le s$. Let $f: Y\to X$ be the blow-up of $X$ in $Z$.
Is it true that the pullback
\eq{para:Q0}{f^*: H^i(X_{\Nis}, \tF_{(X,R)})\lra H^i(Y_{\Nis}, \tF_{(Y, f^*R)}) \qquad (F\in \RSC_{\Nis})}
is an isomorphism for any $i\ge 0$.
Theorem \ref{prop:Omega-char} and Corollary \ref{cor:tOmega} imply that this question has a negative answer 
if ${\rm char}(k)=p>0$, $R=pR'$, for some effective divisor $R'$ with SNC support,   and  $F=\Omega^j$   with $j<d$. 
Indeed, in this case it follows from Theorem \ref{prop:Omega-char} and Corollary \ref{cor:tOmega} that we have 
\eq{para:Q1}{\widetilde{\Omega^j}_{(X, pR')}=\Omega^j_X(pR') \quad \text{and} \quad \widetilde{\Omega^j}_{(Y, p f^*R')}=\Omega^j_Y(p f^*R'),}
where $\Omega^j_X(pR')=\Omega^j_X\otimes_{\sO_X}\sO_X(pR')$ (See the proof of Lemma \ref{lem:tomega} below, where it is shown that a similar statement also holds for the top differentials (i.e. $j=\dim X$) in a more general situation).
The blow-up formula for differential forms 
yields an isomorphism
\[Rf_*\Omega^j_Y\cong  \Omega^j_X\oplus \bigoplus_{m=1}^{r-1} i_*\Omega^{j-m}_Z[-m],\]
where $i:Z\inj  X$ denotes the closed immersion. Twisting with $\sO_X(pR')$, using the projection formula, \eqref{para:Q1} yields
\[H^i(Y_{\Nis}, \widetilde{\Omega^j}_{(Y, p f^*R')})\cong 
H^i(X_{\Nis}, \widetilde{\Omega^j}_{(X, pR')})\oplus\bigoplus_{m=1}^{r-1} H^{i-m}(Z_{\Nis}, \Omega^{j-m}_{Z}\otimes_{\sO_Z} \sO_X(pR')_{|Z}).\]
Since the cohomology groups on $Z$ on the right hand side do not vanish in general (e.g. not if $i=m=j$), we see that \eqref{para:Q0} 
is not an isomorphism in this case.

Note however, that for $j=\dim X$, we have $\Omega^{j-m}_Z=0$, for $m=1,\ldots, r-1$, and thus
the $Z$-cohomology part on the right hand side vanishes and the statement holds. 
In fact we see in Corollary \ref{cor:Ext-MNST} below,  that for top differentials 
a much stronger version than \eqref{para:Q0} is true.
\end{para}

\section{Applications to top differentials}\label{TopOmega}
Throughout this section, we assume $X\in \Sm$ with $\dim X=d$ and ${\rm char}(k)=p\ge 0$.
We write $\Omega^j_X=\Omega^j_{X/k}$.

\begin{lem}\label{lem:tomega}
Let $X$ be as at the beginning of this section and let $R$ be any effective Cartier divisor on $X$. 
Then we have the following isomorphism of  Nisnevich sheaves on $X$
\[\widetilde{\Omega^d}_{(X,R)}\cong \Omega^d_X(R),\]
where $\Omega^d_X(R):=\Omega^d_{X/k}\otimes_{\sO_X} \sO_X(R)$.
\end{lem}
\begin{proof}
Let $j:U=X\setminus R\inj X$ be the open immersion. 
We observe that $\widetilde{\Omega^d}_{(X,R)}$ and $\Omega^d_X(R)$ are subsheaves of $j_*\Omega^d_U$.
Furthermore, since $X$ is smooth, the $\sO_X$-module $\Omega^d_X(R)$ is locally free.
Hence  
\[\Gamma(X, \Omega^d_X(R))=\Gamma(U,\Omega^d)\cap \bigcap_{\eta\in |R|^{(0)}} (\Omega^d_X(R))_\eta.\]
Thus Corollary \ref{cor:tOmega} and Theorem \ref{prop:Omega-char}\ref{prop:Omega-charI} directly imply 
\eq{lem:tomega1}{\widetilde{\Omega^d}_{(X,R)}\subset \Omega^d_X(R).}
To prove the other inclusion, set $Z\Omega^j=\Ker(d:\Omega^{j}\to \Omega^{j+1})$.
\begin{claim}\label{claim:lem-tomega}
Let $L/k$ be a henselian dvf of geometric type. Then for all $n\ge 1$
\[\fm^{-n}_L\Omega^d_{\sO_L}\cap Z\Omega^d_L\subset \widetilde{\Omega^d}(\sO_L, \fm_L^{-n}).\]
\end{claim}
Indeed, by Corollary \ref{cor:tOmega} and Theorem \ref{prop:Omega-char}\ref{prop:Omega-charI} the statement holds if $p|n$.
Assume $p=0$ or $p>0$ and $(p,n)=1$. Let $K\inj \sO_L$ be a coefficient field.  
Letting $t$ be a prime element of $\sO_L$, we can write any $\Phi\in \fm_L^{-n}\Omega^d_{\sO_L}$ as 
\[\Phi=\sum_{i=0}^{n} \left(\alpha_i \frac{dt}{t^i}+ \beta_i\frac{1}{t^i} \right) + \gamma, \]
with $\alpha_i \in \Omega^{d-1}_K$, $\beta_i\in \Omega^{d}_K$ and $\gamma\in \Omega^{d}_{\sO_L}$.
Under our assumptions on $n$ we see that $d\Phi$ can only vanish if $\beta_n=0$ for reasons of pole orders, which 
by Corollary \ref{cor:tOmega} and Theorem \ref{prop:Omega-char}(1) implies $\Phi\in \widetilde{\Omega^d}(\sO_L, \fm_L^{-n})$.
This  proves the claim.

We now prove the other inclusion  of \eqref{lem:tomega1}.  This is a local question, we can therefore assume
$X=\Spec A$ affine and $D=\div (f)$, for some non-unit $f\in A$.
Let  $(\ol{X}, \ol{R}+\Sigma)$ be a compactification of $(X,R)$, i.e., $\ol{X}$ is a proper $k$-scheme, $\ol{R}$ and $\Sigma$ are effective Cartier divisors 
on $\ol{X}$, $X=\ol{X}\setminus \Sigma$, and $\ol{R}_{|X}=R$. 
 By   \cite[Theorem 4.15(4)]{RS} it remains to show the following:
 
Let $\alpha= \frac{1}{f}\cdot \alpha_0$ with $\alpha_0\in \Omega^d_{A/k}$.
Then, for all henselian dvf $L/k$ of geometric type and all morphisms $\rho: \Spec \sO_L\to \ol{X}$, which map the generic point into $U=X\setminus R$, 
there exists a natural number $N\ge 1$ such that
\[\rho^*\alpha\in \widetilde{\Omega^d}(\sO_L,\fm_L^{-v_L(\rho^* \ol{R}+N\Sigma)}).\]
Since $d=\dim X$ and $\Omega^j_A=\Omega^j_{A/k}$ by our convention from the beginning of this section,
we have $\alpha\in Z\Omega^d_A$ and hence by Claim \ref{claim:lem-tomega}
it suffices to show
\eq{lem:tomega2}{\rho^*\alpha\in \fm_L^{-v_L(\rho^* \ol{R}+N\Sigma)}\Omega^d_{\sO_L}, \quad \text{for }N\gg 0.}

We prove  \eqref{lem:tomega2}. Let $\ol{X}=\cup_i V_i$ be a  finite open cover with $V_i=\Spec B_i$, 
$\Sigma_{|V_i}=\div(\sigma_i)$ and $\ol{R}_{|V_i}=\div(f_i)$,  where $\sigma_i, f_i\in B_i$ are non-zero divisors.
Since $\Spec B_i[1/\sigma_i]$ is open in  $X$, we can write 
$f=e_i f_i$ with $e_i\in (B_i[1/\sigma_i])^\times$ so we find $N_1\ge 1$ such that $\sigma_i^{N_1}e_i^{-1}\in B_i$, for all $i$. 
Furthermore, we find $N_2\ge 1$ such that $\sigma_i^{N_2}\cdot \alpha_0\in \Omega^d_{B_i}$.
Set $N:=N_1+N_2$. Let $\rho\in U(L)\cap V_i(\sO_L)$. 
Since $e_i f_i\sigma_i^{N_2}\cdot \alpha_{|V_i}\in \Omega^d_{B_i}$ we get 
\[\rho^*(e_if_i\sigma_i^{N_2})\cdot \rho^*\alpha\in \Omega^d_{\sO_L}.\]
By the choice of $N$ we have 
\[v_L(\rho^*(e_if_i\sigma_i^{N_2}))\le v_L(\rho^*(\ol{R}+N\cdot \Sigma) ).\]
This yields \eqref{lem:tomega2} and  completes the proof.
\end{proof}

\begin{para}\label{para:pr}
Let $Y$ be an integral, normal, Cohen-Macaulay scheme of finite type over $k$ and dimension $d$.
Recall from \cite{LT} (see also \cite[Definition 9.4]{Kovacs}) that $Y$ has
{\em pseudo-rational singularities} if it satisfies the following property:
for any normal scheme $Z$ with a proper and birational morphism  $f:Z\to Y$, the  trace along $f$ induces an isomorphism
\[\Tr_f: f_*\omega_{Z/k}\xr{\simeq} \omega_{Y/k},\] 
where $\omega_{Y/k}=H^{-d}(\pi_Y^!k)$ with $\pi_Y:Y\to \Spec k$ the structure morphism.
We remark that if $Y$ is normal  of dimension $d$ (but not necessarily is CM nor has pseudo-rational singularities) and
$j:U\inj Y$ is a smooth open subset whose complement has codimension $\ge 2$, then
the restriction along $j$ induces an isomorphism (e.g. \cite[Lemma 3.22]{Kovacs})
\eq{eq1;para:pr}{\omega_{Y/k}\simeq  j_*\omega_{U/k}\cong j_* \Omega^d_{U}.}  
By \cite[Corollary 9.13(ii)]{Kovacs}, pseudo-rational singularities are the same as rational singularities.
In particular, toric singularities and tame quotient singularities are pseudo-rational. 
See  \cite{Kovacs} for more examples and the state of the art.  
\end{para}

Recall that a sheaf of abelian groups $H$ on $Y$ is called $S2$ if for any open dense immersion $j:U\inj Y$, 
the restriction $H\to j_*H_{|U}$ is injective, and it is bijective if $Y\setminus U$ has  codimension $\ge 2$.
The following corollary is for $p=0$  equivalent to \cite[Theorem 7.1]{RS-ZNP}, for $p>0$ it gives
a new characterization of pseudo-rational singularities.
\begin{cor}\label{cor:pr}
Let $Y$ be as in the beginning of \ref{para:pr}.
Assume, for each  effective Cartier divisor $R$ on $Y$ such that $Y\setminus R$ is smooth,  
the sheaf  $\widetilde{\Omega^d}_{(Y,R)}$  is $S2$. 
Then $Y$  has pseudo-rational singularities. 
If there exits a proper and birational morphism $f:Z\to Y$ with $Z\in \Sm$,
then the inverse implication holds. 
In both cases  \eqref{eq1;para:pr} induces
\eq{cor:pr2}{\widetilde{\Omega^d}_{(Y, R)}\simeq \omega_{Y/k}(R).}   
\end{cor}
\begin{proof}
Let $f:Z\to Y$ be as in \ref{para:pr}. 
Let $j:U\inj Y$ (resp. $j':V\inj Z$) be  an embedding of a smooth open  with complement of codimension $\ge 2$.
Let $R$ be an effective Cartier divisor of $Y$ such that $Y\setminus R$ and $Z\setminus f^*R$ are smooth.
We denote by $R_U$ the restriction of $R$ to $U$.
We have 
\eq{eq2;para:pr}{j_*(\widetilde{\Omega^d}_{(Y,R)})_{|U}=j_*\widetilde{\Omega^d}_{(U,R_U)}= j_*\Omega^d_U(R_U)= (j_*\Omega^d_U)\otimes \sO_{Y}(R)= \omega_{Y/k}(R),}
where the second (resp. third, resp. last) equality follows from Lemma \ref{lem:tomega} (resp. the projection formula, resp. \eqref{eq1;para:pr}). 
Similarly we find
\eq{eq3;para:pr}{j'_*(\widetilde{\Omega^d}_{(Z,f^*R)})_{|V}= \omega_{Z/k}(f^*R).}
We obtain a commutative diagram 
\eq{cor:pr1}{\xymatrix{
f_*(\widetilde{\Omega^d}_{(Z, f^*R)})\ar[r] & f_*(\omega_{Z/k})\otimes\sO_Y(R)\ar[d]^{\Tr_f\otimes\id}\\
\widetilde{\Omega^d}_{(Y, R)}\ar[u]^{\simeq}\ar[r] & \omega_{Y/k}(R),
}}
where the lower horizontal map is 
\[\widetilde{\Omega^d}_{(Y, R)}\to j_*j^*\widetilde{\Omega^d}_{(Y, R)}
\overset{\eqref{eq2;para:pr}}{\simeq} \omega_{Y/k}(R),\]
and the upper horizontal map is obtained by applying $f_*$ to 
\[\widetilde{\Omega^d}_{(Z, f^*R)}\to j'_*{j'}^* \widetilde{\Omega^d}_{(Z, f^*R)}
\overset{\eqref{eq3;para:pr}}{\simeq}  \omega_{Z/k}(f^*R).\] 
The left vertical map is an isomorphism since for any
\'etale map $Y'\to Y$ the  pairs $(Y',R_{|Y'})$ and $(Z'=Z\times_Y Y', R_{|Z'})$ are isomorphic in the category of modulus pairs, see \cite{KMSY1}.
The isomorphism \eqref{eq1;para:pr} implies that
the map $ \omega_{Y/k}(R)\to h_*h^* \omega_{Y/k}(R)$ is injective for any dense smooth open immersion $h: W\inj Y$, and similarly for
$ \omega_{Z/k}(f^*R)$ instead of $ \omega_{Y/k}(R)$. 
Therefore, the diagram commutes since on an open subset of $Y\setminus R$, over which  $f$ becomes an isomorphism, the vertical map on the left is the pullback
(which is an isomorphism) and the  vertical  map on the right is its inverse.
The last argument also shows that all the maps in the diagram are injective. 
If $\widetilde{\Omega^d}_{(Y, R)}$ is $S2$, then the lower horizontal map in \eqref{cor:pr1} is an isomorphism and hence so is the right vertical map, 
which proves that $Y$ has pseudo-rational singularities. 
On the other hand, if $Z$ is smooth, 
Lemma \ref{lem:tomega} implies that the top horizontal map is an isomorphism.
If $Y$ furthermore has pseudo-rational singularities so that the right vertical map is an isomorphism, we get the isomorphism \eqref{cor:pr2},
which implies that $\widetilde{\Omega^d}_{(Y, R)}$ is $S2$. This proves the second statement.
\end{proof}

\begin{rmk}
We remark that Corollary \ref{cor:pr} together with Lemma \ref{lem:tomega} implies that any smooth $Y$ has pseudo-rational singularities.
This is well-known and holds by \cite[4.]{LT} more generally for regular schemes.
\end{rmk}

Combining the above with some of the main results from \cite{Kovacs} and \cite{KMSY1}
we obtain:

\begin{cor}\label{cor:Ext-MNST}
Let $Y$ be as in the beginning of \ref{para:pr}.
Assume:
\begin{enumerate}[label=(\arabic*)]
\item\label{cor:Ext-MNST1} 
For any proper and birational morphism $Z\to Y$, there exits a proper and birational morphism 
$Z'\to Z$ with $Z'$ smooth (e.g. $d\le 3$ and $Y$ quasi-projective by \cite{Cossart-PiltantI}, \cite{Cossart-PiltantII}).
\item\label{cor:Ext-MNST2}
 $Y$ has pseudo-rational singularities.  
\end{enumerate}
Then, for any effective Cartier divisor  $R$ on $Y$ such that $Y\setminus R$ is smooth, we have
\eq{cor:Ext-MNST3}{H^i(Y_\Zar, \widetilde{\Omega^d}_{(Y,R)})=\Ext^i_{\uMNST}(\Ztr(Y,R), \widetilde{\Omega^d}),}
where $\uMNST$ denotes the category of modulus Nisnevich sheaves with transfers introduced in \cite{KMSY1}
and $\Ztr(Y,R)\in \uMNST$ is the object represented by $(Y,R)$.
\end{cor}
\begin{proof}
By \cite[Theorem 2(2)]{KMSY1} we have 
\eq{cor:Ext-MNST4}{
\Ext^i_{\uMNST}(\Ztr(Y,R), \widetilde{\Omega^d})=\varinjlim_f H^i(Z_\Nis, \widetilde{\Omega^d}_{(Z, f^*R)}),
}
where the limit is over all proper morphisms $f: Z\to Y$ inducing an isomorphism 
$Z\setminus f^*R\cong Y\setminus R$. By assumption \ref{cor:Ext-MNST1}, we may
assume additionally that $Z$ has pseudo-rational singularities (we could even assume that $Z$ is smooth, 
but we  want $Y$ to be in the limit as well).
Thus the  isomorphism \eqref{cor:pr2} holds replacing $(Y,R)$ by any pair $(Z, f^*R)$ appearing in the direct limit.
Therefore we can replace the Nisnevich cohomology with the Zariski cohomology
and the direct limit in \eqref{cor:Ext-MNST4} is constant and equal to 
$H^i(Y_\Zar, \omega_{Y/k})$ since $Rf_*\omega_{Z/k}=\omega_{Y/k}$ by Kovacs' vanishing \cite[Theorem 1.11]{Kovacs}.
\end{proof}

\begin{remark}\label{rmk:MDM}
Assume the  following condition holds over $k$:

\begin{enumerate}[label=(1')]
\item\label{rmk:MDM1} For any integral separated $k$-scheme  of finite type $S$ and effective Cartier divisor $E$ on $S$ with $S\setminus E\in \Sm$, there exists a 
proper and birational morphism $f: S'\to S$ such that $S'\in \Sm$, $(f^*E)_{\red}$ is SNCD, and $f$ induces an isomorphism $S'\setminus f^*E\cong S\setminus E$.
\end{enumerate}
(Note that since there is no restriction on the dimension of $S$, this is completely unknown in positive characteristic; but it  holds if ${\rm char}(k)=0$.)

Then, in the situation of Corollary \ref{cor:Ext-MNST} we have 
\eq{para:MDM2}{H^i(Y_\Zar, \omega_{Y/k}(R))=\Hom_{\uMDM}(\ul{M}(Y,R), \widetilde{\Omega^d}[i]),}
where $\uMDM$ is the triangulated category of cube invariant modulus Nisnevich sheaves with transfers constructed in \cite[Definition 3.2.4]{KMSY3} and 
$\ul{M}(Y,R)$ denotes the motive of the modulus pair $(Y, R)$ in $\uMDM$, see \cite[Theorem 3.3.1]{KMSY3}.

Indeed, by \eqref{cor:Ext-MNST3}, \eqref{cor:Ext-MNST4} and \cite[Theorem 5.2.4 and Theorem B.6.4, b)]{KMSY3}, it suffices
to show that $\widetilde{F}$ is $\bcube$-local in the sense of \cite[Definition B.2.4, a)]{KMSY3}, for $F=\Omega^d$.
The latter is equivalent to the pullback map
\[\Ext^i_{\uMNST}(\Ztr(S,E), \tF)\to \Ext^i_{\uMNST}(\Ztr((S,E)\otimes\bcube), \tF)\]
being an isomorphism for all $i\in \Z$ and all pairs $(S, E)$ as above.
By \cite[Theorem 2(2)]{KMSY1} and the assumption \ref{rmk:MDM1}, this follows from the fact
that the pullback 
\[H^i(S_{\Nis}, \tF_{(S,E)})\xr{\simeq} H^i((S\times \P^1)_{\Nis}, \tF_{(S, E)\otimes \bcube})\]
is an isomorphism for all $(S,E)$ with $S\in \Sm$ and $E_\red$ SNCD, which holds by \cite[Theorem 9.3]{S-purity}. 
\end{remark}

%


\providecommand{\bysame}{\leavevmode\hbox to3em{\hrulefill}\thinspace}
\providecommand{\MR}{\relax\ifhmode\unskip\space\fi MR }
\providecommand{\MRhref}[2]{%
  \href{http://www.ams.org/mathscinet-getitem?mr=#1}{#2}
}
\providecommand{\href}[2]{#2}

\end{document}